\documentclass{siamonline190516}
\usepackage{a4,latexsym,exscale,theorem,epsfig}
\usepackage{amssymb,psfrag,epsf,amsmath,verbatim,bbm,float}
\usepackage{mathtools}
\usepackage{listings}
\usepackage{color}
\usepackage{mathbbol}
\usepackage{enumerate}
\newtheorem{remark}[theorem]{Remark}
\newtheorem{example}[theorem]{Example}
\usepackage{geometry}
\usepackage{accents}
\usepackage{algpseudocode}
\usepackage{todonotes}
\usepackage{caption}
\begin{document}
\newcommand {\eps} {\varepsilon}
\newcommand {\Z} {\mathbbm{Z}}
\newcommand {\R} {\mathbbm{R}}
\newcommand {\T} {\mathbbm{T}}
\newcommand {\Q} {\mathbbm{Q}}
\newcommand {\N} {\mathbbm{N}}
\newcommand {\C} {\mathbbm{C}}
\newcommand {\I} {\mathbbm{I}}
\newcommand {\dist} {{\rm{dist}}}
\newcommand {\cl}{\mathrm{cl}}
\newcommand {\PP} {\mathbbm{P}}
\newcommand {\ang} {\measuredangle}
\newcommand {\e} {{\rm{e}}}
\newcommand {\rank} {{\rm{rank}}}
\newcommand {\Span} {{\mathrm{span}}}
\newcommand {\card} {{\rm{card}}}
\newcommand {\ED} {\mathrm{ED}}
\newcommand {\cA} {\mathcal{A}}
\newcommand {\cO} {\mathcal{O}}
\newcommand {\cF} {\mathcal{F}}
\newcommand {\cC} {\mathcal{C}}
\newcommand {\cN} {\mathcal{N}}
\newcommand {\cV} {\mathcal{V}}
\newcommand {\cG} {\mathcal{G}}
\newcommand {\cB} {\mathcal{B}}
\newcommand {\cD} {\mathcal{D}}
\newcommand {\cP} {\mathcal{P}}
\newcommand {\cQ} {\mathcal{Q}}
\newcommand {\cW} {\mathcal{W}}
\newcommand {\cT} {\mathcal{T}}
\newcommand {\cI} {\mathcal{I}}
\newcommand {\Sn}[1] {\mathcal{S}^{#1}}
\newcommand {\range} {\mathcal{R}}
\newcommand {\kernel} {\mathcal{N}}
\newcommand{\one}{\mathbb{1}}
\renewcommand{\thefootnote}{\fnsymbol{footnote}}
\newcommand{\rle}{\rotatebox[origin=c]{-90}{$\le$}}
\newcommand{\rl}{\rotatebox[origin=c]{-90}{$<$}}
\newcommand{\rg}{\rotatebox[origin=c]{-90}{$=$}}


\title{\bf Angular spectra of linear dynamical systems in discrete time}

\author{Wolf-J\"urgen Beyn\footnotemark[1]\qquad
  Thorsten H\"uls\footnotemark[1]
}
\footnotetext[1]{Department of Mathematics, Bielefeld University,  
33501 Bielefeld, Germany \\
\texttt{beyn@math.uni-bielefeld.de}, \texttt{huels@math.uni-bielefeld.de}}

\maketitle


 \begin{abstract}
   In this work we introduce the notion of an angular spectrum for a
   linear discrete time nonautonomous 
   dynamical system. The angular spectrum comprises all accumulation
   points of longtime averages formed by maximal 
   principal angles between successive subspaces generated by the
   dynamical system. The angular spectrum is bounded by angular values 
   which have previously been investigated by the authors. In this
   contribution we derive explicit formulas 
   for the angular spectrum of
   some autonomous and specific nonautonomous systems. Based on a
   reduction principle we  set up a numerical method for the general
   case; we investigate its convergence and apply the method to
   systems with a homoclinic orbit and a strange attractor. 
   Our main theoretical result is a theorem on the invariance of the
   angular spectrum under summable perturbations 
   of the given matrices (roughness theorem). It applies to systems
   with a so-called complete exponential dichotomy (CED), a concept 
   which we introduce in this paper and which imposes more stringent
   conditions than those 
   underlying the exponential dichotomy spectrum.
 \end{abstract}



\renewcommand*{\thefootnote}{\arabic{footnote}}

\section{Introduction}
\label{sec0}
In this paper we introduce angular spectra for linear dynamical
systems in discrete time and analyze their properties.
The main purpose of this notion is to measure the longtime average
rotation of all subspaces of a fixed dimension 
driven by the  dynamics of a nonautonomous linear system. We propose
the resulting angular spectrum as a novel 
quantitative feature  which specifies the degree of rotation caused by 
the dynamical system. On the one hand, it goes beyond the classical
notion of rotation numbers for the motion of vectors in
two-dimensional planes, and on the other hand, it  
complements well-known characteristics such as the dichotomy (or
Sacker-Sell) spectrum which
specifies the exponential divergence or convergence of trajectories.

The underlying difference equation is of the form
\begin{equation}\label{diffeq}
  u_{n+1} = A_n u_n,\quad A_n \in \R^{d,d},\quad n \in \N_0,
\end{equation}
where we assume the matrices $A_n$ to be  invertible, bounded, and to
have uniformly bounded inverses.  
By $\Phi$ we denote the solution operator of \eqref{diffeq},  defined
by $\Phi(n,m)=A_{n-1}\cdot\ldots\cdot A_m$ 
for $n>m$, $\Phi(n,n)=I$, and $\Phi(n,m)=A_n^{-1} \cdot \ldots \cdot
A_{m-1}^{-1}$ for $n<m$. 

In a series of papers \cite{BeFrHu20,BeHu22,BeHu23X} the authors (in
\cite{BeFrHu20} jointly with G. Froyland) introduced and analyzed so
called angular values 
of a given dimension $0<s \le d$. These values measure the maximal
longtime average of principal angles 
between successive $s$-dimensional subspaces generated by the
dynamical system \eqref{diffeq}. 
To be precise, for every $V$ in the Grassmann manifold $\cG(s,d)$ of
$s$-dimensional subspaces of $\R^d$, one forms 
the average
\begin{equation}\label{sec0:average}
  \alpha_n(V)=  \frac{1}{n} \sum_{j=1}^n\ang(\Phi(j-1,0)V,\Phi(j,0)V),
\end{equation}
where $\ang(\cdot,\cdot)$ denotes the maximal principal angle of
subspaces; see e.g.\ \cite[Ch.6.4]{GvL2013}. 
The maximal asymptotic value 
$ \theta_{s}^{\sup,\varlimsup} = \sup_{V \in
  \mathcal{G}(s,d)}\varlimsup_{n\to\infty} \alpha_n(V)$ is then called
the outer 
angular value of dimension $s$; see Definition \ref{defangularvalues}
and note 
 the variations of this notion which use $\inf$ and $\varliminf$
 instead of $\sup$ and $\varlimsup$. 
 We consider this value to measure the maximal rotational stress that
 an object of dimension $s$ experiences 
 under the evolution of \eqref{diffeq} when considered as a time map
 of a continuous flow; see \cite[Introduction]{BeFrHu20} 
 for a broader discussion of possible physical
 interpretations. Angular values are defined for all dimensions $s\le
 d$ 
 and agree even for $s=1$ only in special cases with classical
 rotation numbers \cite[Ch.11]{KH95}, 
 \cite[Ch.6.5]{A1998}; see \cite[Section 5.1]{BeFrHu20}, \cite[Section
 4.2]{BeHu23X} for a detailed comparison.

 The purpose of this article is to study not only the extreme values
 but all possible accumulation points of 
 the angular averages $\alpha_n(V)$ when $V$ varies over the
 Grassmannian. The set
 \begin{equation*}
   \Sigma_s:= \mathrm{cl} \{ \theta \in [0,\tfrac{\pi}{2}]: \exists V \in \cG(s,d): \varliminf_{n \to \infty} \alpha_n(V) \le \theta
   \le \varlimsup_{n \to \infty} \alpha_n(V)\}
   \end{equation*}
  is called the \emph{ outer angular 
 spectrum of dimension $s$} of the given system \eqref{diffeq}; see
 Definition \ref{def3:1}. 
 In our contribution we pursue two main goals:
 \begin{enumerate}
 \item[-] Derive explicit formulas for the outer angular spectrum in
   certain model cases, set up  
   a numerical method for computing finite time approximations for
   general systems, and investigate its convergence 
   as time goes to infinity.
 \item[-] Discuss the relation of the outer angular spectrum to outer
   angular values and analyze 
   the sensitivity of the outer angular spectrum to perturbations of
   the system matrices $A_n$ in 
   \eqref{diffeq}.
 \end{enumerate}
 In Section \ref{sec1.1} we show that the outer angular values are related to
 $\Sigma_s$ through (see Proposition \ref{prop2:relate})
 \begin{equation*} 
  \{\theta_s^{\inf,\varliminf},\theta_s^{\inf,\varlimsup},
\theta_s^{\sup,\varliminf}, \theta_s^{\sup,\varlimsup}\}\subseteq \Sigma_s \subseteq
[\theta_s^{\inf,\varliminf},\theta_s^{\sup,\varlimsup}]. 
 \end{equation*}
 Then we consider the dichotomy spectrum of \eqref{diffeq} and  
 apply the reduction theory from \cite{BeHu22}. According to this theory, it  suffices to replace the Grassmannian
 in the definition of $\Sigma_s$ by the considerably smaller set of trace spaces, i.e.\ by elements of $\cG(s,d)$ which
 have a basis composed of vectors from the spectral bundle of the dichotomy spectrum (Theorem \ref{thm2:reducespec}).

 In Section \ref{sec3} we consider perturbed systems of the form
 \begin{equation*}
   v_{n+1}=  (A_n+E_n) v_n, \quad n \in \N
 \end{equation*}
 and analyze which perturbations leave the outer angular spectrum invariant.
 A first result shows that this is true for kinematic transformations which become
 orthogonal at infinity (Proposition \ref{prop3:inv}). Our second and main
result Theorem \ref{thm3:maininv} 
 states that the outer angular spectrum stays 
 invariant if the given system has  a so-called complete exponential
 dichotomy (CED; see Definition \ref{def:CED}) and 
 if the perturbations $E_n$ are absolutely summable.  The CED is a
 rather strict concept which  requires all fundamental solutions to be decomposable
 into solutions which have an exact exponential rate in forward and backward time (Proposition \ref{prop3:specCED}).
 For an autonomous 
 system such a property holds if and only if all eigenvalues are
 semi-simple (Example  \ref{ex3:semisimple}).
 The proof of the main theorem involves several steps. First, it is shown how the CED relates to
 a generalized exponential dichotomy (GED); see Proposition \ref{lem3:CEDspec}. Then roughness theorems
 are derived for a GED (Theorem \ref{thm3:rough}) and for a CED (Theorem \ref{roughCED}), and finally  we prove that
a  summable perturbation of a system with a CED is kinematically similar to the unperturbed one (Theorem \ref{thm3:CED}). 
 Let us note that the invariance of the dichotomy spectrum has been
 proved under weaker conditions than a CED 
 in \cite{P2012, PR16}. However, we believe that the conditions of
 Theorem \ref{thm3:maininv} cannot be substantially 
 weakened; see Examples \ref{ex3:counter1}, \ref{ex6:Jordancounter}.

 In Section \ref{Sec_num} we propose and apply a numerical method to
 compute approximate angular spectra 
 via finite time approximations $\alpha_N(V)$ of the sums in
 \eqref{sec0:average}. The main step is to take advantage of the fact that
 it suffices to do the computations for trace spaces rather than for the full
 Grassmannian. For example, 
 if all bundles are one-dimensional it even
 suffices to consider finitely many subspaces. 
 This case occurs, for example, for the well-known Lorenz system
 (Section \ref{sec4:Lorenz}). 
 In Section \ref{sec4:FT} we prove the convergence of finite time
 angular spectra to their infinite counterpart 
 (Proposition \ref{prop4:Hausdorff}) under a uniform Cauchy
 condition. This seemingly strict  property holds 
 if the sequence $\ang(\Phi(j,0)V,\Phi(j-1,0)V)$ is uniformly almost
 periodic (Lemma \ref{estap})  which can 
 be verified for the linearization about a homoclinic orbit of
 Shilnikov type (Example \ref{ex3:henon3}). 

 In Section \ref{sec5} we conclude the paper  with a brief discussion
 of variants of the outer angular spectrum suggested by
 other types of angular values \cite{BeFrHu20}, such as inner or
 uniform outer angular values. It seems, however, 
 that the outer angular spectrum provides the best insight into the
 rotational dynamics, when compared to its variants.

Let us finally note that  we are able to explicitly compute the 
 outer angular spectrum for some specific cases; see Examples \ref{ex2:exauto},
 \ref{ex:normal2}, \ref{ex6:mixed}, \ref{ex3:AUapauto}, 
 \ref{ex6:Jordancounter}, \ref{E2}.  Our explicit  computations and the numerical results
 show that angular spectra may consist of isolated points 
 and/or of perfect intervals and can sometimes be computed by invoking
 Birkhoff's ergodic theorem. 
 Both types of spectrum may occur in the same example and depend
 sensitively on parameters; 
 see Figure \ref{explicit}. In particular, this suggests that angular
 spectra are generally at most 
 upper semicontinuous with respect to small bounded perturbations (similar to
 the dichotomy spectrum). 
 We have no proof of this behavior in general, but upper
 semicontinuity and the failure of lower semicontinuity 
 has been established for the critical two-dimensional
 normal form in Example \ref{ex:normal2}; cf.\ \cite[Section
 4.4]{BeHu23X}. 


\section{Basic definitions and properties}\label{sec1}
In this section we propose the new notion of an angular spectrum
that takes into account all possible rotations, occurring for a
nonautonomous dynamical system of the form 
\eqref{diffeq}. 
Further we study some examples and discuss
the relation between the angular spectrum and the angular values introduced
in \cite{BeFrHu20}, \cite{BeHu22}. Finally, we use the 
theory from \cite{BeHu22} to show that the computation of the
angular spectrum can be reduced to a small set of subspaces called
trace spaces.
To keep the article self-contained,
we also collect some relevant notions, definitions and results
from \cite{BeFrHu20,BeHu22,BeHu23X}. 

\subsection{Subspaces and principal angles} \label{sec1.0}
Let us begin with a useful characterization of the maximum principal
angle between two 
subspaces $V$ and $W$ of $\R^d$, both having the same dimension $s$.
The principal angles between $V$ and $W$ can be computed from the
singular values $ \sigma_1\ge \sigma_2 \ge \cdots \ge \sigma_s>0$ of
$V_B^\top W_B$, where the columns of $V_B,W_B\in \R^{d,s}$ 
form orthonormal bases of $V,W$, respectively; see
\cite[Ch.6.4.3]{GvL2013}, \cite[Prop.2.2]{BeHu23X}. 
The principle angles $\phi_j\in [0,\frac{\pi}{2}]$ are given by
$\sigma_j = \cos(\phi_j)$ 
and we denote its largest value by $\ang(V,W)=\phi_s=\arccos(\sigma_s)$. 
For the one-dimensional case we further use the notion 
\begin{align*}
  \ang(v,w) = \ang(\mathrm{span}(v),\mathrm{span}(w)), \quad v,w \in \R^d,
  v,w \neq 0.
\end{align*}
Note that this value is 
$\arccos(\frac{|v^{\top}w|}{\|v\|\|w\|})$ which ignores the sign of
$v^{\top}w$ and thus the orientation of the spanning
vectors $v$ and $w$.

An alternative characterization of
$\ang(V,W)$ is given in the following proposition; see
\cite[Prop.2.3]{BeFrHu20}. 

  \begin{proposition}\label{Lemma2}
  Let $V, W\subseteq \R^d$ be two $s$-dimensional subspaces. 
  Then the following relation holds
  \begin{equation*}\label{A1}
    \ang(V,W) =
    \max_{\substack{v\in V \\ v\neq 0} }\min_{\substack{w\in W \\ w \neq 0}} \ang(v,w)
  = \arccos\big(\min_{\substack{v\in V\\\|v\|=1}} \max_{\substack{
      w\in W\\\|w\|=1}} v^{\top} w\big).
  \end{equation*}
\end{proposition}

We denote the Grassmannian by 
\begin{equation*} \label{eq1.4}
  \mathcal{G}(s,d) = \{ V \subseteq \R^d \; \text{is a subspace of dimension}
  \; s \}
\end{equation*} 
equipped with the metric
\begin{equation}\label{metric}
d(V,W) = \sin(\ang(V,W)),\quad V,W\in \cG(s,d).
\end{equation}
Note that 
\[
\tfrac 1 \pi \ang(V,W) \le d(V,W) \le \ang(V,W)\quad \forall\; V,W \in \cG(s,d).
\]
In fact, the angle $\ang(V,W)$ itself defines a metric on $\cG(s,d)$; see
\cite[Proposition 2.3]{BeHu23X} for a proof. We refer to \cite{BZA24} for a recent
overview of the geometry and computational aspects of Grassmann manifolds.
For example, the metric space $(\cG(s,d),\ang(\cdot,\cdot))$ is connected; see
\cite[Section 2.1]{BZA24}.

Consider a simple example which will be useful later on.
\begin{example} \label{ex2:s=2,d=3}
 Assume that  $V,W \in \cG(2,3)$ satisfy $V \cap W = \mathrm{span}(z)$
 for some  $z \neq 0$. 
 Then choose $v,w\in \R^3$ such that $V=\mathrm{span}(v,z)$, 
 $v^{\top}z=0$ and $W= \mathrm{span}(w,z)$, $w^{\top}z=0$; 
 see Figure \ref{winkel2} for an illustration.

\begin{figure}[hbt]
\begin{center}
\includegraphics[width=0.65\textwidth]{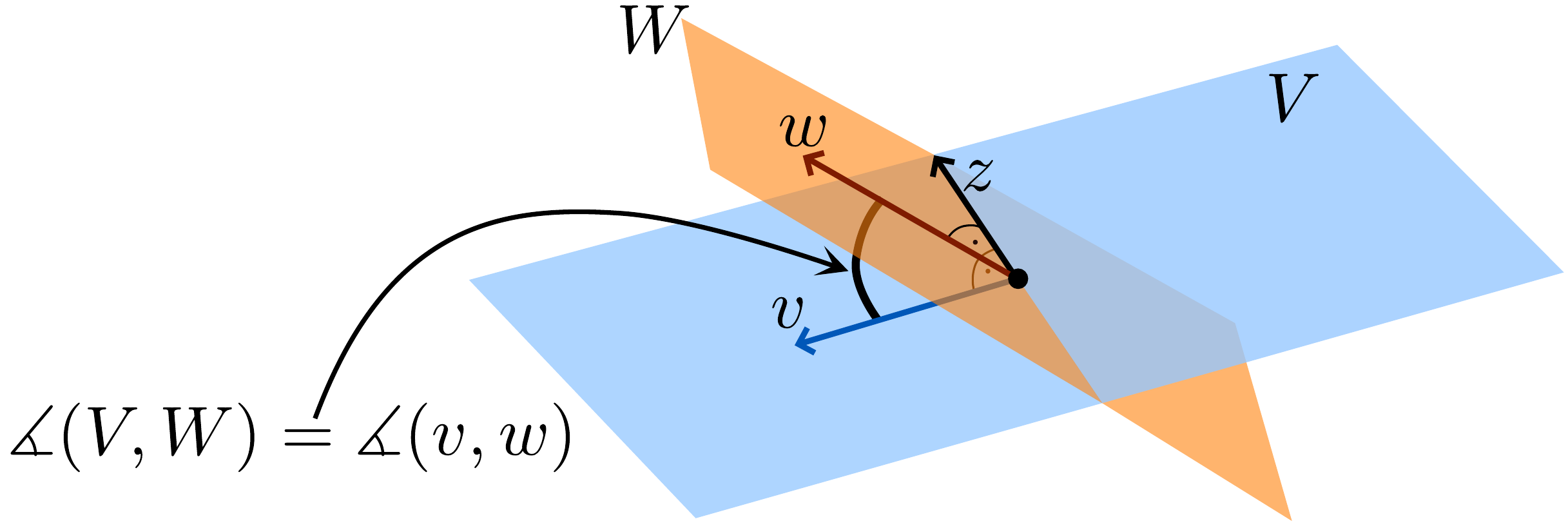}   
\end{center}
\caption{\label{winkel2} Computation of the angle between two planes  
  in $\R^3$.} 
\end{figure}

By normalizing we find the bases 
 \begin{align*}
   V_B&=\begin{pmatrix}\|v\|^{-1} v & \|z\|^{-1}z \end{pmatrix}, \quad
   W_B= \begin{pmatrix}\|w\|^{-1} w & \|z\|^{-1}z \end{pmatrix}
   \end{align*}
  and $V_B^{\top}W_B  = \mathrm{diag} (\frac{v^{\top}w}{\|v\| \|w\|}, 1)$.
 Hence, we obtain $\ang(V,W)=\arccos(\frac{|v^{\top}w|}{\|v\|\|w\|})= \ang(v,w)$. 
  \end{example}

Finally, we recall two estimates of principal angles. 
The first one from \cite[Lemma 2.8]{BeFrHu20} reads:
\begin{lemma}\label{app:Lest1}
Let $V,W\in\cG(s,d)$ and $S\in\mathrm{GL}(\R^d)$. Then we have
\begin{equation}\label{app:est1}
d(SV,SW) \le C d(V,W)\quad \text{with}\quad
C = \pi  \kappa(1+\kappa), \ \kappa = \|S^{-1}\|\|S\|.
\end{equation}
\end{lemma}
The second estimate compares the principal angle of subspaces when
only one space is mapped; see 
\cite[Lemma 3.3]{BeHu23X}.
\begin{lemma}\label{app:Lest2}
  For all $V,W \in \mathcal{G}(s,d)$ and $S \in \R^{d,d}$
  with $SV \in \mathcal{G}(s,d)$ one has
  \begin{equation} 
  \label{app:est2}
  |\ang(SV,W)- \ang(V,W)| \le C_{\pi} \|S-I_d\|, \quad
  C_{\pi}= \tfrac{\pi}{2}+\left(\tfrac{\pi^2}{4}+1\right)^{1/2}.
  \end{equation}
\end{lemma}

\subsection{The outer angular spectrum: definition and an elementary property}
\label{sec2:1}
For $n\in\N$ we abbreviate the average of principal angles between successive
subspaces by
\begin{equation} \label{def:alpha}
\alpha_n:
\begin{array}{rcl}
\cG(s,d) & \to & [0,\frac \pi 2]\\[1mm]
V & \mapsto &\displaystyle \frac 1n \sum_{j=1}^n\ang(\Phi(j-1,0)V,\Phi(j,0)V).
\end{array}
\end{equation}

\begin{definition} \label{def3:1}
  For $s\in\{1,\dots,d\}$ the \textbf{outer angular spectrum} of dimension
  $s$ is defined by  
$$
\Sigma_s \coloneqq \cl\big\{\theta \in [0,\tfrac \pi 2]: \exists V\in\cG(s,d): 
\varliminf_{n\to \infty} \alpha_n(V) \le \theta \le \varlimsup_{n\to \infty}
\alpha_n(V)\big\}.
$$
\end{definition}
When we consider different systems we will write $\Sigma_s(\Phi)$ to
indicate the dependence of the 
angular spectrum on the solution operator.
Since we take the closure the corresponding resolvent set
$[0,\frac{\pi}{2}]\setminus \Sigma_s$ is relatively open. 
Further, each value  $\theta\in \Sigma_s$ is an accumulation point of
$(\alpha_n(V))_{n\in\N}$ for some $V\in\cG(s,d)$.

\begin{lemma}\label{accu}
Let $V\in\cG(s,d)$. Then each 
$$
\theta\in\big[\varliminf_{n\to\infty}\alpha_n(V),\varlimsup_{n\to\infty}\alpha_n(V)\big]
$$
is an accumulation point of $(\alpha_n(V))_{n\in\N}$.
\end{lemma}

\begin{proof}
With $a_j = \ang(\Phi(j-1,0)V,\Phi(j,0)V)\in[0,\frac \pi 2]$ we
observe for $n\in\N$ that 
\begin{align*}
|\alpha_n(V)-\alpha_{n+1}(V)|
&=\left|\frac 1n \sum_{j=1}^n a_j - \frac 1{n+1}
                                \sum_{j=1}^{n+1}a_j\right|
= \left|\frac 1{n^2+n} \sum_{j=1}^n a_j - \frac 1{n+1}
                                a_{n+1}\right|\\
&\le \frac 1{n^2+n}\cdot n\cdot \frac \pi 2 + \frac 1 {n+1}\cdot \frac \pi 2 
= \frac \pi{n+1},
\end{align*}
hence $\lim_{n\to \infty} \alpha_n(V) - \alpha_{n+1}(V) = 0$.
Therefore, each value between $\varliminf_{n\to \infty}\alpha_n(V)$ and 
$\varlimsup_{n\to \infty}\alpha_n(V)$ is approached by a convergent subsequence of $\alpha_n(V)$.
\end{proof}

\subsection{Relation to angular values}
\label{sec1.1}
Several different types of angular values have been proposed in
\cite{BeFrHu20}. These values have in common that the supremum over
$ V \in \cG(s,d)$ is taken. For the relation to the  angular spectrum
defined above, it is useful to consider 
also the infimum w.r.t.\  $V \in \mathcal{G}(s,d)$. We
introduce these notions but restrict the presentation to the so called
outer angular values (see \cite[Section 3.1]{BeFrHu20} for inner and uniform
versions). 
\begin{definition} \label{defangularvalues}
  Let the  nonautonomous system \eqref{diffeq} and $s\in
  \{1,\ldots,d\}$ be given.
    The {\bf outer angular values} of dimension $s$ are defined by

\begin{equation*}\label{doutersup}
\begin{aligned}
  \theta_{s}^{\sup,\varlimsup} = \sup_{V \in \mathcal{G}(s,d)}
  \varlimsup_{n\to\infty} \alpha_n(V), \quad
   \theta_s^{\sup,\varliminf} =\sup_{V \in \mathcal{G}(s,d)}
  \varliminf_{n\to\infty}  \alpha_n(V),\\
  \theta_{s}^{\inf,\varlimsup} = \inf_{V \in \mathcal{G}(s,d)}
  \varlimsup_{n\to\infty}   \alpha_n(V), \quad
   \theta_s^{\inf,\varliminf} =\inf_{V \in \mathcal{G}(s,d)}
  \varliminf_{n\to\infty}  \alpha_n(V).
\end{aligned}
\end{equation*}
\end{definition}

The following relations hold for all $s=1,\ldots,d$
\begin{equation}\label{outercon}
\begin{matrix}
\theta_s^{\inf,\varliminf} & \le &  \theta_s^{\inf,\varlimsup}\\
\rle && \rle\\
\theta_s^{\sup,\varliminf} & \le &  \theta_s^{\sup,\varlimsup}.
\end{matrix}  
\end{equation}
For the $\sup_V$-values we refer to \cite[Lemm 3.3]{BeFrHu20} while
the other relations are rather obvious.
Note that these values are generally not identical; see \cite[Section
3.2]{BeFrHu20}. In particular, for the example 
in \cite[3.10]{BeFrHu20} the $\sup_{V\in \cG(1,2)}$
and $\inf_{V\in\cG(1,2)}$ coincide and the diagram reads 
\[
\begin{matrix}
\theta_1^{\inf,\varliminf} & < &  \theta_1^{\inf,\varlimsup}\\
\rg && \rg\\
\theta_1^{\sup,\varliminf} & < &  \theta_1^{\sup,\varlimsup}.
\end{matrix} 
\]

The following proposition is a consequence of the definitions.
\begin{proposition} \label{prop2:relate}
The outer angular spectrum of dimension $s$ satisfies
\begin{equation} \label{eq2:specbound}
  \{\theta_s^{\inf,\varliminf},\theta_s^{\inf,\varlimsup},
\theta_s^{\sup,\varliminf}, \theta_s^{\sup,\varlimsup}\}\subseteq
\Sigma_s \subseteq
[\theta_s^{\inf,\varliminf},\theta_s^{\sup,\varlimsup}]. 
\end{equation}
\end{proposition}
\begin{proof} For any $\varepsilon >0$ there exists $V \in \cG(s,d)$ such
  that $\theta_s^{\inf,\varliminf} \le \varliminf_{n\to\infty}   \alpha_n(V)
  \le \theta_s^{\inf,\varliminf}+ \varepsilon$. Since
  $\varliminf_{n\to\infty}   \alpha_n(V)\in \Sigma_s$ holds by
  definition and $\Sigma_s$ is closed we 
  obtain $\theta_s^{\inf,\varliminf}\in \Sigma_s$. A similar argument applies
  to the other angular values. The second inclusion in \eqref{eq2:specbound}
  is immediate from the definitions.
\end{proof}
Of course, in general there is more structure to the spectrum inside the
bounding interval as the following example shows.
\begin{example} \label{ex2:exauto}
Consider \eqref{diffeq} for the $3$-dimensional autonomous case 
\begin{equation}\label{mot3d}
A_n = A =
\begin{pmatrix} T_{\varphi} & 0 \\ 0 & 2 \end{pmatrix}, \quad
T_{\varphi}=\begin{pmatrix} \cos(\varphi) & -\sin(\varphi) \\
\sin(\varphi) & \cos(\varphi) \end{pmatrix}, \quad
n \in \N_0,\
0 < \varphi \le \frac \pi 2.
\end{equation}
For  $V=\mathrm{span}(v)$, $v=(z,v_3)^{\top}\neq 0$, $z \in \R^2$ we obtain
\begin{align*}
  \lim_{n \to \infty}\alpha_n(V)=\begin{cases}
  0, & \text{if} \quad v_3 \neq 0, \\
  \varphi, & \text{if} \quad v_3=0.
  \end{cases}
\end{align*}
This follows from
\begin{align*}
  A^jv&= \begin{pmatrix} T_{j\varphi}z \\ 0
  \end{pmatrix}, \quad \ang(A^{j-1}v,A^jv)=\varphi \quad \text{if} \quad v_3=0,\\
  \ang(A^{j-1}v,A^jv)&=\ang\left(
  \begin{pmatrix} T_{(j-1)\varphi}z \\  2^{j-1}v_3\end{pmatrix},
    \begin{pmatrix} T_{j\varphi}z \\  2^jv_3\end{pmatrix}\right)\\
     & =\ang \left(
  \begin{pmatrix}2^{1-j}v_3^{-1} T_{(j-1)\varphi}z \\  1\end{pmatrix},
    \begin{pmatrix}2^{-j}v_3^{-1} T_{j\varphi}z \\  1 \end{pmatrix}\right)\to 0
    \; \text{as} \; j \to \infty,\; \text{if}\; v_3 \neq 0.
\end{align*}
For a precise estimate in the last case see \cite[Section 3.1]{BeFrHu20}.
Therefore, the inclusion \eqref{eq2:specbound} for this example reads
\begin{align*}
  \{\theta_1^{\inf,\varliminf}, \theta_1^{\sup,\varlimsup}\}= \{0,\varphi\}= \Sigma_1
  \subsetneq [0,\varphi]=[\theta_1^{\inf,\varliminf},\theta_1^{\sup,\varlimsup}].
  \end{align*}
\end{example}
A more intriguing example is the following autonomous system
\eqref{diffeq} for the orthogonal normal form of a $2 \times 2$-matrix 
with complex conjugate eigenvalues;  see \cite[Sections 5.1,6.1]{BeFrHu20}.
\begin{example}\label{ex:normal2}
  \begin{align}\label{rhomatrix}
     A(\rho,\varphi) =\begin{pmatrix} \cos(\varphi) & - \rho^{-1} \sin(\varphi)\\
    \rho \sin(\varphi) & \cos(\varphi) \end{pmatrix}, \quad
     0<\rho \le 1, \quad 0 <  \varphi \le \frac \pi 2.
  \end{align}
  In \cite[Proposition 5.2]{BeFrHu20} we showed that all angular
  values agree in this case, and in 
  \cite[Theorem 6.1]{BeFrHu20} we determined an explicit formula for
  $\theta_1^{\sup,\lim}(A(\rho,\varphi))$. It 
  depends on the skewness parameter
  $\mathrm{sk}(\rho,\varphi)=\frac{1}{2}(\rho+
  \rho^{-1})\sin(\varphi)$ 
  and on  $\frac{\varphi}{\pi}$ being rational or irrational. The
  proof of \cite[Theorem 6.1]{BeFrHu20} 
  does not only provide the maximal angular value
  $\theta_1^{\sup,\lim}(A(\rho,\varphi))$ but also 
  the minimal one $\theta_1^{\inf,\lim}(A(\rho,\varphi))$ and thus the
  angular spectrum 
  $\Sigma_1(A(\rho,\varphi))$. Therefore,
  we don't repeat the computation here. The result is the following:
    \begin{equation} \label{finalform}
    \Sigma_1(A(\rho,\varphi)) =
    \begin{cases} \begin{array}{ll}
     \{ \varphi\}, &
      \mathrm{sk}(\rho,\varphi)\le 1,
         \\ \displaystyle
     \Big\{ \varphi+ \frac{1}{\pi}
                    \int_{\{\delta_{\rho,\varphi}<0\}}\delta_{\rho,\varphi}(\theta) 
      \mathrm{d}\theta \Big\}, &
        \mathrm{sk}(\rho,\varphi)>1, \frac{\varphi}{\pi} \notin \Q, \\
        \left[\min_{0 \le \theta \le \frac{\pi}{2}}G_q(\theta),\max_{0
                    \le \theta \le \frac{\pi}{2}}G_q(\theta)\right], 
      &  \mathrm{sk}(\rho,\varphi)>1, \frac{\varphi}{\pi}=\frac{p}{q},
        q \perp p. 
    \end{array} 
    \end{cases}
  \end{equation}
  Here the functions $\delta_{\rho,\varphi},G_q:[0,\pi] \to \R$ and
  $\Psi_{\rho}: \R \to  \R$ are defined as follows: 
  \begin{equation} \label{defdeltag}
  \begin{aligned}
       \Psi_{\rho}(\theta) &=
        \begin{cases} \arctan(\rho \tan(\theta)), &
            |\theta| <  \frac{\pi}{2}, \\ \theta, & |\theta|= \frac{\pi}{2},
        \end{cases}\\
         \Psi_{\rho}(\theta + n\pi) & = \Psi_{\rho}(\theta)+ n \pi, \text{ for }
    |\theta| \le \frac{\pi}{2},n  \in \Z \setminus \{0\}, \\
     \end{aligned}
  \end{equation}
  \begin{equation*} 
  \begin{aligned}
    \delta_{\rho,\varphi}(\theta)&= 2 \Psi_{\rho}(\theta)-2
    \Psi_{\rho}(\theta+ \varphi) 
    + \pi,   \\
       G_q(\theta)& = \frac{1}{q}\sum_{j=1}^q  \min( \theta_{j}- \theta_{j-1},
     \theta_{j-1} +\pi-\theta_{j} ), \quad
     \theta_{j-1} = \Psi_{\rho}((j-1) \varphi + \Psi_{\rho^{-1}}(\theta)).
  \end{aligned}
  \end{equation*}
   If $\frac{\pi}{2} < \varphi < \pi$ then $\Sigma_1(A(\rho,\varphi)) =
  \Sigma_1(A(\rho,\pi-\varphi))$.

\begin{figure}[hbt]
\begin{center}
\includegraphics[width=0.99\textwidth]{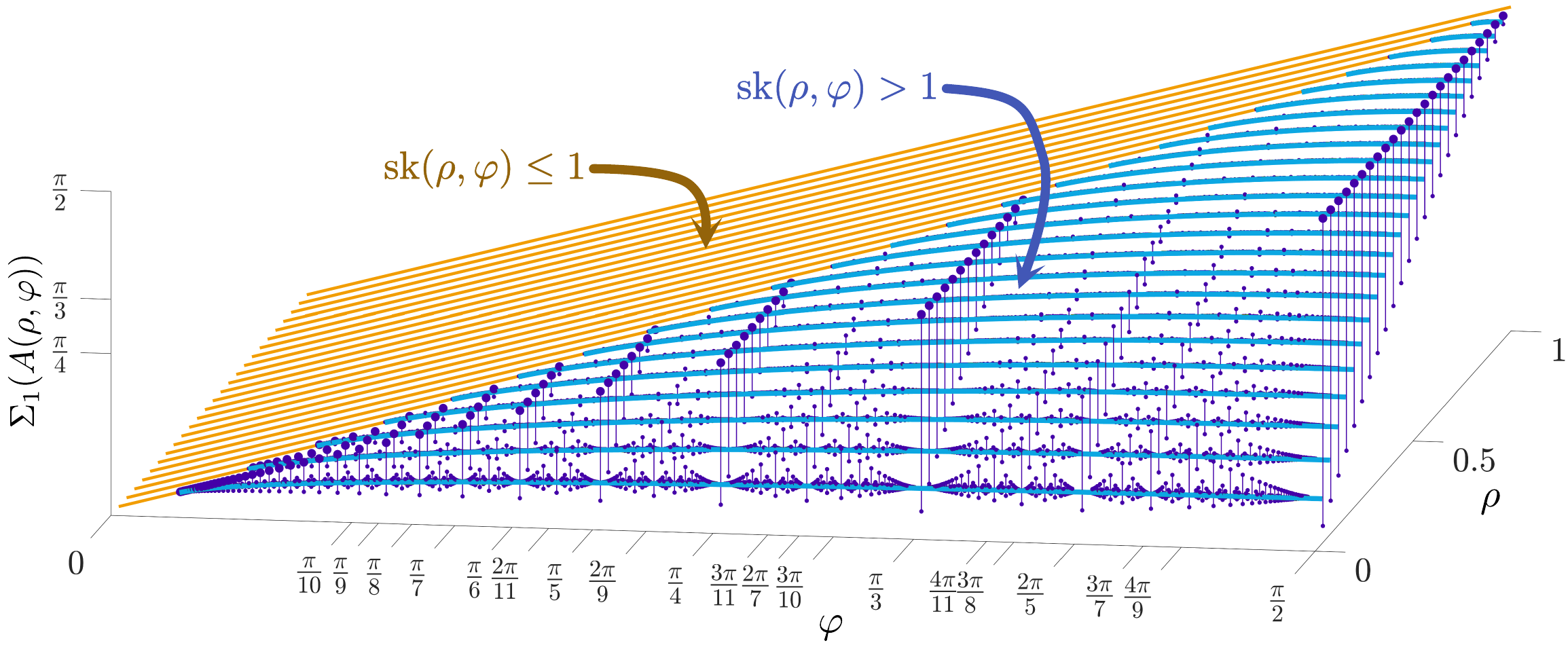}   
\end{center}
\caption{\label{explicit}Outer angular spectrum of the autonomous system
  $u_{n+1}=A(\rho,\varphi)u_n$ for $A(\rho,\varphi)$ from
  \eqref{rhomatrix} as a function of 
  the parameters  $\rho \in (0,1]$, $\varphi\in (0,\frac{\pi}{2}]$.} 
\end{figure}

Figure \ref{explicit} illustrates the rather irregular behavior of the
outer angular spectrum with respect to the parameters $\rho$ and
$\varphi$. In the hatched region $\mathrm{sk}(\rho,\varphi)\le 1$ we
always have point spectrum $\{\varphi\}$. One can show that in this
region no turnover occurs, i.e.\ the angle between the subspaces
$\mathrm{span}(v)$ and $\mathrm{span}(A(\rho,\varphi)v)$ always
coincides with the angle between the spanning vectors.  This changes
in the region $\mathrm{sk}(\rho,\varphi)>1$. Then we find point
spectrum at $\varphi$-values not in $\pi \Q$ while we have perfect
intervals for $\varphi$-values in $\pi \Q$. In the latter case the
left and right endpoint of the spectral interval are indicated in
Figure \ref{explicit} by a dot. Let us further note that for values
$\varphi\notin \pi \Q$ and $\mathrm{sk}(\rho,\varphi)>1$ the function
$\delta_{\rho,\varphi}$ becomes negative in some open subinterval of
$(0,\frac{\pi}{2})$, so that $\Sigma_1(A(\rho,\varphi))$ is strictly
below $\varphi$ (cf.\ \cite[(6.8)]{BeFrHu20}).  On the other hand, by
\cite[Theorem 6.1]{BeFrHu20} we have
$\max_{0 \le \theta \le \frac{\pi}{2}}G_q(\theta)=\varphi$ if
$\mathrm{sk}(\rho,\varphi)>1$ and $\varphi= \frac{\pi}{q}$ for some
$q \in \N, q\ge 2$, i.e.\ the maximum angular value is achieved for a
proper one-dimensional subspace.  Finally, we showed in \cite[Section
4.4]{BeHu23X}, \cite[Section 6.1]{BeFrHu20} that the maximum value
$\theta_1^{\sup,\lim}(A(\rho,\varphi))$ is upper but not lower
semicontinuous w.r.t.\ the parameter $\varphi$. Similarly, one finds
that $\theta_1^{\inf,\lim}(A(\rho,\varphi))$ is lower but not upper
semicontinuous w.r.t.\ $\varphi$. In view of the formula
\eqref{finalform} this shows that the angular spectrum is upper but
not lower semicontinuous w.r.t.\ $\varphi$ in the Hausdorff sense.
\end{example}
 
\subsection{Trace spaces and the reduction theorem for the angular
  spectrum}\label{SackerSell} 
In \cite[Section 3.1]{BeHu22} we introduced so called trace spaces which
are associated with the Grassmannian $\cG(s,d)$ and the system \eqref{diffeq}.
These are subspaces of dimension $s$ with a basis consisting of vectors from
the spectral bundles of the dichotomy spectrum; see Definition \ref{deftrace}
below. Trace spaces turn out to be the natural nonautonomous generalization of
invariant subspaces for autonomous systems.
They allow an efficient computation of outer angular values, by 
reducing the suprema and infima over the whole Grassmannian to the set of trace
spaces, see the reduction Theorem \ref{thm2:reduce} below.

The dichotomy spectrum, see \cite{ss78} is based on the
notion of an exponential dichotomy, cf.\ \cite{he81, ak01, k94, co78,
  dk74, p30}. We define this notion for
linear systems \eqref{diffeq} on a discrete time interval $J=\{n \in
\Z: n \ge n_-\}$ unbounded from above. For our purposes, 
it is convenient to use a generalized version which
allows an arbitrary split of the rates of solutions, not necessarily
into growing and decaying ones.

\begin{definition}\label{edDef}
The system \eqref{diffeq} has a  \textbf{generalized exponential
  dichotomy} (\textbf{GED}) on $J$ with rates
$0 \le \lambda_+ < \lambda_-\le \infty$ (\,
$(\lambda_+,\lambda_-)=(0,\infty)$ excluded), if there exist a
constant $K>0$ and families of 
projectors $P_n^+$, $P_n^- \coloneqq I-P_n^+$, $n\in J$ with the properties:
\begin{itemize}
\item [(i)] $P_{n}^{\pm} \Phi(n,m) = \Phi(n,m)P_m^{\pm}$ for all
  $n,m\in J$,
\item[(ii)] the following estimates hold for all $n,m\in J$
  \begin{align} 
    \|\Phi(n,m) P_m^+\| &\le K\lambda_+^{n-m},\quad n \ge m,\label{eq2:dich+} \\ 
\|\Phi(n,m) P_m^-\| &\le K \lambda_-^{n-m}, \quad n \le m. \label{eq2:dich-}
  \end{align}
\end{itemize}
The tuple $(K,\lambda_{\pm}, P_n^{\pm})_{n \in J}$ is called the
dichotomy data of \eqref{diffeq} resp.\ of $\Phi$.  
\end{definition}
\begin{remark}\label{rem2:extreme}
  If $\lambda_+=0, \lambda_-<\infty$ the estimate \eqref{eq2:dich+} is
  to be read with 
  $0^k=0$ for $k \ge 0$, so that $P_m^+=0$, $P_m^-=I$ follows and
  \eqref{eq2:dich-} is an upper bound for the inverse solution
  operator. Similarly, if 
  $0< \lambda_+,\lambda_-=\infty$ the estimate \eqref{eq2:dich-} is
  to be read with $\infty^{k}=0$ for $k \le 0$, so that $P_m^-=0$, $P_m^+=I$
follows and \eqref{eq2:dich+} is an upper bound for the solution operator.
\end{remark}
Recall that the system \eqref{diffeq} has a (standard) exponential dichotomy (ED)
(see e.g.\ \cite{as01,P2012,BeHu22}), if there exist constants
$\alpha_s,\alpha_u <1$ and projectors $P_n^s,P_n^u=I-P_n^s$  which
satisfy condition (i) of 
Definition \ref{edDef} and 
\begin{align*}
  \|\Phi(n,m)P_m^s\| \le K \alpha_s^{n-m}, \quad \|\Phi(m,n)P_n^u\|\le
  K \alpha_u^{n-m}, \quad \forall \, n\ge m\in J. 
\end{align*}
The relation between a GED and an ED is easily established via the scaled equation
\begin{equation}\label{scale}
u_{n+1} = \frac{1}{\lambda} A_n u_n,\quad n\in J, \quad \lambda>0
\end{equation}
which has the solution operator $\Phi_{\lambda}(n,m) = \lambda^{m-n} \Phi(n,m)$.
\begin{lemma} \label{relateEDGED}
  If $\Phi$ has a GED with data $(K,\lambda_{\pm}, P_n^{\pm})_{n \in J}$
  then $\Phi_{\lambda}$ has  an ED with data
  $( K,  P_n^s=P_n^+,P_n^u=P_n^-,  \alpha_s=\frac{\lambda_+}{\lambda},
       \alpha_u= \frac{\lambda}{\lambda_-})_{n \in J}$ for each
       $\lambda\in (\lambda_+,\lambda_-)$.\\ 
   Conversely, if $\Phi_{\lambda}$ has an ED with data $(K,
   \alpha_{s,u}, P_n^{s,u})_{n \in J}$ 
  then $\Phi$ has a GED with data $(K,\lambda_+=\alpha_s \lambda,
  \lambda_-=\frac{\lambda}{\alpha_u},P_n^+=P_n^s, P_n^- = P_n^u)_{n \in J}$.
\end{lemma}
For the further development of the theory (in particular Section
\ref{sec3}) it is suitable to use the 
GED setting throughout and avoid  working with the scaled equation
\eqref{scale}. 
As an example, we state and prove two properties of GEDs  which are
well-known for EDs. 
\begin{lemma} \label{lem3:GEDprop}
The  ranges of projectors  of a GED are uniquely determined by
\begin{equation} \label{eq2:rangechar}
  \range(P_m^+)= \{x\in \R^d: \exists C>0: \|\Phi(n,m)x\|\le C
  \lambda_+^{n-m}\|x\| \quad  \forall \; n \ge m \in J \}. 
\end{equation}
If $(\tilde{K},\lambda_{\pm},\tilde{P}_n^{\pm})_{n \in J})$ are the
data of  another GED with the same rates,  then 
the following holds
\begin{equation} \label{eq2:estproj2}
  \|P_n^+- \tilde{P}_n^+\| \le K^2
  \left(\frac{\lambda_+}{\lambda_-}\right)^{n-n_-} \|P_{n_-}^+-
  \tilde{P}_{n_-}^+\|, 
  \quad n\in J.
    \end{equation}
\end{lemma}
\begin{proof}
  The relation `` $\subseteq$ `` in \eqref{eq2:rangechar} is
  obvious. For  the converse, consider $x\in \R^d$ 
  with   
  $\|\Phi(n,m)x\|\le C \lambda_+^{n-m}\|x\|$ for $n \ge m$. Then we
  conclude  from \eqref{eq2:dich-}  
\begin{align*}
  \| P_m^- x \|=\|P_m^- \Phi(m,n)\Phi(n,m)x\| \le K C
  \left(\frac{\lambda^+}{\lambda^-}\right)^{n-m}\|x\| , \quad n\ge m. 
\end{align*}
The right-hand side converges to zero as $n \to \infty$, hence we
obtain $P_m^-x=0$ and $x \in \range(P_m^+)$. 
  For the second assertion note that 
$\range(P_n^+)=\range(\tilde{P}_n^+)$ implies $P_n^+\tilde{P}_n^+=
\tilde{P}_n^+$, $\tilde{P}_n^+P_n^+= P_n^+$ 
and, therefore, 
\begin{align*}
  P_n^+(P_n^+-\tilde{P}_n^+)P_n^-=
  P_n^+(P_n^+-\tilde{P}_n^+)(I-P_n^+)=P_n^+\tilde{P}_n^+(P_n^+-I)=P_n^+
  - \tilde{P}_n^+. 
\end{align*}
The dichotomy estimates then show
\begin{align*}
  \|P_n^+ - \tilde{P}_n^+\|&  =\|\Phi(n,n_-)(P_{n_-}^+ -
                             \tilde{P}_{n_-}^+) \Phi(n_-,n)\| \\ 
  & = \|\Phi(n,n_-)P_{n_-}^+(P_{n_-}^+ - \tilde{P}_{n_-}^+)P_n^-
    \Phi(n_-,n)\| \\ 
  & \le K \lambda_+^{n-n_-}\|P_{n_-}^+ - \tilde{P}_{n_-}^+\| K
    \lambda_-^{n_--n}. 
\end{align*}

\end{proof}
  
Next, we define the dichotomy spectrum and the resolvent set in the GED setting.
\begin{definition} \label{def2:dichspectrum}
  The dichotomy resolvent set and dichotomy spectrum are defined by
  \begin{equation} \label{eq2:defres}
    \begin{aligned}
      R_{\ED}& = \bigcup \{ (\lambda_+,\lambda_-): \eqref{diffeq}\; \text{has a GED
      with rates}\; \lambda_+<\lambda_- \; \text{on}\; J \}  \\
        \Sigma_{\ED}& = (0,\infty) \setminus R_{\ED}.
    \end{aligned}
    \end{equation}
\end{definition}
Note that $R_{\ED}$ is open and $\Sigma_{\ED}$ is closed relative to $(0,\infty)$.
Then we always have $(0,\varepsilon),(\varepsilon^{-1},\infty)
\subseteq R_{\ED}$ for some $\varepsilon>0$ since the matrices $A_n$
and their inverses are uniformly bounded. 
Further, the GED and thus also the spectral notions remain unchanged
when we replace 
$J$ by a smaller interval $\{n \in \Z:n \ge N\}$ for some $N\ge n_-$.
However, we 
do not set $n_-=0$ in general, since we will derive estimates with
constants independent of $n_-$ and then let $n_-$ tend to infinity. 

By Lemma \ref{relateEDGED},  the definition \eqref{eq2:defres}  agrees
with the usual 
one for spectrum and resolvent set in the ED setting (\cite{as01,P2012,BeHu22}):
\begin{align*} R_{\ED}&=\{\lambda > 0: \Phi_{\lambda} \text{ has an
  ED  on } J\},\\
  \Sigma_\ED & = \{\lambda > 0 : \Phi_{\lambda} \text{ has no ED on } J\}.
  \end{align*}

The Spectral Theorem \cite[Theorem 3.4]{as01} provides the
decomposition
$\Sigma_\ED = \bigcup_{k=1}^\varkappa \cI_k$ of the dichotomy spectrum into
$\varkappa \le d$ spectral intervals
$$
\cI_k = [\sigma_k^-, \sigma_k^+], \ k = 1,\dots,\varkappa,\quad \text{where}\quad
0 < \sigma_{\varkappa}^- \le \sigma_{\varkappa}^+ < \dots < \sigma_1^-
\le \sigma_1^+ < \infty. 
$$
Similarly, one decomposes the resolvent set
$R_\ED =
\bigcup_{k=1}^{\varkappa+1} R_k$ into disjoint open intervals 
\begin{equation*} \label{resolventset}
 R_k = (\sigma_k^+,\sigma_{k-1}^-),\
k=1,\ldots,\varkappa+1,  \text{ where } \sigma_{\varkappa+1}^+=0,\quad
\sigma_{0}^-= \infty, 
\end{equation*}
see Figure \ref{GED}.
If $\sigma_k^- = \sigma_{k}^+$ holds for an index $k\in
\{1,\dots,\varkappa\}$ then the 
 spectral interval $\mathcal{I}_k$ degenerates into an isolated point.
For every $\lambda \in R_k$, $k\in \{1,\dots,\varkappa+1\}$
one has dichotomy projectors   $P_{n,k}^+$, $P_{n,k}^-=I- P_{n,k}^+$
for $n \in J$ of \eqref{scale}, 
which depend on $k$ but not on $\lambda\in R_k$; see Lemma
\ref{relateEDGED}.  
By the characterization \eqref{eq2:rangechar} it is immediate that the
ranges of the projectors form 
an increasing flag of subspaces, i.e.
\begin{equation} \label{Rhierarchy}
\{0\} = \range(P_{n,\varkappa+1}^+) \subseteq
        \range(P_{n,\varkappa}^+)\subseteq \dots \subseteq \range(P_{n,1}^+) =
        \R^d.
\end{equation}
One can further choose the projectors such that  the nullspaces, or
equivalently the ranges of 
the complementary projectors, form a decreasing flag of subspaces:
\begin{equation} \label{Nhierarchy}
\R^d = \range(P_{n,\varkappa+1}^-) \supseteq
        \range(P_{n,\varkappa}^-)\supseteq \dots \supseteq \range(P_{n,1}^-) =
        \{0\}.
\end{equation}
Given this situation, one introduces spectral bundles as follows:  
\begin{equation}\label{specbun}
\begin{aligned}
\cW^k_n &\coloneqq \range (P_{n,k}^+) \cap \range (P_{n,k+1}^-),\quad k =
1,\dots, \varkappa.
\end{aligned}
\end{equation}

The fiber projector $\cP_{n,k},\ k=1,\ldots,\varkappa$  onto $\cW_n^k$ along
$\bigoplus_{\nu=1,\nu\neq k}^{\varkappa} \cW_n^{\nu}$ is given by
\begin{equation} \label{fiberproj}
  \cP_{n,k}=P_{n,k}^+ P_{n,k+1}^-=P_{n,k+1}^- P_{n,k}^+=P_{n,k}^+- P_{n,k+1}^+
  =P_{n,k+1}^- - P_{n,k}^-.
\end{equation}
Spectral bundles satisfy for $k=1,\dots,\varkappa$ and 
$n,m\in J$ the invariance condition
\begin{equation}\label{invar}
\Phi(n,m)\cW_m^k = \cW_n^k.
\end{equation}

\begin{figure}[hbt]
\begin{center}
\includegraphics[width=0.95\textwidth]{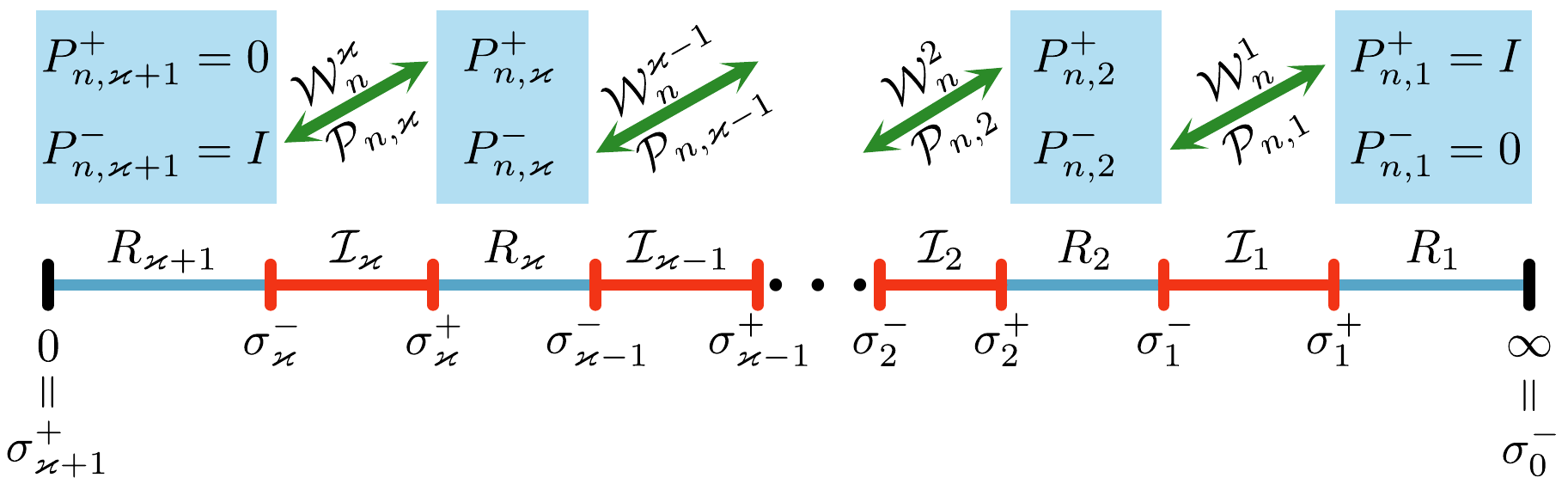}   
\end{center}
\caption{\label{GED} Illustration of spectral intervals (red), of resolvent
  intervals (blue) and the construction of spectral bundles with
  corresponding fiber projectors.}  
\end{figure}

\begin{definition}\label{deftrace}
Every element $V\in \cG(s,d)$ of the form
\begin{equation*} \label{tracespaceX}
  V = \bigoplus_{k=1}^{\varkappa}W_k: \quad W_k \subseteq \cW_n^k \;
  \text{(subspace)}\; k=1,\ldots,\varkappa,\quad
\sum_{k=1}^{\varkappa}\dim W_k=s
\end{equation*}
is called a \textbf{trace space} at time $n$. 
The set of all trace spaces at time $n$ is denoted by $\cD_n(s,d)$.
\end{definition}
The corresponding trace projectors $\cT_n:\cG(s,d)\to \cD_n(s,d)$ are defined by 
\begin{equation}\label{traceproj}
\cT_n(V) = \bigoplus_{k=1}^{\varkappa} (P_{n,k+1}^-(\range(P_{n,k}^+)\cap V) ).
\end{equation}
From (\cite[(3.18)]{BeHu22}) we have that the trace projectors are onto, i.e.
\begin{equation} \label{Donto}
\cD_n(s,d)=\{\cT_n(V):V \in \cG(s,d)\}.
\end{equation}
Moreover, from \eqref{invar}, \eqref{traceproj}, \eqref{Donto} one infers invariance
according to
\begin{equation} \label{Dinvar}
  \Phi(n,m)\cT_m(V)= \cT_n(\Phi(n,m)V), \quad \Phi(n,m) \cD_m(s,d)= \cD_n(s,d).
  \end{equation}
 Note that $\cD_n(s,d)$ is usually substantially smaller than $\cG(s,d)$
 and even finite if all spectral bundles are one-dimensional \cite[(3.25)]{BeHu22}.
We cite the reduction theorem  
\cite[Theorem 3.6]{BeHu22} on $J=\N_0$ by using the averages from \eqref{def:alpha}.
 
\begin{theorem}\label{thm2:reduce}
Assume that the difference equation \eqref{diffeq} has the dichotomy spectrum
$ \Sigma_\ED=\bigcup_{k=1}^{\varkappa}[\sigma_k^-,\sigma_k^+]$ with
fibers $\cW_n^k,\ k=1\ldots,\varkappa$ and trace projector \eqref{traceproj}.
Then the following equality holds for all $V\in \cG(s,d)$
\begin{equation}\label{obs1}
\begin{aligned}
  &\varlimsup_{n\to \infty}
  \alpha_n(V)
= \varlimsup_{n\to \infty}
\alpha_n(\cT_0(V)),
\end{aligned}
\end{equation}
and similarly with $\varliminf$ instead of $\varlimsup$.
\end{theorem}
An immediate consequence of this theorem is the reduction of  angular
values and of the 
outer angular spectrum to the set of traces spaces.
\begin{theorem} \label{thm2:reducespec}
  Under the assumption of Theorem \ref{thm2:reduce}
the outer angular values are given by
\begin{equation} \label{eq2:reduceval}
\begin{aligned}
  \theta_{s}^{\sup, \varlimsup} &= \sup_{V \in \cD_0(s,d)}
  \varlimsup_{n\to\infty} \alpha_n(V), \quad
    \theta_s^{\sup,\varliminf} =\sup_{V \in \cD_0(s,d)}
  \varliminf_{n\to\infty} \alpha_n(V), \\
  \theta_{s}^{\inf, \varlimsup}& = \inf_{V \in \cD_0(s,d)}
  \varlimsup_{n\to\infty} \alpha_n(V), \quad
    \theta_s^{\inf,\varliminf} =\inf_{V \in \cD_0(s,d)}
  \varliminf_{n\to\infty} \alpha_n(V),
\end{aligned}
\end{equation}
and the outer angular spectrum satisfies
\begin{equation} \label{eq2:reduceangspec}
\Sigma_s = \cl\big\{\theta \in [0,\tfrac \pi 2]: \exists V\in\cD_0(s,d): 
\varliminf_{n\to \infty} \alpha_n(V) \le \theta \le \varlimsup_{n\to \infty}
\alpha_n(V)\big\}.
\end{equation}
\end{theorem}

\begin{example} \label{ex2:revisit} (Example \ref{ex2:exauto} revisited) 
Spectral bundles of the dichotomy spectrum are eigenspaces given by 
\begin{align*}
\cW_0^1 &= \Span \begin{pmatrix}0 & 0 & 1\end{pmatrix}^\top,\quad
\cW_0^2 = \Span \left(\begin{pmatrix}1 & 0
  &0\end{pmatrix}^\top,  \begin{pmatrix}0 & 1 &0\end{pmatrix}^\top\right).
  \end{align*}
By Theorem \ref{thm2:reducespec} it suffices to consider for $s=1$ initial spaces
$V=\mathrm{span}(v)$ with either $v \in \cW_0^1$ or $v \in \cW_0^2$. This shows that the
computation for the mixed case in  Example \ref{ex2:exauto} was not necessary. 
Moreover, we  obtain the angular values listed in Table \ref{tab1} and
the angular spectrum for $s=1$ and $s=2$ as follows: 
\[
\Sigma_s = \{0,\varphi\} \subseteq [\theta_s^{\inf,\varliminf},
\theta_s^{\sup, \varlimsup}] = [0,\varphi].
\]
\begin{table}[hbt]
\begin{center}
\begin{tabular}{c|l}
angular value & achieved in\\[2mm]
$\theta_1^{\inf,\varliminf} = 0$ & $V = \cW_0^1$\\
$\theta_1^{\sup,\varlimsup} = \varphi $ & $V=\Span \begin{pmatrix}1 &
  0 & 0\end{pmatrix}^\top$\\ 
$\theta_2^{\inf,\varliminf} = 0$ & $V = \cW_0^2$\\
$\theta_2^{\sup,\varlimsup} = \varphi $ & $V=\Span
\left(\begin{pmatrix}0 & 0 & 1\end{pmatrix}^\top, \begin{pmatrix}1 
  & 0 & 0\end{pmatrix}^\top\right)$
\end{tabular}
\end{center}
\caption{Outer angular values for \eqref{mot3d}.\label{tab1}}
\end{table}
\end{example}

\begin{example} \label{ex6:mixed}
  As a continuation of Examples \ref{ex:normal2} and \ref{ex2:revisit}
  we  determine $\Sigma_2(A)$ for  
  \begin{equation} \label{eq6:auto3}
    A = \begin{pmatrix} B & b \\ 0 & \lambda \end{pmatrix} \in
    \R^{3,3}, \quad  b\in \R^2,\ |\lambda|\neq 1,\ B=A(\rho,\varphi)\ 
    (\text{see } \eqref{rhomatrix}).
  \end{equation}
    Note that  $\left(\begin{smallmatrix} w \\
        1 \end{smallmatrix}\right)$ with $w = (\lambda I_2 - B)^{-1}b$
    is an eigenvector 
    of $A$ with eigenvalue $\lambda$. Further, the matrix $A$ in
    \eqref{eq6:auto3} is the Schur normal form 
    of a general real $3 \times 3$-matrix which has a complex
    conjugate eigenvalue of modulus $1$ and a real eigenvalue 
    of modulus $\neq 1$.
  The dichotomy spectrum  is $\{1,|\lambda|\}$ and Theorem
  \ref{thm2:reducespec} shows  that  
  it suffices to study two-dimensional subspaces $V_0$ of $\R^3$ which
  have  basis  vectors 
  from $\R^2 \times \{0\}$ or
  $\mathrm{span}\left(\begin{smallmatrix} w \\ 1 \end{smallmatrix}
  \right)$. If both are from 
  the first space then the resulting angular value vanishes. Hence, we
  consider 
  \begin{align*}
    V_0 = \mathrm{span} \left( \begin{pmatrix} x_0 \\
        0 \end{pmatrix}, \begin{pmatrix} w \\ 1 \end{pmatrix} \right), 
    \quad x_0\in \R^2,\ x_0 \neq 0.
  \end{align*}
  For the iterates we  have
  \begin{align*}
    V_j =A^j V_0 =  \mathrm{span} \left( \begin{pmatrix} x_j \\
        0 \end{pmatrix}, \begin{pmatrix} w \\ 1 \end{pmatrix} \right),
    \quad x_j = B^j x_0. 
  \end{align*}
   An orthogonal (but not necessarily orthonormal) basis is 
  \begin{align*}
    V_j = \mathrm{span}\left(Qx_j, \begin{pmatrix} w\\
        1 \end{pmatrix}\right), \quad Q= \begin{pmatrix} 
      (1+ \|w\|^2)I_2 - w w^{\top} \\ - w^{\top} \end{pmatrix} \in
    \R^{3,2} . 
    \end{align*}
   According to Example \ref{ex2:s=2,d=3} the maximal principal angle
   between $V_j$ and $V_{j-1}$ is given by 
    \begin{equation} \label{eq6:vjformula} 
    \ang(V_j,V_{j-1})= \ang( Qx_j,Qx_{j-1}).
  \end{equation}
    Similar to \cite[Section 6.1]{BeFrHu20} we determine
  \begin{align*}
    \theta_2(\eta)= \lim_{n \to \infty} \frac{1}{n} \sum_{j=1}^n
    \ang(V_j,V_{j-1}), \quad 
    x_0=r(\eta) \coloneqq \begin{pmatrix} \cos(\eta) \\
      \sin(\eta) \end{pmatrix},\quad  \eta \in [0,2 \pi)
  \end{align*}
  by applying ergodic theory to the function
  \begin{equation} \label{eq6:gdef}
    g(\eta) = \ang \big( Q r(T_{\rho,\varphi}\eta),Q r(\eta)
    \big),\quad   T_{\rho,\varphi}\eta= \Psi_{\rho}\big(\varphi+
    \Psi_{\rho}^{-1}(\eta)\big), 
  \end{equation}
  where $\Psi_{\rho}$ is defined in \eqref{defdeltag}. The result is
     \begin{equation} \label{eq6:erglimit}
        \theta_2(\eta)=
        \begin{cases} \frac{1}{2 \pi}\int_0^{2 \pi} g(\Psi_{\rho}(\xi)) d\xi, &
          \frac{\varphi}{\pi} \notin \Q, \\
          \frac{1}{q} \sum_{\ell=0}^{q-1}
          g(T_{\rho,\varphi}^{\ell}\eta), & \frac{\varphi}{\pi}= 
          \frac{p}{q}, p \perp q, p,q \in \N.
        \end{cases}
     \end{equation}
     The second outer angular spectrum is then of the form
     \begin{align} \label{eq6:Sigma2}
     \Sigma_2(A) =\begin{cases}  \{0,\theta_2\}, \theta_2\equiv
       \theta_2(\eta), & \frac{\varphi}{\pi} \notin \Q, \\ 
     \{0\} \cup [\min_{\eta} \theta_2(\eta), \max_{\eta}
     \theta_2(\eta)], & \frac{\varphi}{\pi} \in \Q, 
     \end{cases}
     \end{align}
     where $\theta_2(\eta)$ is given by \eqref{eq6:gdef},
     \eqref{eq6:erglimit}.

We numerically compute the second outer angular spectrum for the
non-resonant case $\varphi= 1.25$ as well as
for the resonant case $\varphi = \frac 25 \pi\approx 1.2566$ nearby. 
 In particular, we illustrate in Figures \ref{reso} and \ref{nonreso} the
influence of $\rho\in(0,1]$ and of the angle between the eigenvector 
$\left(\begin{smallmatrix}w\\1\end{smallmatrix}\right)$ and the
$xy$-plane $X \coloneqq
\Span\left(\left(\begin{smallmatrix}1\\0\\0\end{smallmatrix}\right),
  \left(\begin{smallmatrix}0\\1\\0\end{smallmatrix}\right) \right)$.
This angle is given by
\[
\ang\left(X,\begin{pmatrix}w\\1\end{pmatrix}\right)
=: \gamma_w\quad \text{where} \quad
\cos(\gamma_w)  
=\frac{\|w\|}{\sqrt{\|w\|^2+1}}.
\]
We choose $\gamma_w\in(0,\frac \pi 2]$ and $w = \frac
{\cos(\gamma_w)}{\sqrt{1+\cos(\gamma_w)^2}}
\left(\begin{smallmatrix}1\\0\end{smallmatrix}\right)$    
in Figures \ref{nonreso} and \ref{reso}.

\begin{figure}[hbt]
    \begin{center}
      \includegraphics[width=0.99\textwidth]{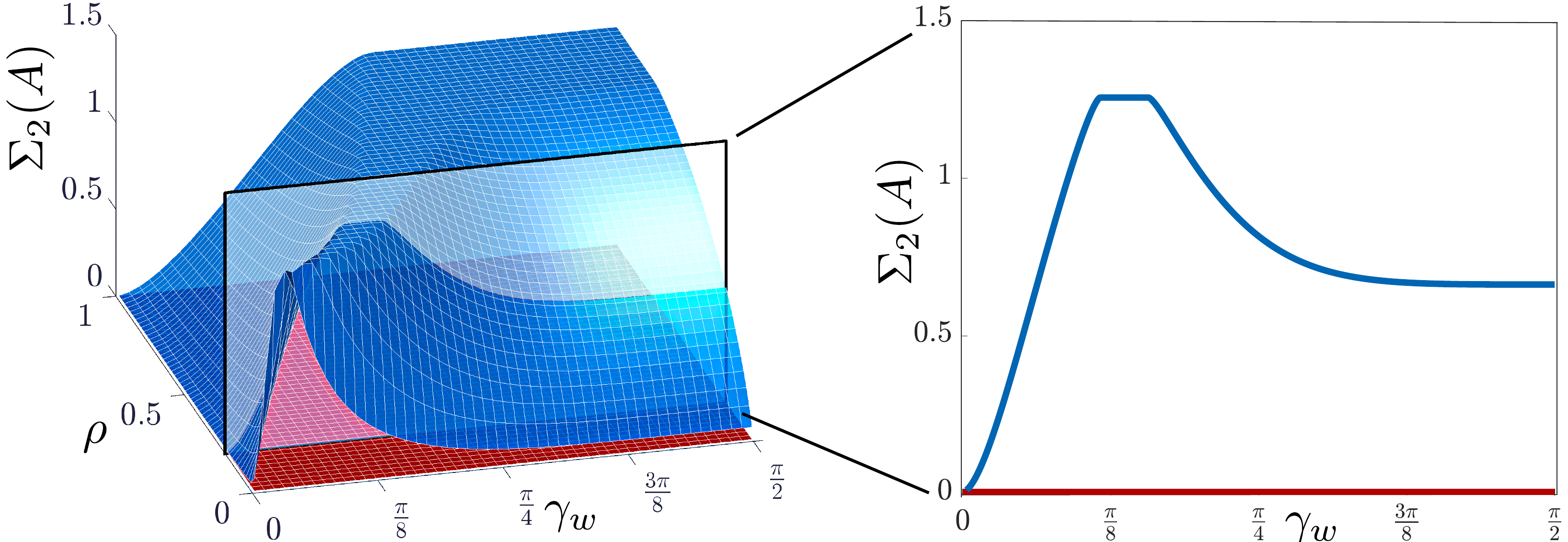}
    \end{center}
\caption{Computation of the second angular spectrum for the 
non-resonant case $\varphi = 1.25$. 
Left panel:  For each $(\gamma_w,\rho)$-pair we obtain point
  spectrum only (red plane at $0$ and blue
  surface). Right panel: Point spectrum for fixed $\rho =
  0.2$.\label{nonreso}} 
\end{figure}

\begin{figure}[hbt]
    \begin{center}
      \includegraphics[width=0.99\textwidth]{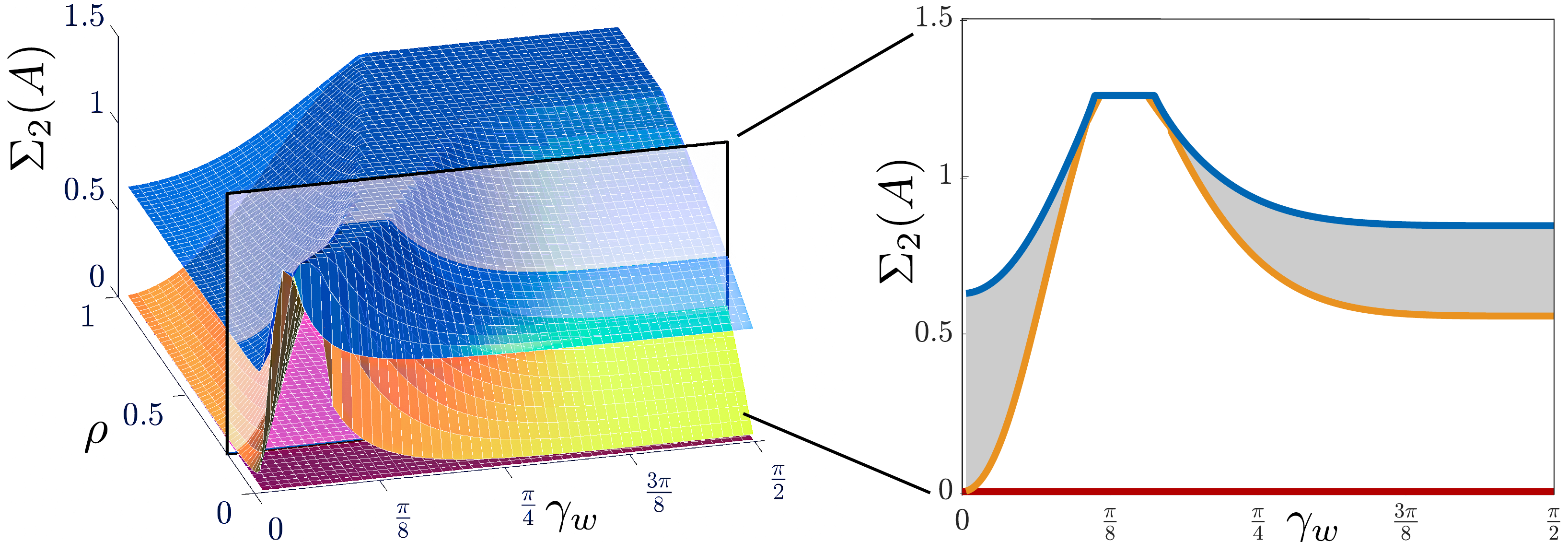}
    \end{center}
\caption{Computation of the second angular spectrum for the 
resonant case $\varphi = \frac 25 \pi$. 
 Left panel: For each $(\gamma_w,\rho)$-pair the spectrum consists of
 the point spectrum $0$ (red plane) and of the interval between the
 orange and the 
 blue surface. Note that in certain areas, both surfaces coincide and the
 spectral intervals degenerate to a point. 
 In the intersection of this figure with the plane at $\rho
 = 0.2$ (right panel) the spectrum consists of the curves and the grey
 area between them. \label{reso}} 
\end{figure}
 
\end{example}


\section{Perturbation theory and invariance of the outer angular spectrum}
\label{sec3}
In this section we present some of the main theoretical results of the
paper. We first show that two 
dynamical systems have the same outer angular spectrum if they are
kinematically similar by a transformation 
which becomes orthogonal at infinity. 
In the second step we consider dynamical systems which have a rather
strict dichotomy structure called 
complete exponential dichotomy (CED). This notion may be considered as
a generalization of an autonomous 
system with a matrix having only semi-simple eigenvalues. Such systems
do not only have a
dichotomy point spectrum which is invariant under
$\ell^1$-perturbations (a property that holds more generally 
\cite{P2012}),  but also their outer angular
spectra persist under these perturbations.
\subsection{Kinematic similarity and invariance}

Consider a system which is kinematically similar to \eqref{diffeq},
i.e.\ we transform variables by 
 $v_n=Q_n u_n$ with $Q_n \in \mathrm{GL}(\R^d)$ to obtain
\begin{equation} \label{difftransform2}
  v_{n+1} = \tilde{A}_n v_n, \quad \tilde{A}_n =
  Q_{n+1}A_n Q_n^{-1}, \quad n \in J.
\end{equation}
The corresponding solution operators
$\tilde{\Phi}(n,m)$ and $\Phi(n,m)$ are related by
\begin{equation}\label{relatePhi2}
  \tilde{\Phi}(n,m)Q_m= Q_n \Phi(n,m), \quad n,m \in J.
\end{equation}
We show that the outer angular spectrum is invariant under two types of
similarity transformation.
\begin{proposition}(Invariance of angular spectrum) \label{prop3:inv}
  The outer angular spectrum of the systems \eqref{diffeq} and
  \eqref{difftransform2} coincide in the following two cases: 
  \begin{itemize}
    \item[(i)] $Q_n= q_n I_d$ holds with $q_n\neq 0$ for all $n \in \N_0$.
    \item[(ii)]
   The limit  $\lim_{n \to \infty}Q_n=Q\in \R^{d,d}$ exists and is
  orthogonal.
   \end{itemize}
  \end{proposition}
\begin{proof} {\it (i):}
  From \eqref{relatePhi2} we have $\tilde{\Phi}(n,0)= \frac{q_n}{q_0}
  \Phi(n,0)$, hence  $\tilde{\Phi}(n,0)V$ and $\Phi(n,0)V$ agree and
  the assertion follows from Definition \ref{def3:1}. 

  {\it (ii):}
  In \cite[Proposition 4.3]{BeHu23X} we have shown that for any
  $\varepsilon>0$ there exists $N=N(\varepsilon)$ such that for all $n
  \ge N$ and for all $V \in \mathcal{G}(s,d)$ 
  \begin{equation} \label{basicPhi2}
   \frac{1}{n} \sum_{j=1}^n
   |\ang(\tilde{\Phi}(j-1,0)Q_0V,\tilde{\Phi}(j,0)Q_0V) -
   \ang(\Phi(j-1,0)V,\Phi(j,0)V)| \le \varepsilon. 
  \end{equation}
  From this estimate we infer that the averages
  \begin{align*}
    \tilde{\alpha}_n(V)= \frac 1n
    \sum_{j=1}^n\ang(\tilde{\Phi}(j-1,0)V,\tilde{\Phi}(j,0)V), 
    \quad V \in \cG(s,d) 
  \end{align*}
  satisfy
  \begin{align*}
    \varlimsup_{n\to \infty}\tilde{\alpha}_n(Q_0V)=\varlimsup_{n \to
    \infty}\alpha_n(V), 
    \quad
    \varliminf_{n\to \infty}\tilde{\alpha}_n(Q_0V)=\varliminf_{n \to
    \infty}\alpha_n(V). 
  \end{align*}
  Hence the equality of the outer angular spectra and the outer
  angular values follows. 
\end{proof}

\subsection{Complete exponential dichotomy (CED) and persistence of
  angular spectrum} 
\label{sec3:2}
The following definition will be essential for the subsequent perturbation theory.
\begin{definition} \label{def:CED}
  The system \eqref{diffeq} has a \textbf{complete exponential
    dichotomy (CED) } 
  if there exist nontrivial projectors $\cP_{n,k}$, $n\in J$,
  $k=1,\ldots,\varkappa$ 
  and constants $K>0$ and $0 < \sigma_{\varkappa}< \dots < \sigma_1 <\infty$ 
  such that the following properties hold for all  $n,m \in J$ and $
  k,\ell =1,\ldots, \varkappa$: 
  \begin{align} \label{eq3:prop1CED}
           \cP_{n,k} \cP_{n,\ell}& =\delta_{k \ell}\cP_{n,k}, \quad
                                   \sum_{j=1}^{\varkappa} \cP_{n,j} =
                                   I,  
        \\
      \Phi(n,m)\cP_{m,k}& = \cP_{n,k} \Phi(n,m),   \label{eq3:prop2CED} \\
      \|\Phi(n,m)\cP_{m,k}\| & \le K \sigma_k^{n-m}. \label{eq3:prop3CED}
    \end{align}
 \end{definition}
Note that there is no restriction $n \ge m$ in the estimate 
\eqref{eq3:prop3CED}. As usual, we denote by $(K,\sigma_k,\cP_{n,k})_{n\in
  J}^{k=1,\ldots,\varkappa}$ the dichotomy data of the CED.

\begin{proposition} \label{prop3:specCED}
  If the system \eqref{diffeq} has a CED with data
  $(K,\sigma_k,\cP_{n,k})_{n\in J}^{k=1,\ldots,\varkappa}$
  then  the following estimates hold:
  \begin{equation} \label{eq3:estboth}
    \frac{1}{K} \sigma_{k}^{n-m}\|x\| \le \| \Phi(n,m) x \| \le K
    \sigma_{k}^{n-m}\|x\| \quad \forall  x\in \range(\cP_{m,k}),\
    n,m\in J,\ 
    k=1,\ldots,\varkappa.
  \end{equation}
\end{proposition}
\begin{proof}
  For $x=\cP_{m,k}x$ the estimate \eqref{eq3:prop3CED} implies the
  upper estimate in \eqref{eq3:estboth} and also 
  the lower one as follows:
 \begin{align*}
 \|x\|& = \|\cP_{m,k} x\| = \|\Phi(m,n)\Phi(n,m)\cP_{m,k} x\| \le
    \|\Phi(m,n)\cP_{n,k}\| \|\Phi(n,m) x\|\\
&\le K\sigma_{k}^{m-n} \|\Phi(n,m)  x\|.  
\end{align*}
\end{proof}
 By this proposition the elements of $\range(\cP_{m,k})$ evolve at the
 exact rate $\sigma_k$ and the rates 
 are uniquely determined. In fact, they form the dichotomy spectrum; see
 Theorem \ref{thm3:maininv} below. Nevertheless, the projectors need not be
unique, as the following example shows.

\begin{example}\label{ex3:notunique}
  Consider the autonomous system $u_{n+1}=Au_n$ with 
  \begin{align*}
    A= \mathrm{diag}(\sigma_1,\sigma_2,\sigma_3)\in \R^3, \quad 0<\sigma_3 <
    \sigma_2=1 <\sigma_1.
  \end{align*}
  Obviously, the system has  a CED with $\varkappa=3$
  and projectors $\cP_{n,k}= e^k (e^k)^{\top}$, $k=1,2,3$. However, we can
  also set
  \begin{align*}
    \tilde{\cP}_{0,3}&=e^3 (e^3-e^2)^{\top},\quad \tilde{\cP}_{n,3}=
                       A^n \tilde{\cP}_{0,3} A^{-n}=
                       e^3(e^3-\sigma_3^n e^2)^{\top}, \\ 
    \tilde{\cP}_{0,2}& = (e^3+e^2)(e^2)^{\top}, \quad
                       \tilde{\cP}_{n,2}= A^n \tilde{\cP}_{0,2}
                       A^{-n}= (\sigma_3^n e^3 +e^2)(e^2)^{\top},\\ 
 \tilde{\cP}_{n,1}& =P_{n,1}.
  \end{align*}
  Since we have $\|uv^{\top}\|=\|u\| \|v\|$ for the spectral norm of a
  rank $1$ matrix we find 
  for $n,m \in \N_0$
  \begin{align*}
    \|A^{n-m}\tilde{\cP}_{m,3}\| &
                                   =\|(\sigma_3^{n-m}e^3)(e^3-\sigma_3^m
                                   e^2)^{\top}\|
=\sigma_3^{n-m}\sqrt{1 + \sigma_3^{2m}} \le \sqrt{2} \sigma_3^{n-m}, \\ 
    \|A^{n-m}\tilde{\cP}_{m,2}\| &
                                   =\|(\sigma_3^{n}e^3+e^2)(e^2)^{\top}\|
                                   =\sqrt{1 + \sigma_3^{2n}} \le
                                   \sqrt{2}=\sqrt{2} \sigma_2^{n-m},\\ 
    \|A^{n-m}\tilde{\cP}_{m,1}\| & = \sigma_1^{n-m}.
      \end{align*}
  Thus the system also has a CED with data
  $(\sqrt{2},\sigma_k,\tilde{\cP}_n^k)_{n\in \N_0}^{k=1,2,3}$. 
\end{example}
In the following we consider a perturbed system
\begin{equation} \label{eq3:perturbsyst}
  v_{n+1}=\tilde{A}_nv_n, \quad \tilde{A}_n=A_n + E_n, \quad n \in J
\end{equation}
with solution operator $\tilde{\Phi}$
and state our main result.
\begin{theorem} \label{thm3:maininv}
  Let the system \eqref{diffeq} have a CED with data
  $(K,\sigma_k,\cP_{n,k})_{n\in J}^{k=1,\ldots,\varkappa}$ and let the perturbations
  $E_J= (E_n)_{n \in J}$ be such that $A_n+E_n$ is invertible for all
  $n \in J$ and  
  \begin{equation} \label{eq3:L1norm}
    \|E_J\|_{\ell^1}= \sum_{n \in J}\|E_n\| <\infty.
  \end{equation}
  Then the perturbed system \eqref{eq3:perturbsyst} has a CED with data
$(\tilde{K},\sigma_k,\tilde{\cP}_{n,k})_{n\in J}^{k=1,\ldots,\varkappa}$
  and there exists a constant $C>0$ such that 
  \begin{equation} \label{eq3:projest}
    \sup_{n \in J} \|\cP_{n,k}- \tilde{\cP}_{n,k} \| \le C \sum_{n \in
      J} \|E_n\|,  \quad 
     \lim_{n  \to \infty}\cP_{n,k}- \tilde{\cP}_{n,k}=0, \quad k=1,\ldots,\varkappa.
  \end{equation}
  Moreover, both systems have the same dichotomy spectrum and the same
  outer angular spectrum 
  \begin{equation} \label{eq3:specequal}
    \Sigma_{\ED}(\Phi)=\{\sigma_1,\ldots,\sigma_{\varkappa}\}=
    \Sigma_{\ED}(\tilde{\Phi}), 
    \quad \Sigma_s(\Phi) = \Sigma_s(\tilde{\Phi}).
  \end{equation}
\end{theorem}
\begin{remark} \label{rem3:assume}
  From \eqref{eq3:L1norm} and the uniform invertibility of $A_n$ we
  conclude that 
  $A_n+E_n$ is uniformly invertible for $n$ sufficiently
  large. Therefore, it suffices 
  to assume invertibility of $A_n+E_n$ for finitely many $n\in
  J$. Likewise, one may shrink 
  $J$ to $\{n\in \Z:n\ge N\}$ so that this assumption is not needed at all.
\end{remark}
The proof of this theorem needs several preparation and will be given
in section \ref{roughproof}. The final 
goal is to show that the system \eqref{eq3:perturbsyst} is
kinematically similar to 
the original one \eqref{diffeq} so that Proposition \ref{prop3:inv}
applies. 

Before proceeding we consider the autonomous case.
\begin{example} \label{ex3:semisimple}
   In the autonomous case $A_n\equiv A$ the system \eqref{diffeq} has a CED
  if $A$ is invertible and (complex) diagonalizable.
  To see this, transform $A$ into real block-diagonal form where each
  block collects 
  eigenvalues of the same absolute value, i.e.
  \begin{equation} \label{eq3:blockabs}
    \begin{aligned}
      S^{-1}A S & = \mathrm{diag}(A_1,\ldots,A_{\varkappa}),\; A_k\in \R^{d_k,d_k},\\
      A_k& = \sigma_k \mathrm{diag}(T_{\varphi_1}, \ldots, T_{\varphi_{\ell_k}},
      s_1,\ldots, s_{d_k -2 \ell_k}),
    \end{aligned}
  \end{equation}
  where $T_{\varphi}$ is the rotation matrix from \eqref{mot3d}  and
  $s_j \in \{-1,1\}$. For the spectral norm $\|\cdot\|_2$ we obtain
  $\|A_k\|_2^{\nu} = \sigma_k^{\nu}$ for all $\nu \in \Z$. Further the
  projectors are given 
  for $k=1,\ldots,\varkappa$, $n \in \N_0$
  by $\cP_{n,k} = S \mathrm{diag}(0,\ldots,0,I_{d_k},0,\ldots,0) S^{-1}$, and 
  they commute with $A$. Finally, the norm is defined by
  \begin{align*}
    \|x\|_S^2 =\sum_{k=1}^{\varkappa}\|
    \mathrm{diag}(0,\ldots,0,I_{d_k},0,\ldots,0)S^{-1} x \|_2^2, \quad
    x \in \R^d. 
  \end{align*}
  A computation shows that $\|A^{\nu}\cP_{n,k}x\|_S = \sigma_k^{\nu} \|x\|_S$
  holds  for all $x$ and $\nu \in \Z$ so that conditions
  \eqref{eq3:prop1CED}-\eqref{eq3:prop3CED} are satisfied. 
  Conversely, let $A$ have an eigenvalue $\lambda_k\in \C$ of absolute
  value $\sigma_k=|\lambda_k|$ which is 
  not semi-simple. Then there exists a vector $x \in \R^d$ with
  $\|A^nx\| \asymp n \sigma_k^n \|x\|$ which violates 
  the rate estimate \eqref{eq3:estboth}. Hence the semi-simplicity of
  eigenvalues is also a necessary condition for a CED 
  to hold.
\end{example}
\begin{remark} According to Theorem \ref{thm3:maininv}, a CED implies
  pure point spectrum. 
  However, the converse is not necessarily true as shown by Example
  \ref{ex3:semisimple}. 
  Therefore,  the CED is rather a
  generalization of semi-simple eigenvalues than of an isolated
  spectrum to nonautonomous linear 
  dynamical systems.
  \end{remark}

\subsection{Relation between GED and CED}
The following proposition relates the GED and the CED notion to each other.
  
\begin{proposition} \label{lem3:CEDspec}
  Let the system \eqref{diffeq} and
  $\sigma_{\varkappa+1}=0<\sigma_{\varkappa}< \cdots < 
  \sigma_{1} < \sigma_{0}= \infty$ be given. Then the following holds.
  \begin{itemize}
    \item[(i)]
  If the system \eqref{diffeq} has a CED with data
  $(K,\sigma_k,\cP_{n,k})_{n \in J}^{k=1,\ldots,\varkappa}$ 
  then it has a GED with data
  $(\varkappa
  K,\sigma_{k+1},\sigma_k,P_{n,k}^+=\sum_{\ell=k+1}^{\varkappa}
  \cP_{n,\ell}, 
  P_{n,k}^- = \sum_{\ell=1}^{k} \cP_{n,\ell})_{n \in J}$ for each
  $k=0,\ldots,\varkappa$. 
\item[(ii)] If the system \eqref{diffeq} has a GED with data
  $(K,\sigma_{k+1},\sigma_k,P_{n,k}^+,P_{n,k}^-)_{n \in J}$ for 
  $k=0,\ldots,\varkappa$ then it has a CED with data
  $(K^{\varkappa+1},\sigma_k,\cP_{n,k})_{n \in
    J}^{k=1,\ldots,\varkappa}$ where for $n \in J$ 
  \begin{equation} \label{eq3:defcP}
    \cP_{n,k} = P_{n,k}^-  P_{n,k-1}^+ \cdots P_{n,0}^+=(I-P_{n,k}^+)
    P_{n,k-1}^+ \cdots P_{n,0}^+, \quad k=1,\ldots,\varkappa. 
  \end{equation}
  \end{itemize}
\end{proposition}

\begin{remark} \label{rem3:trivialcase} In assertion (i) we set
  $\sum_{\emptyset} =0$  so that we have projectors 
  $P_{n,\varkappa}^+=0$ and $P_{n,0}^- =0$. Similarly, note that
  $P_{n,0}^+=I$ and $P_{n,\varkappa}^-=I$ are used  in
  \eqref{eq3:defcP}. 
\end{remark}

\begin{proof}
  (i) \\
  Because of the relations \eqref{eq3:prop1CED}- \eqref{eq3:prop3CED}
  the matrices $P_{n,k}^+=\sum_{\ell=k+1}^{\varkappa} \cP_{n,\ell}$ 
  and $P_{n,k}^- = \sum_{\ell=1}^{k} \cP_{n,\ell}$  satisfy the invariance
  condition (i) in Definition \ref{edDef} and moreover 
  \begin{align*}
    \|\Phi(n,m)P_{m,k}^+\| &  \le \sum_{\ell=k+1}^{\varkappa}
                             \|\Phi(n,m) \cP_{m,\ell}\| 
    \le K \sum_{\ell=k+1}^{\varkappa} \sigma_{\ell}^{n-m} \le
                             (\varkappa-k) K \sigma_{k+1}^{n-m},\quad
                             n\ge m, \\ 
    \|\Phi(n,m)P_{m,k}^-\| &  \le \sum_{\ell=1}^{k} \|\Phi(n,m) \cP_{m,\ell}\|
    \le K \sum_{\ell=1}^{k} \sigma_{\ell}^{n-m} \le k  K
                             \sigma_{k}^{n-m}, \quad n<m. 
  \end{align*}
  (ii) \\
  From the ordering $\sigma_{k+1} < \sigma_{k}$ and the
  characterization \eqref{eq2:rangechar} we obtain 
  \begin{equation} \label{eq3:orderrange}
    \range(P_{m,k}^+)\subseteq \range(P_{m,\ell}^+) \text{ ,
      equivalently, } \quad P_{m,\ell}^+ P_{m,k}^+ = P_{m,k}^+ \;
    \text{for} \; 
     \ell \le k,
  \end{equation}
  i.e.\ projectors with a larger index annihilate those with smaller
  index from the right. 
  We verify the properties \eqref{eq3:prop1CED}-\eqref{eq3:prop3CED}
  for the operators $\cP_n^k$ from \eqref{eq3:defcP}: 
  \begin{align*}
    \sum_{k=1}^{\varkappa} \cP_{m,k}& =
                                      \sum_{k=1}^{\varkappa}P_{m,k-1}^+
                                      \cdots P_{m,0}^+ -
                                      \sum_{k=1}^{\varkappa}P_{m,k}^+ 
 \cdots P_{m,0}^+ \\
 &= P_{m,0}^+ -  P_{m,\varkappa}^+ \cdots P_{m,0}^+ = I - 0 =I.
  \end{align*}
  Further,   we have $\Phi(n,m)\cP_{m,k}= \cP_{n,k} \Phi(n,m)$ since
  all factors in the definition \eqref{eq3:defcP} 
  have this property.
  Now we use \eqref{eq3:orderrange} to show the orthogonality
  relations \eqref{eq3:prop1CED}. For $k > \ell$ we find with
  $P_{m,\ell}^-= I - P_{m,\ell}^+$ 
  \begin{align*}
    \cP_{m,k} \cP_{m,\ell} & = P_{m,k}^- P_{m,k-1}^+ \cdots P_{m,0}^+
                             P_{m,\ell-1}^+ \cdots P_{m,0}^+ 
    - P_{m,k}^- P_{m,k-1}^+ \cdots P_{m,0}^+ P_{m,\ell}^+ \cdots P_{m,0}^+ \\
  &=  P_{m,k}^-  P_{m,k-1}^+ \cdots P_{m,\ell}^+ P_{m,\ell-1}^+ \cdots P_{m,0}^+
    -P_{m,k}^- P_{m,k-1}^+ \cdots P_{m,\ell}^+ \cdots P_{m,0}^+=0.
  \end{align*}
  For $k=\ell$ the first term reproduces $\cP_{m,k}$ while the second term vanishes
  \begin{align*}
    \cP_{m,k} \cP_{m,k} & = P_{m,k}^- P_{m,k-1}^+ \cdots P_{m,0}^+
                          P_{m,k-1}^+ \cdots P_{m,0}^+ 
    - P_{m,k}^- P_{m,k-1}^+ \cdots P_{m,0}^+ P_{m,k}^+ \cdots P_{m,0}^+ \\
  &=  P_{m,k}^-  P_{m,k-1}^+  \cdots P_{m,0}^+
    -P_{m,k}^- P_{m,k}^+ \cdots P_{m,0}^+= \cP_{m,k}.
  \end{align*}
  For $k < \ell$ both terms vanish due to  $P_{m,k}^- P_{m,j}^+=0$ for
  $j=\ell-1,\ell$: 
  \begin{align*}
    \cP_{m,k} \cP_{m,\ell} & = P_{m,k}^- P_{m,k-1}^+ \cdots P_{m,0}^+
                             P_{m,\ell-1}^+ \cdots P_{m,0}^+ 
    - P_{m,k}^- P_{m,k-1}^+ \cdots P_{m,0}^+ P_{m,\ell}^+ \cdots P_{m,0}^+ \\
  &=  P_{m,k}^-  P_{m,\ell-1}^+  \cdots P_{m,0}^+
    -P_{m,k}^-  P_{m,\ell}^+ \cdots P_{m,0}^+=0-0.
  \end{align*}
  Finally, we use the bounds $\|P_{m,k}^+\|,\|P_{m,k}^-\| \le K$ of
  projectors and the dichotomy estimates 
  \begin{align*}
    \| \Phi(n,m) \cP_{m,k}\|& =\| P_{n,k}^- \Phi(n,m) P_{m,k-1}^+\cdots P_{m,0}^+\|
    \le K^2 \sigma_k^{n-m} \|P_{m,k-2}^+\cdots P_{m,0}^+\| \\
    & \le K^{k+1}\sigma_k^{n-m} \quad \text{for} \quad n \ge m,\\
    \| \Phi(n,m) \cP_{m,k}\|& = \| \Phi(n,m) P_{m,k}^- P_{m,k-1}^+
                              P_m^{k,1}\cdots P_{m,0}^+\| 
    \le K \sigma_k^{n-m}\| P_{m,k-1}^+\cdots P_{m,0}^+\| \\
    & \le K^{k+1} \sigma_k^{n-m} \quad \text{for} \quad n<m.
  \end{align*}

\end{proof}

\begin{remark} \label{rem3:prodproj}
 Note that the construction \eqref{eq3:defcP} of the fiber projectors
 as a product is more involved than \eqref{fiberproj} 
since only the first flag \eqref{Rhierarchy} holds automatically by
\eqref{eq3:orderrange} but not necessarily 
the second one \eqref{Nhierarchy}.
\end{remark}

\subsection{Perturbation theory for CEDs}
\label{sec3:3}
Results on perturbations of an exponential dichotomy have a long
tradition, beginning with the classical 
roughness theorems by Coppel \cite[Lecture 4]{co78} for ODEs. If the
perturbations are small in  $L^{\infty}$ 
then one obtains an exponential dichotomy with slightly weaker 
exponents (resp.\ rates) while $L^1$ perturbations 
allow to keep the exponents (resp.\ rates) for the perturbed system
\cite[Propositions 1 \& 2]{co78}. 
The roughness theorem below transfers this principle to GEDs for
systems in discrete time. 
From this a  roughness theorem for CEDs will then follow via
Proposition \ref{lem3:CEDspec}. 
Several roughness theorems for EDs in discrete time and even in
noninvertible systems are well-known in the literature; 
see e.g.\ \cite{ak01},\cite[Theorem 7.6.7]{he81},
\cite{BH04,P2012,PR16,P2018}. However, they all vary in one detail or 
another from the version with $\ell^1$-perturbations  below. For
example, preservation of the spectrum  follows under the weaker
condition of an $\ell^0$-perturbation \cite[Cor. 3.26]{P2012}, while
we need preservation of rates and  convergence of projectors 
at infinity in the resolvent set. For completeness we therefore
provide a proof in  Appendix \ref{app2b}. 

\begin{theorem}[Roughness of a  GED]\label{thm3:rough}
  Assume that the system \eqref{diffeq} has a GED on $J=\{n \in \Z:n \ge
  n_-\}$
with data $(K,\lambda_{\pm},P_n^{\pm})_{n\in J}$. Let
$E_n\in \R^{d,d}$, $n \in J$ be perturbations satisfying
\begin{equation} \label{eq3:condpert}
  \begin{aligned}
  q& \coloneqq K \lambda_{\star} \|E\|_{\ell^1(J)} < 1,  \text{ where}\\
  \|E\|_{\ell^1(J)}&= \sum_{n\in J}\|E_n\|, \quad
   \lambda_{\star}= \begin{cases} \lambda_+^{-1}, & 0< \lambda_+ <
     \lambda_- \le \infty, \\ 
          \lambda_-^{-1}, & 0=\lambda_+ < \lambda_- < \infty.
        \end{cases}
   \end{aligned}
  \end{equation}
Then the perturbed
equation \eqref{eq3:perturbsyst} has a GED on $J$ with data 
$(\tilde K,\lambda_{\pm},\tilde{P}_n^{\pm})_{n\in J}$, where
\begin{equation}\label{Qest}
 \sup_{n \in J} \|P_n^{\pm}-\tilde{P}_n^{\pm}\|  \le \frac{qK}{1-q}, \quad
  \tilde{K} = \frac{K}{1-q}.
\end{equation}
\end{theorem}
The following corollary holds for bounded rather than small
$\ell^1$-perturbations. 
\begin{corollary} \label{cor3:roughasymp}
    Let the system \eqref{diffeq} have a GED on $J$
    with data  $(K,\lambda_{\pm},P_n^{\pm})_{n\in J}$. Further assume
    that  the perturbation 
    satisfies $\|E_J\|_{\ell^1(J)}= \sum_{n\in J}\|E_n\|< \infty$, and let 
    $A_n+E_n$ be invertible for all $n \in \N_0$. Then there exist constants
    $\tilde{K}, C>0$ such that the perturbed system \eqref{eq3:perturbsyst}
    has a GED on $\N_0$ with data
    $(\tilde{K},\lambda_{\pm},\tilde{P}_n^{\pm})_{n \in J}$. 
    The perturbed projectors  satisfy
        \begin{equation} \label{eq3:estdiffproj}
          \sup_{n\in  J} \|P_n^{\pm}-\tilde{P}_n^{\pm}\| \le C \sum_{n
            \in J}\|E_n\|, \quad 
           P_n^{\pm}-\tilde{P}_n^{\pm} \to 0 \text{  as } n \to \infty.
        \end{equation}
   \end{corollary}
\begin{proof}
  Since $\|E_J\|_{\ell^1(J)}<\infty$ there exists $n_+\in J$  such
  that $q_+=K\lambda_{\star}\|E_{J_+}\|_{\ell^1(J_+)} <1$. 
  Then Theorem \ref{thm3:rough} applies to
  \eqref{eq3:perturbsyst} on $J_+=\{n\in \N:n  \ge n_+\}$ and 
  we obtain a  GED on $J_+$  with data
  $(K_+,\lambda_{\pm},\tilde{P}_n^{\pm})_{n \in J_+}$ which satisfies
  the estimate 
  \begin{align} \label{eq3:estJ+}
    \sup_{n \in J_+} \|P_n^{\pm}-\tilde{P}_n^{\pm}\| \le C_+ \sum_{n \in J_+}\|E_n\|, \quad
    C_+= \frac{K\lambda_{\star}}{1- q_+}.
  \end{align}
  Then we extend the GED  over a finite distance from $J_+$ to $J$ in
  a standard way by setting 
  $\tilde{P}_n^+=\tilde{\Phi}(n,n_+)\tilde{P}_{n_+}\tilde{\Phi}(n_+,n)$
  and by adapting the constant 
  $K_+$ to some $\tilde{K}$.
  This yields a GED on $J$ with data
  $(\tilde{K},\lambda_{\pm},\tilde{P}_n^{\pm})_{n \in J}$. 
  The estimate \eqref{eq3:estdiffproj} on $J$ follows from
  \eqref{eq3:estJ+} since the difference 
\begin{align*}
  \tilde{P}_n^+ -
  P_n^+=\tilde{\Phi}(n,n_+)\tilde{P}_{n_+}^+\tilde{\Phi}(n_+,n)-
  \Phi(n,n_+)P_{n^+}^+\Phi(n_+,n), 
  \quad n_- \le n < n_+ 
\end{align*}
can be bounded in terms of $\|E_n\|$, $n=n_-,\ldots,n_+$ and
$\|\tilde{P}_{n_+}^+-P_{n_+}^+\|$. 
 
   It remains to show
  $P_n^{\pm}-\tilde{P}_n^{\pm}\to 0$ as $n \to\infty$. For that purpose let
  us apply Theorem \ref{thm3:rough} again for arbitrary $N\ge n_+$ to
  $J_N=\{n\in J:n\ge N\}$ 
  with $q_N= K \lambda_{\star} \|E_{J_N}\|_{\ell_1(J_N)}\le  q_+ <1$.
  Then we we find another GED of \eqref{eq3:perturbsyst} with data
  $(K_N,\lambda_{\pm},P(N)_n^{\pm})_{n \in J_N}$ where the constant 
  $K_N=\frac{K}{1-q_N}\le \frac{K}{1-q_+}$ is uniformly bounded w.r.t.\ $N$
  and the estimate
  \begin{align*}
    \sup_{n \ge N}\|P_n^+ - P(N)_n^+\| \le \frac{K^2
    \lambda_{\star}}{1-q_+} \sum_{n  \ge N}\|E_n\| 
  \end{align*}
  holds. Now recall that \eqref{eq2:estproj2}  implies
    \begin{align*}
      \|\tilde{P}_n^+- P(N)_n^+\| \le \tilde{K}^2 \left(
      \frac{\lambda_+}{\lambda_-} \right)^{n-N} 
      \|\tilde{P}_N^+- P(N)_N^+\|, \quad n \ge N.
    \end{align*}
    By the triangle inequality we have for $n\ge N$
    \begin{align*}
      \|P_n^+ - \tilde{P}_n^+\| \le \frac{K^2 \lambda_{\star}}{1-q_+}
      \sum_{j \ge N}\|E_j\|+ 
      \tilde{K}^2 \left( \frac{\lambda_+}{\lambda_-} \right)^{n-N}
      \|\tilde{P}_N^+- P(N)_N^+\|.
    \end{align*}
    For a given $\varepsilon>0$ take $N=N_1$ so large that the first
    term is below $\frac{\varepsilon}{2}$ 
    and then $n \ge N_1$ so large that the second term is  below
    $\frac{\varepsilon}{2}$. This finishes the proof. 
    \end{proof}
 
\begin{theorem}[Roughness of a CED]\label{roughCED}
Assume that the difference equation \eqref{diffeq} has a CED  
with data $(K,\sigma_k,\cP_n^k)_{n\in J}^{k=1,\ldots,\varkappa}$ on a
one-sided interval $J=\{n \in \Z:n \ge n_-\}$.  Let 
  $(E_n)_{n\in J}$ be a perturbation such that $A_n+E_n$ is invertible for
  all $n\in J$ and such that $\| E\|_{\ell^1} < \infty$. Then there exist
  constants $\tilde{K}, C>0$ such that the perturbed system
  \eqref{eq3:perturbsyst} has a CED with data  
$(\tilde{K},\sigma_k, \tilde{\cP}_{n,k})_{n\in
  J}^{k=1,\ldots,\varkappa}$ and such that for all 
 $k=1,\ldots,\varkappa$ 
  \begin{equation} \label{eq3:CEDproj}
    \begin{aligned}
      \sup_{n \in J}\|\cP_{n,k}- \tilde{\cP}_{n,k}\| &\le C \sum_{\ell
        \in J} \|E_\ell\|, \quad 
      \lim_{n \to \infty}\cP_{n,k}- \tilde{\cP}_{n,k}=0.
    \end{aligned}
  \end{equation}
  \end{theorem}

\begin{proof}
  By Proposition \ref{lem3:CEDspec}(i), system \eqref{diffeq} has GEDs
  with data  
  $(K,\sigma_{k+1},\sigma_k,P_{n,k}^{\pm})_{n \in J}$ where $P_{n,k}^+
  =\sum_{\ell=k+1}^{\varkappa}\cP_{n,\ell}$, 
  $P_{n,k}^- =I-P_{n,k}^+=\sum_{\ell=1}^{k}\cP_n^{\ell}$ for
  $k=0,\ldots,\varkappa$  and $\sigma_0=0$, 
  $\sigma_{\varkappa +1} = \infty$. The Roughness Theorem
  \ref{thm3:rough} for GEDs and Corollary 
  \ref{cor3:roughasymp} ensure that the perturbed system
  \eqref{eq3:perturbsyst} has GEDs 
  with data
  $(\tilde{K},\sigma_{k+1},\sigma_k,\tilde{P}_{n,k}^{\pm})_{n \in J}$
  for $k=0,\ldots,\varkappa$. 
  Moreover, by \eqref{eq3:estdiffproj} we have 
  \begin{equation} \label{eq3:diffk}
    \sup_{n \in J}\| P_{n,k}^+-\tilde{P}_{n,k}^+\| \le \tilde{C}
    \sum_{n \in J} \| E_{\ell}\|, \quad 
      \lim_{n \to \infty}P_{n,k}^+-\tilde{P}_{n,k}^+=0, \quad
      k=1,\ldots,\varkappa. 
  \end{equation}
  Now we apply Proposition \ref{lem3:CEDspec} (ii) to the perturbed
  system and obtain that 
  \eqref{eq3:perturbsyst} has a CED on $J$ with data
  $(\tilde{K},\sigma_k,\tilde{\cP}_{n,k})_{n\in
    J}^{k=1,\ldots,\varkappa}$ 
  and fiber projectors given by \eqref{eq3:defcP}
  \begin{align*}
   \tilde{\cP}_{n,k}= \tilde{P}_{n,k}^-\tilde{P}_{n,k-1}^+ \cdots
    \tilde{P}_{n,0}^+,  \quad k=1, \ldots, \varkappa. 
  \end{align*}
  Next we observe that the unperturbed fiber projectors satisfy the
  same relations due to the orthogonality 
  conditions \eqref{eq3:prop1CED}: 
   \begin{align*}
    P_{n,k}^- P_{n,k-1}^+\cdots P_{n,0}^+&= \Big(\sum_{\ell=1}^{k}\cP_{n,\ell}\Big)
       \Big( \sum_{\ell =k}^{\varkappa} \cP_{n,\ell}\Big) P_{n,k-2}^+
                                           \cdots P_{n,0}^+ \\ 
      & = \cP_{n,k} \Big( \sum_{\ell =k-1}^{\varkappa}
        \cP_{n,\ell}\Big) \cdots P_{n,0}^+= \cdots =\cP_{n,k}, \quad
        k=1, \ldots,\varkappa. 
  \end{align*}
   Using \eqref{eq3:diffk} and  the boundedness of projectors,  a
   telescope sum then leads to 
   \begin{align*}
     \sup_{n \in J}\| \cP_{n,k} - \tilde{\cP}_{n,k}\| \le C \sum_{n
     \in J} \|E_{\ell}\|, \quad 
     \lim_{n  \to \infty}\cP_{n,k}- \tilde{\cP}_{n,k}=0 \text{ for }
     k=1,\ldots,\varkappa. 
     \end{align*}
\end{proof}

For the final step we relate CEDs and kinematic transformations:
\begin{theorem} \label{thm3:CED}
  Assume that the systems \eqref{diffeq} resp.\
  \eqref{eq3:perturbsyst} have CEDs with 
  data \\ $(K,\sigma_k,\cP_n^k)_{n\in J}^{k=1,\ldots,\varkappa}$ resp.\
  $(\tilde{K},\sigma_k,\tilde{\cP}_n^k)_{n\in J}^{k=1,\ldots,\varkappa}$
 and that the following properties hold:
  \begin{equation} \label{eq3:projasym}
   \lim_{n \to \infty} \cP_{n,k}- \tilde{\cP}_{n.k}=0, \quad
   k=1,\ldots,\varkappa, 
    \end{equation}
  \begin{equation} \label{eq3:perturbsmall}
    \|E_J\|_{\ell^1}=\sum_{n\in J} \| E_n\| < \infty.
  \end{equation}
  Then the systems  \eqref{diffeq} and \eqref{eq3:perturbsyst} are
  kinematically 
  similar with transformations $Q_n\in \mathrm{GL}(\R^d)$, $n \in J$
  satisfying $\lim_{n \to \infty}Q_n=I$. 
\end{theorem}
\begin{proof}
  For every fixed $N \in J$ we define the transformations
  \begin{equation} \label{eq3:transQ}
    Q^N_j = \tilde{\Phi}(j,N)\Big(
    \sum_{k=1}^{\varkappa}\tilde{\cP}_{N,k} \cP_{N,k} \Big) \Phi(N,j),
    \quad j \in J. 
  \end{equation}
  They satisfy the relations \eqref{relatePhi2}, hence form a
  kinematic transformation 
  \begin{equation} \label{eq3:funcrel}
    \begin{aligned}
      \tilde{\Phi}(n,m)Q_m^N& =
      \tilde{\Phi}(n,m)\tilde{\Phi}(m,N)\Big(
      \sum_{k=1}^{\varkappa}\tilde{\cP}_{N,k} \cP_{N,k} \Big)
      \Phi(N,m)  \\ 
      & = \tilde{\Phi}(n,N)\Big(
      \sum_{k=1}^{\varkappa}\tilde{\cP}_{N,k} \cP_{N,k} \Big)
      \Phi(N,n) \Phi(n,m)  = Q_n^N \Phi(n,m). 
    \end{aligned}
  \end{equation}
  Next we show that $Q_j^N$ converges as $N \to \infty$ 
  to some $Q_j$ uniformly in $j$.
  We verify the Cauchy property for indices $N \ge M$ using a telescope sum:
  \begin{align*}
    Q_j^N -Q_j^M & = \tilde{\Phi}(j,M) \Big[ \tilde{\Phi}(M,N)
      \sum_{k=1}^{\varkappa}\tilde{\cP}_{N,k} \cP_{N,k} \Phi(N,M) -
      \sum_{k=1}^{\varkappa}\tilde{\cP}_{M,k} \cP_{M,k} \Big] \Phi(M,j)\\
    &=  \sum_{k=1}^{\varkappa}\tilde{\Phi}(j,M) \tilde{\cP}_{M,k} \big[
      \tilde{\Phi}(M,N)\Phi(N,M)- I\big] \cP_{M,k} \Phi(M,j)\\
    & = \sum_{k=1}^{\varkappa}\tilde{\Phi}(j,M)\tilde{\cP}_{M,k} T(N,M)
    \cP_{M,k} \Phi(M,j), \quad {where} \\
    T(N,M)&= 
    \sum_{\ell=M+1}^N \tilde{\cP}_{M,k} \tilde{\Phi}(M,\ell)
    \big[\Phi(\ell,\ell-1)- \tilde{\Phi}(\ell,\ell-1)\big]
    \Phi(\ell-1,M)\cP_{M,k}\\
    & = \sum_{\ell=M+1}^N \tilde{\cP}_{M,k} \tilde{\Phi}(M,\ell)
    (-E_{\ell-1})
    \Phi(\ell-1,M)\cP_{M,k}.
  \end{align*}
  With the CED-estimates we arrive at
  \begin{align*}
    \|Q_j^N -Q_j^M\| & \le \sum_{k=1}^{\varkappa} \tilde{K}
                       \sigma_k^{j-M} \sum_{\ell=M+1}^N \tilde{K}
                       \sigma_k^{M-\ell} \|E_{\ell-1}\| K
                       \sigma_k^{\ell -1 -M} K \sigma_k^{M-j} \\ 
    & = \tilde{K}^2K^2 \sum_{k=1}^{\varkappa} \sigma_k^{-1}
      \sum_{\ell=M+1}^N \|E_{\ell-1}\|, 
  \end{align*}
  which by \eqref{eq3:perturbsmall} is below a  given $\varepsilon>0$
  uniformly in $j$ 
  for $M$ sufficiently large.
  Define $Q_j = \lim_{N \to \infty}Q_j^N$, $j \in J$. Taking the limit
  $N \to \infty$ in \eqref{eq3:funcrel} leads to
  \begin{align} \label{eq3:funcQrel}
    \tilde{\Phi}(n,m) Q_m = Q_n \Phi(n,m), \quad n,m \in J.
  \end{align}
  Further, we obtain from \eqref{eq3:transQ} the bound
  \begin{align*}
    \| Q_j^N\| \le \sum_{k=1}^{\varkappa} \tilde{K} \sigma_k^{j-N} K \sigma_k^{N-j}
    = \tilde{K} K \varkappa,
  \end{align*}
  so that $Q_j$ has the same bound. Next we show $Q_j \to I$ as $j\to \infty$.
  From \eqref{eq3:transQ} and \eqref{eq3:prop1CED},
  \eqref{eq3:prop3CED} we find 
  \begin{align*}
    \| Q_N^N - I\|& = \|\sum_{k=1}^{\varkappa}\tilde{\cP}_{N,k} \cP_{N,k} -
    \sum_{k=1}^{\varkappa} \cP_{N,k} \cP_{N,k}\| \le K
                    \sum_{k=1}^{\varkappa}\|\tilde{\cP}_{N,k}-\cP_{N,k}\|, 
  \end{align*}
  which converges to zero by \eqref{eq3:projasym}.
  Given $\varepsilon>0$ take $N(\varepsilon)$ such that for all
  $N \ge N(\varepsilon)$
  \begin{align*}
    \| Q_N^N - I\|\le \frac{\varepsilon}{2}, \quad \sup_{j \in J}
    \|Q_j^N - Q_j\| \le \frac{\varepsilon}{2}.
  \end{align*}
  Then we obtain for $N \ge N(\varepsilon)$
  \begin{align*}
    \|Q_N- I\|& \le \|Q_N - Q_N^N\| + \| Q_N^N - I\| \le
                \frac{\varepsilon}{2}+ 
    \frac{\varepsilon}{2} = \varepsilon.
  \end{align*}
  In particular, $Q_N$ is invertible and its inverse is bounded for
  $N \ge N_0$. By the relation \eqref{eq3:funcQrel} the invertibility
  and boundedness extends to finite indices  $N \le N_0$. Thus
  $(Q_j)_{j \in J}$ is a kinematic similarity transformation of the 
  systems \eqref{diffeq} and \eqref{eq3:perturbsyst} as claimed. 
\end{proof}

\subsection{Proof of Theorem \ref{thm3:maininv}}\label{roughproof}

\begin{proof}
We show that the system \eqref{diffeq}
  has pure point dichotomy spectrum
  $\Sigma_{\mathrm{ED}}= \{\sigma_1,\ldots, \sigma_{\varkappa}\}$ with
  associated 
  fiber bundle $\mathcal{W}_n^k= \range(\cP_n^k)$, $k=1,\ldots,\varkappa$.
  By Proposition \ref{lem3:CEDspec}(i) the intervals
  $(\sigma_{k+1},\sigma_k)$, $k=0,\ldots, \varkappa$ belong to 
  the resolvent set $R_{\ED}$ in \eqref{eq2:defres}. Hence  we have
  $\Sigma_{\ED} \subseteq \{\sigma_1,\ldots, \sigma_{\varkappa}\}$. 
  In fact, each value $\sigma_k$ belongs to the spectrum since the
  projectors $\cP_{n,k}$ are nontrivial and the rate 
  estimate \eqref{eq3:estboth} shows that $\sigma_k$ cannot lie in the
  interior of a GED interval. 

    By Theorem \ref{roughCED} we have proved the estimate \eqref{eq3:projest} and
    the statement about the unperturbed 
    dichotomy spectrum follows as above. 
    Further, Theorem \ref{roughCED} shows that the assumptions on the projectors $\cP_{n,k}$ in
    Theorem \ref{thm3:CED} are satisfied. Thus we conclude that
    the systems \eqref{diffeq} and \eqref{eq3:perturbsyst} are kinematically
    similar. Then Proposition \ref{prop3:inv} applies and yields the persistence of
    the outer angular spectrum. This finishes the proof of Theorem
    \ref{thm3:maininv}. 
\end{proof}
\subsection{Illustrative examples}

    The following simple example shows that the summability condition of the
    perturbations in Theorem \ref{thm3:CED} is rather sharp.
    \begin{example} \label{ex3:counter1}
      Consider for $a >0$ the scalar systems
      \begin{equation} \label{eq3:exsimple}
        u_{n+1}=a u_n, \quad v_{n+1}=a(1+\frac{1}{n})v_n.
      \end{equation}
      We claim that they are not kinematically similar with transformations $Q_n$,
      for which $Q_n$ and $Q_n^{-1}$ are uniformly bounded.
      Assuming the converse, there exists a subsequence $\N'\subseteq \N$ such
      that $\lim_{\N' \ni n \to \infty}Q_n \neq 0$. From the
      similarity property we infer 
      \begin{align*}
        Q_{n-1} &= a^{-1} Q_{n}\frac{an}{n-1}=\frac{n}{n-1}Q_{n}, \\
        |Q_1|&= |nQ_{n}| \to \infty \text{ as } \N' \ni n \to \infty,
      \end{align*}
      a contradiction.
     It is not difficult to
      extend this example to hyperbolic matrices of higher dimension.
    \end{example}

\begin{example} \label{ex3:henon3}
We consider a three-dimensional variant of H\'enon's map
\begin{equation}\label{Hmap}
F:\begin{array}{rcl} \R^3&\to & \R^3\\
x & \mapsto &
\begin{pmatrix}
-x_1^2 - \frac 9{10}x_3 + \frac 75\\
x_1 \cos(\omega) - x_2 \sin(\omega)\\
x_2 \cos(\omega) + x_1 \sin(\omega)
\end{pmatrix}
\end{array}
\quad \text{with}\quad \omega = 0.2.
\end{equation}
The map $F$ has the fixed point $\xi \approx \begin{pmatrix} 0.5674 & 
0.4639 & 0.5674\end{pmatrix}^\top$ and the Jacobian $DF(\xi)$
possesses the unstable eigenvalue $-1.4736$ and the stable pair 
$0.0701 \pm 0.7784i$. 

Of particular interest are homoclinic orbits $(\bar x_n)_{n\in\Z}$
w.r.t.\ the fixed point $\xi$,
i.e.\ $\bar x_{n+1} = F(\bar x_n)$, $n\in\Z$, $\lim_{n\to \pm \infty}
\bar x_n = \xi$. 
On the finite interval $[-10^3,10^3]\cap \Z$ we compute a numerical
approximation by solving a periodic boundary value problem. The center part
$(\bar x_n)_{n\in[-20,20]\cap \Z}$ is shown in the left panel of
Figure \ref{Henon}. In addition, the right panel shows approximations
of the one-dimensional unstable manifold (red) and of the two-dimensional
stable manifold that is computed using the contour algorithm from
\cite{h16}. 

For the half-sided variational equation 
\begin{equation}\label{variational}
u_{n+1} = DF(\bar x_n)u_n,\quad n \in \N
\end{equation}
we aim for its angular spectrum. Exploiting the results from above, we
start with the autonomous system
\begin{equation}\label{var_auto}
u_{n+1} = DF(\xi)u_n,\quad n\in \N
\end{equation}
and analyze its first angular spectrum.
Trace spaces of \eqref{var_auto} are the unstable eigenspace and each
one-dimensional subspace of the two-dimensional stable
eigenspace. The unstable eigenspace results in the
spectral value $0$. For getting the whole spectrum, we separate the
dynamics within the two-dimensional stable eigenspace, using a
reordered Schur decomposition, see \cite[Algorithm 6.1]{BeFrHu20} and
apply 
Example \ref{ex:normal2}. It turns out that $\mathrm{sk}(\rho,\varphi)
\approx 1.062 > 1$ and $\frac \varphi \pi \notin \Q$, thus the
explicit representation \eqref{finalform} gives the first outer
angular spectrum 
\begin{equation}\label{Henspec1}
\Sigma_1 = \{0, 1.33566342\}.
\end{equation} 

For the second angular spectrum we can employ Example
\ref{ex6:mixed} since $DF(\xi)$ has a stable complex eigenvalue
and a real unstable eigenvalue. We transform $DF(\xi)$ into
\eqref{eq6:auto3} with $A = \left(\begin{smallmatrix}
-0.08965 &  -1.42890 &  -0.70779\\
0.69421& -0.08965 &  -0.33648\\
0&0&1.88559\end{smallmatrix}\right)$ using a reordered Schur
decomposition, scaling and an orthogonal similarity
tranformation. The second outer angular spectrum is then computed from
formula \eqref{eq6:Sigma2} via numerical integration
\begin{equation}\label{Henspec2}
\Sigma_2 = \{0, 1.32818438\}.
\end{equation} 

In the following we show that the angular spectra of
\eqref{variational} and 
of \eqref{var_auto} coincide. By Example \ref{ex3:semisimple} we observe
that \eqref{var_auto} has a CED on $\N$. Furthermore, $DF(\bar x_n)$
is invertible for all $n\in\N$ and the perturbation 
$E_\N= (DF(\bar x_n)-DF(\xi))_{n\in\N}$ satisfies the summability
condition \eqref{eq3:L1norm} due to the exponentially fast convergence
of the orbit $(\bar x_n)_{n\in\N}$ towards the fixed point $\xi$. 
Thus, Theorem \ref{thm3:maininv} yields the claim.

Note that in the numerical experiments in Section \ref{Sec_num}, we
apply our algorithm to both systems \eqref{var_auto} and
\eqref{variational} and compare the results. 

\begin{figure}[hbt]
\begin{center}
\includegraphics[width=0.95\textwidth]{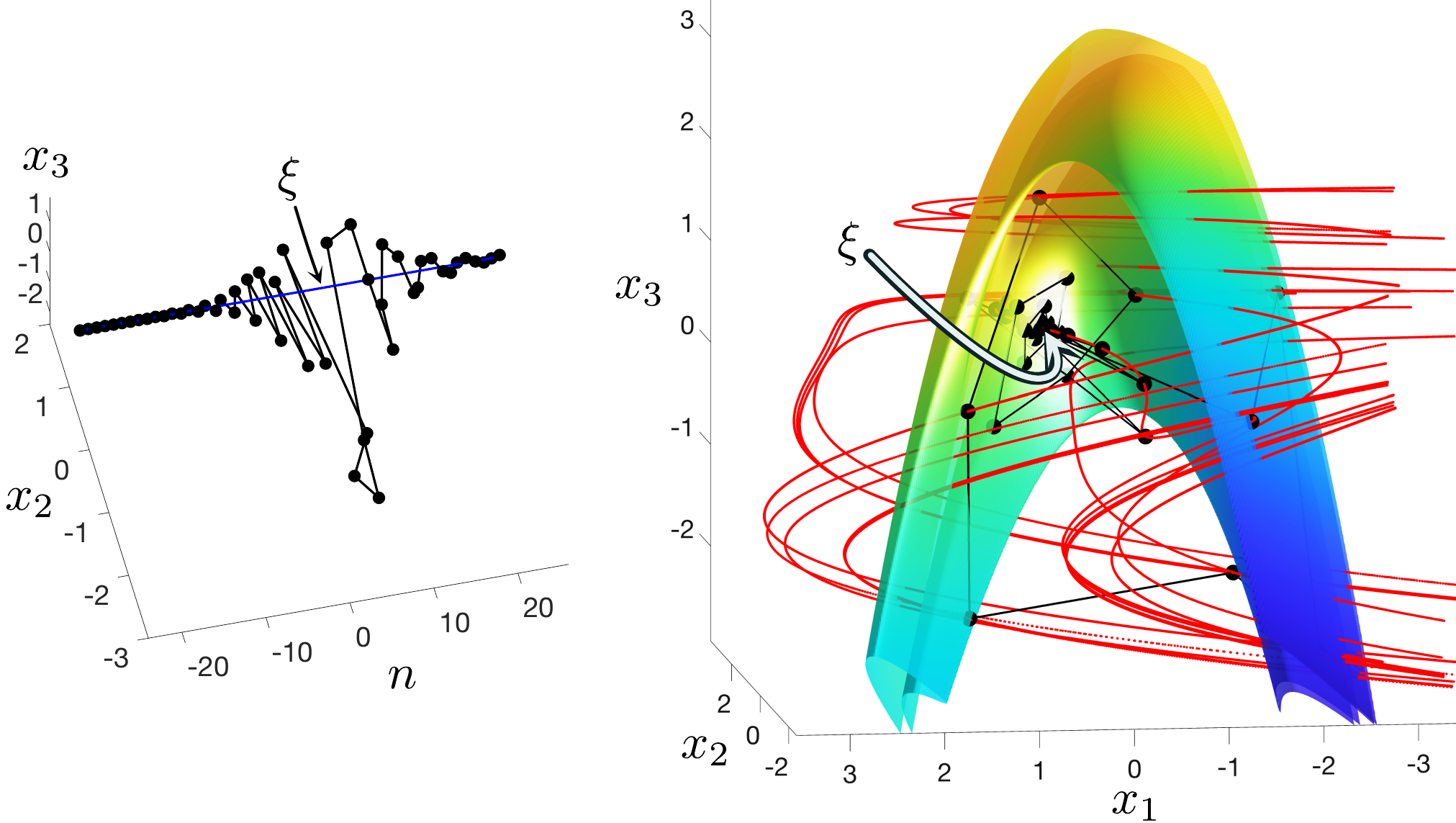}   
\end{center}
\caption{\label{Henon}Center part of a homoclinic orbit of
  \eqref{Hmap} (left). The same orbit in phase space (right) with
  approximations of unstable (red) and stable manifolds of $\xi$.}  
\end{figure}
\end{example}
    
\subsection{Almost periodicity and outer angular spectrum}
The property of almost periodicity is a well-known concept to analyze
the long time behavior 
of sequences which converge to some periodic orbit or to some ergodic
motion on a circle; see 
\cite[Ch.4.1]{P89} for the general theory. In \cite{BeFrHu20} we have
used this property to establish the existence 
and the equality of various angular values from Definition
\ref{defangularvalues}. In the following we study the sequence of maps 
\begin{equation} \label{eq3:bndef}
  b_n \colon \cG(s,d) \to \R, \quad b_n(V) = \ang(\Phi(n-1,0)V,\Phi(n,0)V),
\end{equation}
from which the averages $\alpha_n= \frac{1}{n} \sum_{j=1}^n b_j(V)$ in
\eqref{def:alpha} are derived. 
For example, we have shown in \cite[Lemma 3.6, Proposition
3.7]{BeFrHu20} that uniform almost periodicity (Uap) 
  of the sequence $b_n$ implies the Cauchy property for $\alpha_n$
  uniformly in $V$. The Uap notion is 
  slightly weaker than in \cite[Ch.4.1]{P89}  
 and is further weakened in the following definition.

\begin{definition}\label{auap2}
Given a set $\cV$ and a metric space $(\cW,d)$. A sequence of mappings
$b_n: \cV\to \cW$, $n \in \N$ is called 
\begin{itemize}
\item[(i)] \textbf{asymptotically uniformly almost periodic} (AUap) if
\begin{equation}\label{asymp}
\begin{aligned}
&\forall \eps >0\ \exists P\in\N\ \exists L\in\N: \forall V \in \cV\
                 \forall \ell\in\N\ \exists
p\in\{\ell,\dots,\ell+P\}: \\
&\forall n \ge L: d(b_n(V),b_{n+p}(V))\le \eps.
\end{aligned}
\end{equation}
\item[(ii)] \textbf{uniformly almost periodic} (Uap) if it is AUap
  with $L = 1$. 
\end{itemize}
\end{definition}
\begin{remark} Recall from \cite[Ch.4.1]{P89} the standard definition
  of uniform almost periodicity for $b_n$: for every $\varepsilon>0$
  there exists a relatively dense 
  set $\cP\subset \N$ such that $|b_{n}(V)-b_{n+p}(V)| \le
  \varepsilon$ for all $n\in \N, p\in \cP,V \in \cV$; the set $\cP$ is
  relatively dense 
  iff  there exists $P>0$ such that
  $\cP \cap \{\ell,\ldots,\ell+P\}\neq \emptyset$ for all $\ell \in
  \N$. Recall from \cite[Remark 3.5]{BeFrHu20} that Uap 
  is weaker since we allow $p \in \{\ell,\ldots,\ell +P\}$ to depend
  on $V \in \cV$,  and note that 
  AUap is still weaker since  the estimate holds for $n \ge L$ only.
   \end{remark}
Definition \ref{auap2} applies to dynamical systems of the form
\eqref{diffeq} with the setting $\cV =\cW = \cG(s,d)$,  $b_n(V) =
\Phi(n,0)V$ 
and the  metric $\ang(\cdot,\cdot)$, but likewise to $\cV=\cG(s,d)$,
$W=\R$ and 
$b_n(V)= \ang(\Phi(n,0)V,\Phi(n-1,0)V)$ for  $n \ge 1$. 

First note that (A)Uap carries over from subspaces to 
angles of successive subspaces. 
\begin{lemma}\label{spacetoang}
  If $\varphi_n:\cG(s,d) \to \cG(s,d)$, $\varphi_n(V)= \Phi(n,0)V$ ($n
  \in \N_0$) is (A)Uap in the metric space 
  $(\cG(s,d),\ang(\cdot,\cdot))$, then the sequence of mappings 
$$
b_n:\cG(s,d) \to \R, \quad b_n(V)= \ang(\Phi(n-1,0)V,\Phi(n,0)V),\quad
n\in\N 
$$
is (A)Uap.
\end{lemma}
\begin{proof}
Given $\eps >0$,  choose $P$ and $L$ as in \eqref{asymp} w.r.t.\
$\frac \eps 2$.  
For $n \ge L$ then use the  triangular inequality for the angle:
\begin{align*}
& \big|\ang(\Phi(n-1,0)V,\Phi(n,0)V) -
                 \ang(\Phi(n+p-1,0)V,\Phi(n+p,0)V)\big|\\ 
&\le \ang(\Phi(n-1,0)V,\Phi(n+p-1,0)V)+\ang(\Phi(n,0)V,\Phi(n+p,0)V)\\
&\le \tfrac \eps 2 + \tfrac \eps 2 = \eps,
\end{align*}
and the proof is complete. The case Uap follows by setting $L=1$. 
\end{proof}
Next we observe that the  (A)Uap-property is invariant under kinematic
transformations. 
\begin{proposition}\label{invariant}
Assume that the system \eqref{diffeq} is AUap resp.\ Uap.\\
Then the kinematically equivalent system \eqref{difftransform2} is
also AUap resp.\ Uap provided 
the transformations $Q_n,Q_n^{-1}$ are uniformly bounded.
\end{proposition} 
\begin{proof}Let $\|Q_m\| \|Q_n^{-1}\| \le \kappa$ for $n,m \in \N$
  and recall the constants $C=\pi \kappa (1+\kappa)$ from
  \eqref{app:est1} 
  and  $C_{\pi}$ from \eqref{app:est2}. Given $\eps>0$,  
  choose $P$ and $L$ such that \eqref{asymp} holds with
  $\eps'=\frac{\eps}{2CC_{\pi}(1+\kappa)}$ 
  for $n-1 \ge L$.  
    By \eqref{relatePhi2} the quantities $\tilde{b}_n$ which belong to
    $\tilde{\Phi}$ satisfy 
  \begin{align*}
     \tilde{b}_n(Q_0V) = \ang(\tilde{\Phi}(n-1,0)Q_0V,\tilde{\Phi}(n,0)Q_0V)=
     \ang(Q_{n-1}\Phi(n-1,0)V,Q_n\Phi(n,0)V).
     \end{align*}
  With Lemmas \ref{app:Lest1}, \ref{app:Lest2} we then obtain  from
  \eqref{asymp} for all $V \in \cG(s,d)$, $\ell\in \N$, 
  $n \ge L$ and some $p \in \{\ell,\ldots,\ell+P\}$
  \begin{align*}
    & |\tilde{b}_n(Q_0V)-\tilde{b}_{n+p}(Q_0V)|\\ 
  &  = | \ang(Q_{n-1}\Phi(n-1,0)V,Q_n\Phi(n,0)V)
     -  \ang(Q_{n+p-1}\Phi(n+p-1,0)V,Q_{n+p}\Phi(n+p,0)V)|\\
    & \le  \ang(Q_{n-1}\Phi(n-1,0)V, Q_{n+p-1}\Phi(n+p-1,0)V ) 
      +\ang(Q_n \Phi(n,0)V,Q_{n+p}\Phi(n+p,0)V) \\
     & \le C\left[ \ang(Q_{n+p-1}^{-1}Q_{n-1}\Phi(n-1,0)V,\Phi(n+p-1,0)V) 
        + \|Q_{n+p}^{-1}Q_n \Phi(n,0)V,\Phi(n+p,0)V)\right]\\
      & \le C C_{\pi}\left[\|Q_{n+p-1}^{-1}Q_{n-1}-I\|
        \ang(\Phi(n-1,0)V,\Phi(n+p-1,0)V) \right.\\ 
        &  \left.   \qquad  \quad    +\|Q_{n+p}^{-1}Q_n-I \|
          \ang(\Phi(n,0)V, \Phi(n+p-1,0)V) \right]\\ 
      & \le C C_{\pi}(1+\kappa)2 \eps' = \eps.
  \end{align*}
\end{proof}
Note that this property allows us to obtain the AUap property for the
$\ell^1$-perturbation of a 
system which has the AUap and the CED property; see Theorems
\ref{roughCED} and \ref{thm3:CED}. 

The following Lemma \ref{estap}  shows that the partial sums formed
from an AUap sequence  
$b_n:\cV \to \cW$, $n\in\N$ have the uniform Cauchy property.  This
extends \cite[Lemma 3.6]{BeFrHu20} 
where Uap was assumed. The property will be useful for deriving
convergence of the 
numerical methods in Section \ref{Sec_num}.  
\begin{lemma}\label{estap}
Let $b_n: \cV\to \cW$, $n \in \N$ be a sequence of AUap
and uniformly bounded functions.
Then for all $\eps >0$ there exists $N\in\N$ such that for all $n\ge
m\ge N$, 
$ k\in\N$, $V\in\cV$
\begin{equation*}
 \Big\|\frac 1n \sum_{j=1}^n b_j(V) - \frac 1m \sum_{j=1}^m
b_{j+k}(V)\Big\| \le \eps.
\end{equation*}
\end{lemma}

\begin{proof}
Let $\eps >0$. By AUap there exists a
$P\in\N$ and $L \in\N$ such that for every $V\in\cV$ and
each $k\in\N_0$ we find a $p_k\in\{k,\dots,P+k\}$ (which may depend on
$V$) with 
\begin{equation}\label{apbasic}
\|b_n(V)-b_{n+p_k}(V)\| \le \frac \eps 8 \quad \forall n\ge L.
\end{equation}
Let $b_{\infty}= 2\sup_{n,V}\|b_n(V)\|$ and $M = \lceil \frac {8}\eps
b_\infty (L+P) \rceil$. It follows for each $k\in\N_0$ that
\begin{equation}\label{finally}
\begin{aligned} 
\Big\| \sum_{j=1}^M b_j(V) - \sum_{j=1}^M b_{j+k}(V)\Big\| 
& \le \sum_{j=1}^M \|b_j(V)-b_{j+p_k}(V)\| + \Big\|\sum_{j=1}^M
      b_{j+p_k}(V)-\sum_{j=1}^M b_{j+k}(V)\Big\| \\
& \le \sum_{j=1}^{L-1} \|b_j(V)-b_{j+p_k}(V)\| 
  +\sum_{j=L}^M\|b_j(V)-b_{j+p_k}(V)\|\\
&\phantom{\le}   + \Big\|\sum_{j=1}^M
      b_{j+p_k}(V)-\sum_{j=1}^M b_{j+k}(V)\Big\| \\
&\le L b_\infty + M \frac \eps 8 + P b_\infty \le  M \frac \eps 4.
\end{aligned}
\end{equation}
Let $N= \lceil \frac 4\eps M b_\infty\rceil$ and decompose $n\ge m\ge N$ 
modulo $M$, i.e.\
\begin{equation*}
m = \ell_m M + r_m,\ 0\le r_m<M,\quad  n = \ell_n M + r_n,\ 0\le r_n<M.
\end{equation*}
For $c(V)\coloneqq \sum_{j=1}^M b_j(V)$ we obtain from \eqref{finally}
for each $k\in \N_0$ 
the estimates
\begin{align*}
\Big\|\sum_{j=1}^{M \ell_n} b_j(V) - \ell_n c(V)\Big\|
 &  \le \sum_{i=1}^{\ell_n} \Big\| \sum_{j=1}^M b_{j+(i-1)M}(V) -
     c(V)\Big\|
\le  \ell_n M \frac \eps 4,\\
\Big\| \frac{\ell_n}{n} c(V) - \frac{\ell_m}{m} c(V)\Big\|
&= \Big|\frac{\ell_n r_m - \ell_m r_n}{nm}\Big| \|c(V)\|\le
\frac {\ell_n M}{nm} \|c(V)\|
\le \frac 1m \|c(V)\| \le \frac{M}{N} b_{\infty}  \le \frac \eps 4.
\end{align*}
 Combining these estimates, we find for every $k\in\N_0$
\begin{align*}
&\Big\|\frac 1n \sum_{j=1}^nb_j(V) - \frac 1m \sum_{j=1}^m
                 b_{j+k}(V)\Big\|\\
& \le \Big\|\frac 1n \sum_{j=1}^{M\ell_n} b_j(V) - \frac 1m
   \sum_{j=1}^{M\ell_m}b_{j+k}(V)\Big\|
+ \Big\|\frac 1n \sum_{j=M\ell_n+1}^{M\ell_n+r_n} b_j(V) - \frac 1m
   \sum_{j=M\ell_m+1}^{M\ell_m+r_m}b_{j+k}(V)\Big\|\\
&\le \Big\|\frac 1n \sum_{j=1}^{M\ell_n} b_j(V) - \frac{\ell_n}{n}
                                                       c(V)\Big\|
+ \Big\|\frac{\ell_n}{n} c(V) - \frac{\ell_m}{m} c(V)\Big\|\\
&+ \Big\| \frac{\ell_m}{m} c(V) -  \frac 1m
   \sum_{j=1}^{M\ell_m}b_{j+k}(V)\Big\|
+ \Big\|\frac 1n \sum_{j=M\ell_n+1}^{M\ell_n+r_n} b_j(V) - \frac 1m
   \sum_{j=M\ell_m+1}^{M\ell_m+r_m}b_{j+k}(V)\Big\|\\
&\le \frac{\ell_n M}{n} \frac \eps 4 + \frac \eps 4 + \frac{\ell_m M}{m}
\frac \eps 4 + \frac{M}{N} b_\infty \le \eps.
\end{align*}
\end{proof}

\begin{example} \label{ex3:AUapauto} (Example \ref{ex3:semisimple}
  revisited) \\ 
  Consider the autonomous case $A_n \equiv A$ where $A$ is invertible
  and (complex) diagonalizable. 
  Let  $\lambda_j =\sigma_j e^{i \varphi_j}$, $\sigma_j>0$, $\varphi_j
  \in (0,\pi)$, $j=1,\ldots,k$ be the complex eigenvalues and let
  $\lambda_{j+k}\neq 0 $,  
  $j=1,\ldots \ell$ be the real ones (hence $d=2k+\ell$). We claim
  that the sequence 
  \begin{align*}
    b_n: \cD(s,d) \to \R, \quad b_n(V)= \ang(A^{n-1}V,A^nV), \; n \in \N
  \end{align*}
  is Uap if the angles $\varphi_j$, $j=1,\ldots,k$ and $\pi$ are
  rationally independent, i.e.\  $\sum_{j=1}^k n_j \varphi_j= 
  n_0\pi$
  for some $n_j\in \Z$, $ j=0,\ldots,k$ implies $n_0= n_1= \cdots=
  n_k=0$. Recall the  Definition \ref{deftrace} of the trace space 
  and note that $\cD_n(s,d)\equiv \cD(s,d)$ for all $n \in \N$ in the
  autonomous case. By  assumption 
 there exists a decomposition $\R^d= \bigoplus_{j=1}^{k+\ell}W_j$ into
 invariant subspaces of $A$ where 
   $W_j=\range(Q_j)$,  $j=1,\ldots,k+\ell$ for 
    matrices of full rank $Q_j\in \R^{d,2},j=1,\ldots,k$ and
    $Q_{j+k}\in \R^{d,1},j=1,\ldots,\ell$. Further, we have 
    \begin{align} \label{eq3:evrel}
      AQ_j=\sigma_j  Q_jT_{\varphi_j},\ j=1,\ldots,k, \quad
      AQ_{j+k}=\lambda_{j+k}Q_{j+k},\ j=1,\ldots,\ell. 
    \end{align}
    For every trace space  $V\in \cD(s,d)$ there exist subspaces $V_j$
    of $W_j$, $j=1,\ldots,k+\ell$ (possibly trivial $V_j=\{0\}$) such
    that 
    \begin{align*}
      V= \bigoplus_{j=1}^{k+\ell} V_j, \quad \sum_{j=1}^{k+\ell} \dim(V_j)=s.
    \end{align*}
    The index set $J=\{j\in \{1,\ldots,k\}: \dim(V_j)=1\}$ singles out
    the one-dimensional parts of $V$ in two-dimensional spaces  
    so that $V_j = \mathrm{span}(v_j)$ for some $v_j=Q_ju_j$,
    $\|u_j\|=1$, $j\in J$.  
    For $j \notin J$ either $V_j=W_j$ or $V_j=\{0\}$ holds and,
    therefore, $AV_j=V_j$ follows from \eqref{eq3:evrel}. 
    On the other hand, equation \eqref{eq3:evrel} implies 
    \begin{align*} AV_j = \mathrm{span}(Av_j)= \mathrm{span}(Q_j
      T_{\varphi_j}u_j), \quad \text{ for } j \in J. 
    \end{align*}
    Next, consider the $|J|$-dimensional torus $\T^J=
    \{u_J=(u_j)_{j\in J}: u_j \in \R^2, \|u_j\|=1, j\in J \}$ and  the 
    maps
    \begin{align*}
      F_J&: \T^J \to \T^J,\quad F_J(u_J)= (T_{\varphi_j}u_j)_{j \in
           J}, \\ 
      g&: \T^J \to \R,\qquad \, g(u_J)=\ang\Big( \bigoplus_{j\in
         J}\mathrm{span}(Q_j u_j) \oplus \bigoplus_{j \notin J}V_j, 
     \bigoplus_{j \in J} \mathrm{span}(Q_jT_{\varphi_j} u_j) \oplus
         \bigoplus_{j \notin J}V_j \Big). 
          \end{align*}
    With these settings we find $g(u_J)= \ang(V,AV)$,
    $g(F_J(u_J))=\ang(AV,A^2V)$, and thus 
    \begin{align*}
     b_n(V)= g(F_J^{n-1}(u_J)), \quad  \alpha_n(V) = \frac{1}{n}
      \sum_{j=1}^n b_j(V) = \frac{1}{n} \sum_{j=1}^n  g(F_J^{j-1}
      (u_J)). 
    \end{align*}
    Due to the rational independence of $\pi$ and $\varphi_j$,
    $j=1,\ldots,k$ the map $F_J$ is ergodic 
    w.r.t.\ Lebesgue measure; see \cite[Prop.1.4.1]{KH95}. Moreover,
    $F_J$ is an isometry w.r.t.\ the metric  
    $d(u_J,v_J)=\max_{j \in J}\ang(u_j,v_j)$  on $\T^J$. The result in
    \cite[Ch. 4, Remark 1.3]{P89} 
    then shows that the map $F_J$ is uniformly almost periodic i.e.\
    for every $\varepsilon >0$ there exists a relatively 
    dense set $\cP\subset \N_0$ with $d(u_J,F_J^p(u_J))\le
    \varepsilon$ for all $u_J \in \T^J$, $p \in \cP$.  
    We prove that $b_n$ is then  Uap on the set $\cD_J(s,d)= \{V\in
    \cD(s,d): \dim(V_j)=1 \text{ for } j \in  J\}$ 
    as follows (cf.\ the proof of \cite[Proposition 5.2]{BeFrHu20}):
    by the uniform continuity of $g$ on $\T^J$ 
    there exists for every $\varepsilon'>0$ some
    $\varepsilon>0$ such that $|g(u_J)-g(v_J)|\le \varepsilon'$
    whenever $d(u_J,v_J) \le \varepsilon$, $u_J,v_J\in \T^J$. 
    For the relatively dense set $\cP$ which belongs to $\varepsilon$
    we then obtain 
    \begin{align*}
      |b_n(V)-b_{n+p}(V)|=
      |g(F_J^{n-1}(u_J))-g(F_J^p(F_J^{n-1}(u_J)))| \le \varepsilon'
      \quad \forall n\in \N, p\in \cP,u_J 
      \in \T^J.
    \end{align*}
    Thus $b_n$ is Uap on $\cD_J(s,d)$ and also on $\cD(s,d)=
    \bigcup_{J} \cD_J(s,d)$, since there are only 
    finitely many index sets $J \subseteq \{1,\ldots,k\}$. 
\end{example} 
The following example shows that the AUap property does not hold for
matrices $A$ which have generalized eigenvectors. 
\begin{example} \label{ex6:Jordancounter}
  For the Jordan matrix $A= \left(\begin{smallmatrix} 1 & 1 \\ 0 &
      1 \end{smallmatrix}\right)$ 
    the sequence $b_n(V) = \ang(A^{n-1}V,A^{n}V)$, $V \in \cG(1,2)$ is
    not AUap. Suppose the contrary 
  and let $P,L$, $p=p(\ell,V)\in \{\ell,\ldots,\ell+P\}$ be the data
  from  \eqref{asymp} with $\varepsilon=\frac{\pi}{8}$. 
  For the vectors $v_n= \begin{pmatrix} -n +1 &
    1 \end{pmatrix}^{\top}$, $n \ge L$ we 
  obtain
  \begin{align*}
    A^{n-1}v_n= \begin{pmatrix} 0 \\ 1 \end{pmatrix},\quad A^n v_n
    = \begin{pmatrix} 1 \\ 1 \end{pmatrix},\quad 
    A^{n+p-1}v_n= \begin{pmatrix} p \\ 1 \end{pmatrix},\quad A^{n+p}
    v_n = \begin{pmatrix} p+1 \\ 1 \end{pmatrix}. 
  \end{align*}
  Then we have $b_n(v_n)=\ang(A^{n-1}v_n,A^nv_n)= \frac{\pi}{4}$ and
  \begin{align*}
    b_{n+p(\ell,V_n)}(v_n) = \ang\left( \begin{pmatrix} p(\ell,v_n) \\
        1 \end{pmatrix},\begin{pmatrix} p(\ell,v_n)+1 \\
        1 \end{pmatrix} 
    \right) \rightarrow 0 \text{ as } \ell \to \infty
  \end{align*}
  since $p(\ell,v_n) \to \infty$ as $\ell \to \infty$.
  Thus $|b_n(v_n)-b_{n+p(\ell,v_n)}(v_n)| \le \frac{\pi}{8}$ is
  violated for $\ell$ sufficiently large. 
\end{example} 
Contrary to the (A)Uap-property, the uniform Cauchy \eqref{uC}
property is generally not 
invariant w.r.t.\ an autonomous similarity transformation as the
following Example \ref{gegen} 
shows. 
\begin{example}\label{gegen}
Let 
$$
D=\begin{pmatrix}
J & 0\\ 0 & J
\end{pmatrix},\quad J=\begin{pmatrix} 0 & -1 \\ 1 & 0 \end{pmatrix},
\quad 
X = \begin{pmatrix}
0 & I_2\\ I_2 & 0
\end{pmatrix}
$$
and let $A_n$ be defined in Table \ref{Adef}. Denote by $\Phi$ the
corresponding solution operator. 

 \begin{table}[hbt]
  \begin{center}
  \begin{tabular}{c|cccccccccccccccccccccc}
    $n$ & 0 & 1 & 2 & 3 & 4 & 5 & 6 & 7 & 8 & \dots & 15 & 16 & 17 & \dots
    & 32 & 33 \\\hline

   $A_n$ & $D$ & $D$ & $X$ & $D$ & $D$ & $D$ & $D$ & $X$ & $D$ & $\dots$ &
   $D$ & $X$ & $D$ & \dots & $D$ & $X$
  \end{tabular}
  \normalsize
\caption{Construction of $(A_n)_{n\in\N_0}$.\label{Adef}}
\end{center}
\end{table}
We analyze the case $s=1$. 
Since $\ang(v,Dv) = \frac \pi 2$ for each $0\neq v \in \R^4$ we
observe, that the system $u_{n+1} = A_n u_n$, $n\in\N_0$ satisfies the
uniformly Cauchy condition \eqref{uC}. 

Let 
$$
B_n = S A_n S^{-1}
\quad \text{with}\quad
S = 
\begin{pmatrix}
1 & 0 & 0 & 0\\
0 & 1 & 0 & 0\\
0 & 0 & 1 & 1\\
0 & 0 & 0 & 1
\end{pmatrix}.
$$
\end{example}
Denote by $e_i$ the $i$-th unit vector. 
We get
$$
\ang(Sv,SDv) = 
\begin{cases}
\tfrac \pi 2,&\text{for } v \in \{e_1,e_2\},\\
\tfrac \pi 4,&\text{for } v \in \{e_3,e_4\}.
\end{cases}
$$
For $v=e_1$, the sequence $\frac 1n \sum_{j=1}^n
\ang(S\Phi(j-1,0)v,S\Phi(j,0)v)$, $n\in\N$ does not converge and thus, the
transformed system does not satisfy \eqref{uC}.
 
\section{Numerical approximation}\label{Sec_num}
In \cite[Section 4.1]{BeHu22} we developed an algorithm for the
approximation of outer angular values. In Section \ref{algo} we extend
these ideas and 
obtain an algorithm that approximates the outer angular spectrum
$\Sigma_s$. In Section \ref{sec4:FT} we introduce finite time outer
angular spectra and investigate their lower and upper semi-continuity.

\subsection{Convergence of finite time angular spectra}
\label{sec4:FT}
For the numerical calculation of $\Sigma_s$, we introduce the
\textbf{finite time outer angular $N$-spectrum}
\[
\Sigma_{s}^{N} \coloneqq \{\alpha_N(V): V \in \cD_0(s,d)\}.
\]
This spectrum is numerically accessible by computing the dichotomy
spectrum and its spectral bundles first, and then by solving
optimization problems. 
For the simple Example \ref{ex2:revisit}, finite time and infinite outer
angular spectra coincide, i.e.\ $\Sigma_s = \Sigma_{s}^{N}$
for $s\in\{1,2\}$  and all $N \in \N$. 

We observe the following (modest) connection between these spectra.
\begin{lemma}\label{modest} For every $\varepsilon>0$ and $\theta \in
  \Sigma_s$ there exists some $N \in \N$ sucht that 
\[
\dist(\theta,\Sigma_{s}^{N}) \le \varepsilon.
\]
\end{lemma}

\begin{proof}
Let $\eps >0$ and $\theta\in\Sigma_s$. Then, there exists a
$V\in\cD_0(s,d)$ such that 
\[
\theta \in [\varliminf_{n\to\infty}
\alpha_n(V)-\tfrac \eps 2, \varlimsup_{n\to\infty}\alpha_n(V)+\tfrac \eps
2].
\]
By Lemma \ref{accu}, each angle in 
$[\varliminf_{n\to \infty} \alpha_n(V), \varlimsup_{n\to\infty}\alpha_n(V)]$
is an accumulation point of $(\alpha_n(V))_{n\in\N}$.
Thus, we find an $N\in\N$ such that 
$|\alpha_N(V) - \theta| \le \eps$.
\end{proof}

However, finite and infinite spectra do not conicide in general. 
We revisit the key examples from \cite[Section 3.2]{BeFrHu20}. 
These examples  show that one can generally not
expect a sharper result than Lemma \ref{modest}. 
The well known notions of upper semi-continuity and lower
semi-continuity do not hold  true. 
\begin{example}\label{E1}
For given $0 \le \varphi_0 <\varphi_1 \le \frac \pi 2$ we define
\begin{equation}\label{Ex1}
A_n = 
\begin{cases}
T_{\varphi_0},&
\text{ for } n = 0 \lor n \in \bigcup_{\ell=1}^\infty
[2^{2\ell-1},2^{2\ell} -1]\cap \N,\\[3mm]
T_{\varphi_1},&\text{ otherwise.}
\end{cases}
\end{equation}
Following the computation of the extremal values in \cite[Example
3.10]{BeFrHu20} we obtain the outer angular spectra 
\[
\Sigma_1 = \big[\tfrac 23 \varphi_0 + \tfrac 13 \varphi_1, \tfrac
13 \varphi_0 + \tfrac 23 \varphi_1\big],\quad
\Sigma_{1}^{N} = \{\alpha_N(V): V\in\cG(1,2)\}.
\]
Note that the dichotomy
spectrum $\{1\}$  provides no spectral separation, i.e.\ $\cW_0^1 =
\R^2$. Thus we find 
$\cD_0(1,2) = \cG(1,2)$. 
All one-dimensional
subspaces lead to the same angular value w.r.t.\ \eqref{Ex1}. 
For every  fixed $V\in\cG(1,2)$ we get
\begin{align*}
\Sigma_{1}^{N} &= \{\alpha_N(V)\}
= \Big\{\frac 1N \sum_{j=1}^N \ang(\Phi(j-1,0)V,\Phi(j,0)V)\Big\}
= \Big\{\frac 1N \sum_{j=1}^N g(j)\Big\}
\end{align*}
where
\[
g(j) = 
\begin{cases}
\varphi_0,&
\text{ for } j = 0 \lor j \in \bigcup_{\ell=1}^\infty
[2^{2\ell-1},2^{2\ell} -1]\cap \N\\[3mm]
\varphi_1,&\text{ otherwise}
\end{cases}
\]
The finite time outer angular spectrum  $\Sigma_{1}^{N}$ consists of
one element only for 
each $N\in\N$ and, therefore,  cannot approximate the interval
$\Sigma_1$. 
\end{example}
This example shows that  Lemma \ref{modest} does not imply lower
semi-continuity, i.e.\
\[
\dist(\Sigma_1,\Sigma_{1}^{N}) = \sup_{\theta\in
  \Sigma_1}\inf_{\vartheta\in\Sigma_{1}^{N}} |\theta-\vartheta| \to
0 \text{ as } N \to \infty.
\]
However, lower semi-continuity can be achieved for the union of finite time spectral
sets. For $N,M\in\N$, $N<M$ we define 
\begin{equation}\label{extended}
\Sigma_s^{N,M} \coloneqq \bigcup_{j=N}^M \Sigma_{s}^{j}
\end{equation}
and get lower semi-continuity in the following sense:
\begin{lemma} \label{lem4:lsc}
For any $N\in\N$ the following holds:
\[
\lim_{M\to \infty}
\dist\left(\Sigma_s,\Sigma_s^{N,M}\right) = 0.
\]
\end{lemma}

\begin{proof} Let us denote by $B_{\varepsilon}(X)=\{y \in \R:
  \dist(y,X) \le \varepsilon\}$ the $\varepsilon$-neighborhood of a
  set $X \subseteq \R$ . 
Fix $N\in\N$ and $\eps >0$. Similar to the proof of Lemma
\ref{modest}, we find for each $\theta\in\Sigma_s$ an element
$V\in\cD_0(s,d)$ such that 
\[
B_{\frac \eps 4}(\theta) \subset [\varliminf_{n\to \infty} \alpha_n(V) -
\tfrac \eps 2, \varlimsup_{n\to \infty}\alpha_n(V) + \tfrac \eps 2].
\]
Since $\lim_{n\to \infty} \alpha_n(V)-\alpha_{n+1}(V) = 0$, we obtain
a number $L=L(\theta)\ge N$ such that $B_{\frac \eps 4}(\theta) \subset
B_\eps(\alpha_{L(\theta)}(V))$ and  thus $B_{\frac \eps 4}(\theta)
\subset B_\eps(\Sigma_{s}^{L(\theta)})$.

Since $\Sigma_s$ is compact there
exist some $\ell\in\N$ and 
$\theta_1,\dots,\theta_\ell\in \Sigma_s$ such that $\Sigma_s
\subseteq \bigcup_{j=1}^\ell B_{\frac \eps 4}(\theta_j)$.
Let $M = \max_{j=1,\dots,\ell} L(\theta_j)$.
Then we get for all $\theta \in \Sigma_s$:
\[
\theta\in \bigcup_{j=N}^M B_\eps\left(\Sigma_{s}^{j}\right)
= B_\eps \Big(\bigcup_{j=N}^M \Sigma_{s}^{j}\Big)
= B_\eps \left(\Sigma_s^{N,M}\right)
\]
and 
\[
\dist\left(\Sigma_s,\Sigma_s^{N,M}\right) 
= \sup_{\theta\in\Sigma_s}\inf_{\vartheta\in\Sigma_s^{N,M}}
|\theta-\vartheta| \le \eps.
\]\
\end{proof}
The next example shows that the finite time outer angular spectra can
be much larger than the infinite ones. 
\begin{example}\label{E2}
Consider the system  \eqref{diffeq} with the following setting:
\begin{equation}\label{Ex2}
A_n \coloneqq
\begin{cases}
\left(\begin{smallmatrix}
-1 & 0 \\ 0 & 1
\end{smallmatrix}\right), 
& \text{ for } n \in \bigcup_{\ell = 1}^\infty [2\cdot 2^\ell -
4,3\cdot 2^\ell -5],\\
\left(\begin{smallmatrix}
1 & 0 \\ 0 & \frac 12
\end{smallmatrix}\right), 
&\text{ otherwise.}
\end{cases}
\end{equation}
It turns out that $\Sigma_1=\{0\}$ and we determine $\Sigma_{1}^{N} =
\{\alpha_1(V): V\in \cG(1,d)\}$. For a fixed  $N\in\N$ there exists
$\varepsilon>0$ such that  
\[
\alpha_N\left(\Span\begin{pmatrix}0\\1\end{pmatrix}\right) = 0 ,\quad
\alpha_N\left(\Span\begin{pmatrix}\eps\\1\end{pmatrix}\right) \ge
\frac \pi{12}.
\]
By the continuity of $\alpha_N$ we obtain $\Sigma_{1}^{N} \supseteq
[0,\frac \pi{12}]$ which is much larger than the outer angular
spectrum $\Sigma_1$. 
In particular, there seems to be no 
easy relation between $\Sigma_s$ and $\Sigma_{s}^{N}$. 
\end{example}
Example \ref{E2} shows that Lemma \ref{modest} does not imply upper
semi-continuity, i.e.\
\[
\dist(\Sigma_{1}^{N},\Sigma_{1}) = \sup_{\theta\in
  \Sigma_{1}^{N}}\inf_{\vartheta\in\Sigma_{1}} |\theta-\vartheta| \to
0 \text{ as } N \to \infty.
\]
For each $N\in\N$ we find 
$\dist(\Sigma_{1}^{N},\Sigma_{1})
= \sup_{\theta\in\Sigma_{1}^{N}} |\theta| \ge \tfrac \pi{12}$.
We also do not  get upper semi-continuity when replacing
$\Sigma_{1}^{N}$ by $\Sigma_1^{N,M}$.
In the following we look for  stronger assumptions which guarantee
upper semi-continuity. 
Recall from Lemma \ref{estap} the uniform Cauchy property, i.e.\ 
the sequence  $\alpha_{N}(V)$ satisfies 
\begin{equation}\label{uC}
\forall \eps>0\ \exists N=N(\eps)\in\N: \forall n,m\ge N: \forall V\in\cD_0(s,d):
|\alpha_n(V) - \alpha_m(V)| \le \eps.
\end{equation}
\begin{lemma}\label{oberhalb}
  If the sequence $\alpha_N(V)$, $V\in \cG(s,d)$, $N \in \N$ is uniformly Cauchy then
upper semi-continuity holds true in the following sense:
\begin{equation}\label{usc}
\forall \eps>0 \ \exists N\in\N: \forall M\ge N: \dist(
\Sigma_s^{N,M},\Sigma_s) \le \eps.
\end{equation}
\end{lemma}
\begin{proof}
  The uniform Cauchy property implies
 \begin{equation} \label{eq4:vorusc}
\forall \eps>0 \ \exists \bar N\in\N: \forall N\ge \bar N \ \forall
V\in\cD_0(s,d): \alpha_N(V) \in
B_\eps\left(\lim_{n\to\infty}\alpha_n(V)\right). 
\end{equation}
Thus, if we fix $N = \bar N(\eps)$ as in \eqref{eq4:vorusc} and let
$M\ge N$, then  
for every $\theta \in \Sigma_s^{N,M}$ there exist a $V\in\cD_0(s,d)$
and an index $j\in[N,M]\cap \N$ such that $\theta = \alpha_j(V)$, hence
$$
\theta \in B_\eps(\lim_{n\to\infty}\alpha_n(V)) \subset B_\eps(\Sigma_s).
$$\
This proves our assertion.
\end{proof}
\begin{remark}
Note that lower semi-continuity requires only the index $M$ 
in $\Sigma_s^{N,M}$ to be large while
upper semi-continuity also needs  $N$ to be  sufficiently large. 
\end{remark}
The uniform Cauchy property  even guarantees that the finite time
angular spectra converge in the 
Hausdorff distance without forming the union \eqref{extended}.
\begin{proposition} \label{prop4:Hausdorff}
If the uniform Cauchy condition \eqref{uC} is satisfied then
convergence holds w.r.t.\ the Hausdorff distance, i.e. 
\begin{equation}\label{dH}
\forall \eps>0\ \exists N\in\N: \forall n\ge N:
d_H(\Sigma_s,\Sigma_{s}^{n}) \le \eps, 
\end{equation}
where $d_H(U,V) \coloneqq \max\{\dist(U,V),\dist(V,U)\}$.
\end{proposition}
\begin{proof}
Given $\varepsilon>0$ choose $N \in \N$ such that 
\begin{equation}\label{ap-folg}
   |\alpha_n(V)- \lim_{\ell \to \infty}\alpha_{\ell}(V)| \le
   \frac{\varepsilon}{2} \quad 
  \forall n\ge N,\   \forall V\in\cD_0(s,d).
\end{equation}
Consider now $n\ge N$: 
\begin{itemize}
\item For $\theta\in \Sigma_{s}^{n}$ we find some $V\in\cD_0(s,d)$
  such that $\alpha_n(V) = \theta$. Since $\lim_{\ell \to
    \infty}\alpha_{\ell}(V)\in \Sigma_s$  the property \eqref{ap-folg} 
  shows  $\dist(\Sigma_{s}^{n},\Sigma_s) \le \frac \eps 2$. 
\item For $\theta\in\Sigma_s$ we find a $V\in\cD_0(s,d)$ such
  that $|\theta-\lim_{\ell\to\infty}\alpha_\ell(V)|\le \frac \eps 2$. 
With our choice of $N$ and $n$ in \eqref{ap-folg} we obtain
$$
|\theta-\alpha_n(V)| \le |\theta-\lim_{\ell\to \infty}\alpha_\ell(V)|
+ | \lim_{\ell\to \infty}\alpha_\ell(V) - \alpha_n(V)| \le \frac \eps
2 + \frac \eps 2 = \eps.
$$
Thus we have shown that $\dist(\Sigma_s,\Sigma_{s}^{n})\le \eps$ follows.
\end{itemize}
Combining both results proves \eqref{dH}.
\end{proof}
According to this result we can  determine the outer angular spectrum
of a system \eqref{diffeq}  
in an efficient way if it is uniformly Cauchy. It suffices to calculate
$\Sigma_{s}^{n}$ for  $n$ sufficiently large, while systems without
this property 
require to compute the union of several sets  $\Sigma_s^{N,M}$ with $M
\ge N$.  

Let us recall from Section \ref{sec3} that the uniform Cauchy property
is implied by uniform almost periodicity 
of the sequence $b_n(V)= \ang(\Phi(n-1,0)V,\Phi(n,0)V)$ (Lemma
\ref{estap}). The last property, in turn,  
holds for autonomous systems where the matrix $A$ has semi-simple
eigenvalues and the frequencies satisfy 
a nonresonance condition (Example \ref{ex3:AUapauto}). Moreover, this 
property is inherited by kinematically 
similar systems (Proposition \ref{invariant}). Among these
kinematically similar systems are those for which 
the system matrices $A_n$ are $\ell^1$-perturbations of $A$ (Theorem 
\ref{thm3:maininv}). This theorem needs  the CED property which
follows if the matrix $A$ has  only semi-simple eigenvalues; see
Example \ref{ex3:semisimple}. 
Summing all up, we find that the approximation result in Proposition
\ref{prop4:Hausdorff} applies to the $3$D-H\'{e}non system from
Example \ref{ex3:henon3} for which the nonresonance condition is
trivially satisfied. 

\subsection{An algorithm for computing the finite time outer angular
  $N$-spectrum}\label{algo}
We know turn the definition of the finite time outer angular spectrum
into an algorithm, having in mind our main application -- the
$3$D-H\'enon map from Example 
\ref{ex3:henon3}. By Proposition \ref{prop4:Hausdorff} it suffices to
compute the finite time outer angular $N$-spectrum for sufficiently
large $N$ in case $s\in\{1,2\}$ since $\Sigma_{s}^{N}$ converges to
$\Sigma_s$ w.r.t.\ the Hausdorff distance. In particular, all limits
exist. 

We assume that $\dim(\cW_0^k)\in\{1,2\}$ for all
$k=1,\dots,\varkappa$, cf.\ Definition \ref{deftrace}. 
Fix $N\in \N$ and for $V \in \cG(s,d)$ we abbreviate 
\[
  \theta_s(V) = \alpha_N(V) = \frac 1N \sum_{j=1}^{N}
\ang(\Phi(j-1,0)V,\Phi(j,0)V), \quad \theta_s(v)= \theta_s(\mathrm{span}(v))
\]
and define 
\[
\cB_0^k=\{v\in \cW_0^k: \|v\|=1 \}\quad  \text{for } k\in\{1,\dots,\varkappa\}.
\]

\subsubsection*{Step 1: Computation of the dichotomy spectrum and of
  the spectral bundles}
The computation of the dichotomy spectrum
$\Sigma_{\mathrm{ED}}^{\mathrm{approx}}$ and of the corresponding
spectral bundles is described in detail in \cite[Section 4.1, Step 1
and 2]{BeHu22}. In the following we use these techniques for approximating
the fibers $\cW_0^k$, $k=1,\dots,\varkappa$.  

\subsubsection*{Step 2: Computation of the first finite time outer
  angular $N$-spectrum $\Sigma_{1}^{N}$}
\
\medskip
 
\begin{center}
\begin{minipage}{.55\linewidth}
\begin{algorithmic}
\State{initialize $\Sigma_{1}^{N} = \emptyset$}
\For {$k = 1,\dots,\varkappa$}
\If{$\dim(\cW_0^k) =1$}
\State{$\Sigma_{1}^{N} = \Sigma_{1}^{N} \cup \theta_1(\cW_0^k)$}
\Else
\State{$\displaystyle \Sigma_{1}^{N} = \Sigma_{1}^{N} \cup
  \left[\min_{v \in \cB_0^k} \theta_1(v), 
    \max_{v \in \cB_0^k} \theta_1(v)\right]$}
\EndIf
\EndFor
\end{algorithmic}
\end{minipage}
\end{center}
\medskip

Note that we solve one-dimensional optimization problems in case
$\dim(\cW_0^k)=2$ with the MATLAB-routine \texttt{fminbnd}. 

\subsubsection*{Step 3: Computation of the second finite time outer
  angular $N$-spectrum $\Sigma_{2}^{N}$}
\
\medskip
 
\begin{center}
\begin{minipage}{.9\linewidth}
\begin{algorithmic}
\State{initialize $\Sigma_{2}^{N} = \emptyset$}
\For {$k = 1,\dots,\varkappa$}
\If{$\dim(\cW_0^k) =2$}
\State{$\Sigma_{2}^{N} = \Sigma_{2}^{N} \cup \theta_2(\cW_0^k)$} 
\EndIf
\EndFor
\For {$k_1 = 1,\dots,\varkappa-1$}
\For {$k_2 = k_1+1,\dots,\varkappa$}
\State{$\displaystyle \Sigma_{2}^{N} = \Sigma_{2}^{N} \cup
  \left[\min_{x\in\cB_0^{k_1},y\in\cB_0^{k_2}}\theta_2(\Span(x,y)), 
 \max_{x\in\cB_0^{k_1},y\in\cB_0^{k_2}}\theta_2(\Span(x,y))\right]$}
\EndFor
\EndFor
\end{algorithmic}
\end{minipage}
\end{center}
\medskip
The optimization problems that we solve in the last step
have dimension $\dim(\cW_0^{k_1}) + \dim(\cW_0^{k_2}) - 2$. 
For two-dimensional optimization problems, we apply the MATLAB-routine
\texttt{fminsearch}.

\subsection{Numerical results for the 3D-H\'enon system}
For $N\in\{50,100,1000,2000\}$, we apply the algorithm from above 
to the $3$D-H\'enon systems
\eqref{variational}, see Table \ref{num1} and to \eqref{var_auto},
cf.\ Table \ref{num2}. 
Note that we shift the index in \eqref{variational} such that the
main excursion of the homoclinic orbit $\bar x_0$ lies at $\frac N2$.

Errors that occur while computing homoclinic orbits and the
corresponding fiber bundles decay exponentially fast towards the
midpoint of the finite interval, see \cite[Section 2.6]{h17}. These
errors can be neglected, if the orbit and the fibers are
computed on sufficiently large intervals, while only the center part
of length $N$ is used. For this task, we introduce left and right
buffer intervals of length $500$, which are skipped in the final output. 

\begin{table}[hbt]
\begin{center}
\begin{tabular}{c|c|c|c}
$N$ & $\Sigma_{\mathrm{ED}}^{\mathrm{approx}}$ & $\Sigma_1^N$ &
 $\Sigma_2^N$\\\hline
$50$ & $[0.775,0.787]\cup[1.467,1.482]$ & $\{0.358\}\cup [1.108,1,264]$ &
$\{0.487\}\cup[1.211,1.275]$\\\hline
$100$ & $[0.776,0.785]\cup[1.468,1.482]$ & $\{0.178\}\cup [1.222,1,307]$ &
$\{0.242\}\cup[1.289,1.296]$\\\hline
$1000$ & $[0.779,0.784]\cup[1.470,1.478]$ & $\{0.018\}\cup [1.325,1,333]$ &
$\{0.024\}\cup[1.324,1.325]$\\\hline
$2000$ & $[0.779,0.784]\cup[1.471,1.477]$ & $\{0.009\}\cup [1.330,1,331]$ &
$\{0.012\}\cup[1.326,1.327]$
\end{tabular}
\end{center}
\caption{Numerical computation of the dichotomy spectrum and of 
  outer angular spectra for the nonautonomous
  $3$D-H\'enon system \eqref{variational}.\label{num1}}
\end{table}

Note that for the autonomous system \eqref{var_auto}, trace spaces are
eigenspaces. For a fair comparison however, we compute the trace
spaces also in the 
autonomous setting with the more general nonautonomous algorithm. 

\begin{table}[hbt]
\begin{center}
\begin{tabular}{c|c|c|c}
$N$ & $\Sigma_{\mathrm{ED}}^{\mathrm{approx}}$ & $\Sigma_1^N$ &
 $\Sigma_2^N$\\\hline
$50$ & $[0.781,0.782]\cup[1.473,1.474]$ & $\{10^{-16}\}\cup [1.328,1,344]$ &
$\{10^{-16}\}\cup[1.320,1.337]$\\\hline
$100$ & $[0.781,0.782]\cup[1.473,1.474]$ & $\{10^{-16}\}\cup [1.328,1,341]$ &
$\{10^{-16}\}\cup[1.321,1.334]$\\\hline
$1000$ & $[0.781,0.782]\cup[1.473,1.474]$ & $\{10^{-16}\}\cup [1.335,1,336]$ &
$\{10^{-15}\}\cup[1.328,1.328]$\\\hline 
$2000$ & $[0.781,0.782]\cup[1.473,1.474]$ & $\{10^{-16}\}\cup [1.335,1,336]$ &
$\{10^{-15}\}\cup[1.328,1.328]$\\
\end{tabular}
\end{center}
\caption{Numerical computation of the dichotomy spectrum and of 
  outer angular spectra for the autonomous
  $3$D-H\'enon system \eqref{var_auto}.\label{num2}}
\end{table}

The data in Table \ref{num2} show that the finite time outer angular
$N$-spectra depend on $N$. This causes $\rho\approx 0.697 \neq 1$ in
the normal form \eqref{rhomatrix} that belongs to the two-dimensional
eigenspace of $Df(\xi)$. However, these spectra converge quickly in
$N$ towards the infinite time spectra \eqref{Henspec1} and
\eqref{Henspec2}, respectively. 

Next, we discuss the results for the variational equation along the
homoclinic orbit in Table \ref{num1}. On the one hand, the center
part of the homoclinic orbit, see Figure 
\ref{Orbit_norm}, has for small $N$ a strong influence on the finite
time outer 
angular $N$-spectra. For sufficiently large $N$, on the other hand,
$Df(\bar x_n)$ is 
exponentially close to $Df(\xi)$, if $n$ lies at the boundaries of the
finite interval $[0,N]\cap \N_0$. Hence,
the finite time outer angular $N$-spectra of both systems
\eqref{var_auto} and \eqref{variational} converge for increasing $N$
towards each other. These observations are in line with the
theoretical results from Section \ref{sec4:FT}. 

\begin{figure}[hbt]
\begin{center}
\includegraphics[width = 0.6\textwidth]{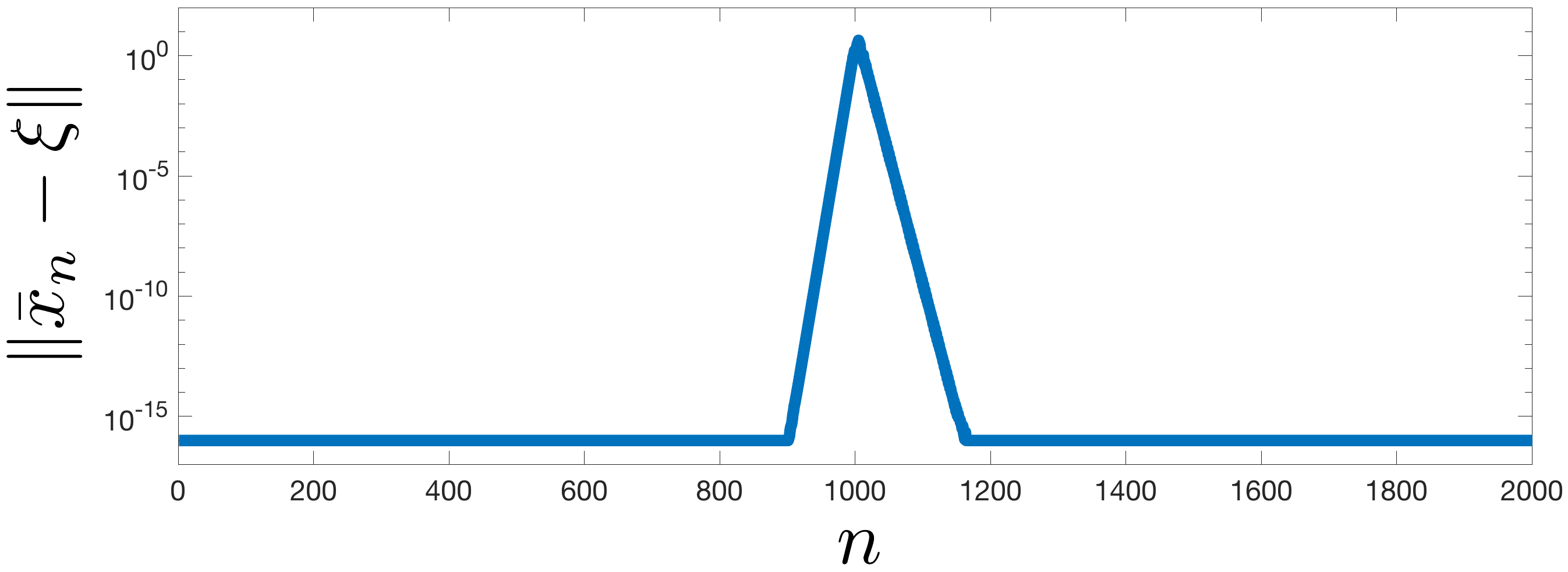}
\end{center}
\caption{Distance of a homoclinic orbit\protect\footnotemark
   $(\bar x_n)_{n\in\{0,1,\dots,2000\}}$ of the $3$D-H\'enon system
  \eqref{Hmap} to the fixed point $\xi$ in a semilogarithmic
  scale.\label{Orbit_norm}} 
\end{figure}
\footnotetext{The index of the
    orbit $(\bar x_n)_{n\in\{-1000,1000\}}$ from Example
    \ref{ex3:henon3} is shifted to the interval $[0,2000]\cap\N_0$.}

\subsection{Multi-humped homoclinic orbits for the 3D-H\'enon system}
While outer angular spectra of the variational equation along a
homoclinic $3$D-H\'enon orbit essentially depend only on the limit
matrix $DF(\xi)$, the situation changes when we consider so called
multi-humped homoclinic orbits, see \cite{BHS16}. We construct these
orbits by copying the center part of length $M$ of the primary homoclinic
orbit repeatedly, see Figure \ref{Multi}. 
Then Newton's method, applied to this pseudo-orbit
results in an $F$-orbit $(\bar x^M_n)_{n\in J}$ on the finite interval
$J=[0,2000]\cap \N_0$. 
\begin{figure}[hbt]
\begin{center}
\includegraphics[width = 0.55\textwidth]{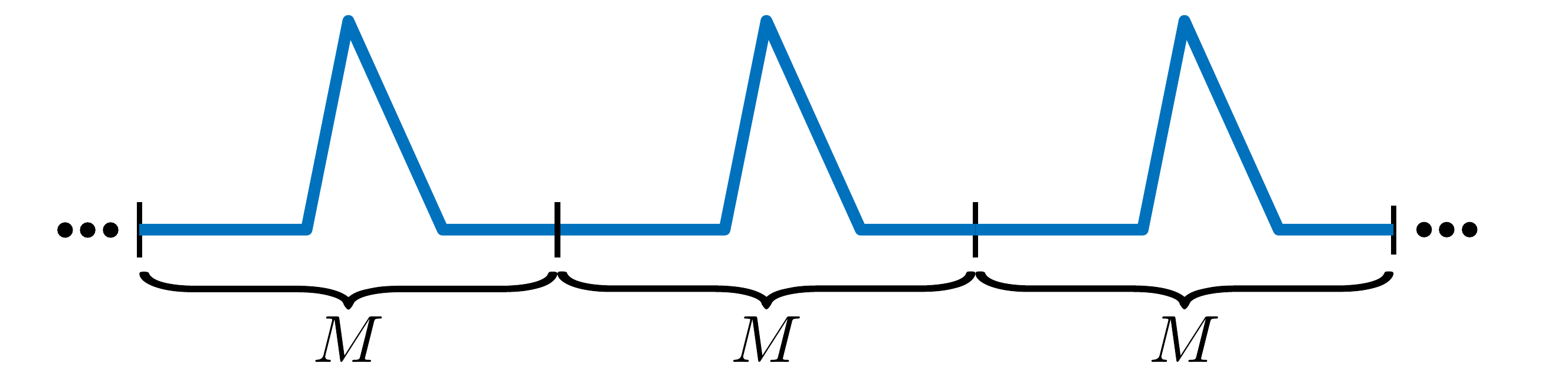}
\end{center}
\caption{Construction of multi-humped orbits.\label{Multi}}
\end{figure}

For large $M$, the dynamics at the fixed point essentially determines
the outer angular spectra. However, for small values of $M$, the
multiple center parts of the primary homoclinic orbit alter the outer
angular spectra, see Table \ref{num3}. For $M\in\{50,100,200\}$ we
observe that the dichotomy
spectrum consists of three intervals, resulting in one-dimensional
trace spaces. For $M = 400$, the algorithm from \cite[Section
4.1]{BeHu22} no longer separates the
values of the dichotomy spectrum and finds the interval $[0.773,
0.788]$. This interval leads to a two-dimensional
trace space and, as a consequence, intervals also occur 
in the outer angular spectra for $M = 400$.

Passing from $M$ to $2M$, we further observe that any value within 
the outer angular spectrum $\Sigma_1^{2001}$ or $\Sigma_2^{2001}$ for
$2M$ is roughly the arithmetic mean of 
the corresponding value for $M$ and
the respective value from \eqref{Henspec1} and \eqref{Henspec2}.  
In other words, the distance to the values from \eqref{Henspec1} and
\eqref{Henspec2} is halved when $M$ is doubled.  
Such a behavior is suggested by the fact that orbits stay twice as
long near the fixed point. 

\begin{table}[hbt]
\begin{center}
\begin{tabular}{l}
\begin{tabular}{c|c}
$M$ & $\Sigma_{\mathrm{ED}}^{\mathrm{approx}}$  \\\hline
$50$ & $[0.7230, 0.7234]\cup [0.8298,0.8302]\cup[1.4992, 1.4994]$ \\\hline
$100$ & $[0.7492, 0.7494]\cup[0.8077,0.8079]\cup[1.4868,1.4870]$  \\\hline
$200$ & $[0.7634, 0.7662] \cup [0.7935, 0.7966]\cup [1.4775, 1.4827]$ \\\hline
$400$ & $[0.7733, 0.7884] \cup [1.4741, 1.4791]$ \\\hline
\end{tabular}\\[15mm]
\begin{tabular}{c|c|c}
$M$ & $\Sigma_1^{2001}$ &
$\Sigma_2^{2001}$\\\hline
$50$ & 
$\{0.353, 1.135, 1.260\} $ & $\{0.480, 1.248, 1.264\}$\\\hline
$100$ & 
$\{0.177, 1.229, 1.300\}$ & $\{0.239, 1.295, 1.296\}$\\\hline
$200$ & 
$\{0.088, 1.281, 1.318\}$ & 
$\{0.120, 1.312, 1.313\}$\\\hline
$400$ &  $\{0.044\}\cup[1.311, 1.314]$ &
$\{0.060\}\cup[1.320, 1.321]$\\\hline
\end{tabular}
\end{tabular}
\caption{Numerical computation of the dichotomy spectrum and of 
  outer angular spectra for the $3$D-H\'enon system
  w.r.t.\ multi-humped homoclinic orbits with center parts of length
  $M$.\label{num3}} 
\end{center}
\end{table}

\subsection{Angular spectra for the Lorenz system}
\label{sec4:Lorenz}
The famous Lorenz system \cite{L63} is given by the ODE
\begin{equation}\label{ODElorenz}
\begin{pmatrix}
y_1\\y_2\\y_3
\end{pmatrix}'
=
\begin{pmatrix}
\sigma(y_2-y_1)\\\rho y_1 - y_2 - y_1 y_3\\ y_1 y_2 - \beta y_3
\end{pmatrix}
\quad \text{with parameters}\quad \sigma = 10,\ \rho = 28,\ \beta = \tfrac 83.
\end{equation}
In order to apply our results, we discretize this ODE and compute the 
$h$-step map $F_h$ for $h\in\{0.05, 0.1, 0.2\}$. For this task, we
apply the classical 
Runge-Kutta scheme with step size $10^{-4}$. For the resulting discrete
time system
\begin{equation}\label{lorenzh}
x_{n+1} = F_h(x_n),\quad n\in\N_0
\end{equation}
we calculate an orbit $(\bar x_n)_{n\in\{0,\dots,11000\}}$ w.r.t.\ the
initial value 
$\bar x_0=\begin{pmatrix}10 & 10 & 10\end{pmatrix}^\top$ that
converges towards the Lorenz attractor.
To the resulting variational equation 
\[
u_{n+1} = DF_h(\bar x_n)u_n,\quad n = 0,\dots,11000
\]
we apply the algorithm from Section \ref{algo} with left and right
buffer intervals of length $500$. Here, the Jacobians $DF_h(\bar x_n)$ are
computed numerically using the central difference quotient.
Table \ref{num4} gives the resulting spectra for $h\in\{0.05, 0.1,
0.2\}$.

\begin{table}[hbt]
\begin{center}
\begin{tabular}{l}
\begin{tabular}{c|c}
$h$ & $\Sigma_{\mathrm{ED}}^{\mathrm{approx}}$  \\\hline
$0.05$ & $[0.4821, 0.4833]\cup [0.9995, 1.0005]\cup[1.0445, 1.0478]$
  \\\hline 
$0.1$ & $[0.2325, 0.2332]\cup [0.9995, 1.0007]\cup[1.0928, 1.0971]$
  \\\hline 
$0.2$ & $[0.0542, 0.0544]\cup [0.9994, 1.0006]\cup[1.1956, 1.2004]$
  \\\hline 
\end{tabular}\\[15mm]
\begin{tabular}{c|c|c}
$h$ & $\Sigma_1^{10001}$ &
$\Sigma_2^{10001}$\\\hline
$0.05$ & 
$\{0.2039, 0.3803, 0.4234\} $ & $\{0.0689, 0.3604, 0.4188\}$\\\hline
$0.1$ & 
$\{0.3925, 0.7282, 0.8268\} $ & $\{0.1356, 0.7021, 0.8197\}$\\\hline
$0.2$ & 
$\{0.6552, 0.7334, 0.9727\} $ & $\{0.2475, 0.7934, 0.9752\}$\\\hline 
\end{tabular}
\end{tabular}
\caption{Dichotomy spectrum and angular spectra for the $h$-step
  Lorenz map w.r.t.\ an orbit on the Lorenz attractor. \label{num4}}
\end{center}
\end{table}

We observe that spectral values of the dichotomy spectrum are squared
when passing from $h$ to $2h$. For the corresponding values of the
outer angular spectrum, we expect them to double for sufficient small
$h$, while this doubling cannot be expected for larger values of
$h$. This interpretation is in line with the numerical data in Table
\ref{num4}. 

We shift the orbit outside the buffer intervals to $(\bar
x_n)_{n\in\{0,\dots,10^4\}}$ and illustrate the trace spaces
$\cD_n(1,3)=\{\cW_n^{1}, \cW_n^{2}, \cW_n^{3}\}$ for $n=2400$ and
$h=0.05$ in Figure \ref{lorenz1}. 

\begin{figure}[hbt]
\begin{center}
\includegraphics[width = 0.70\textwidth]{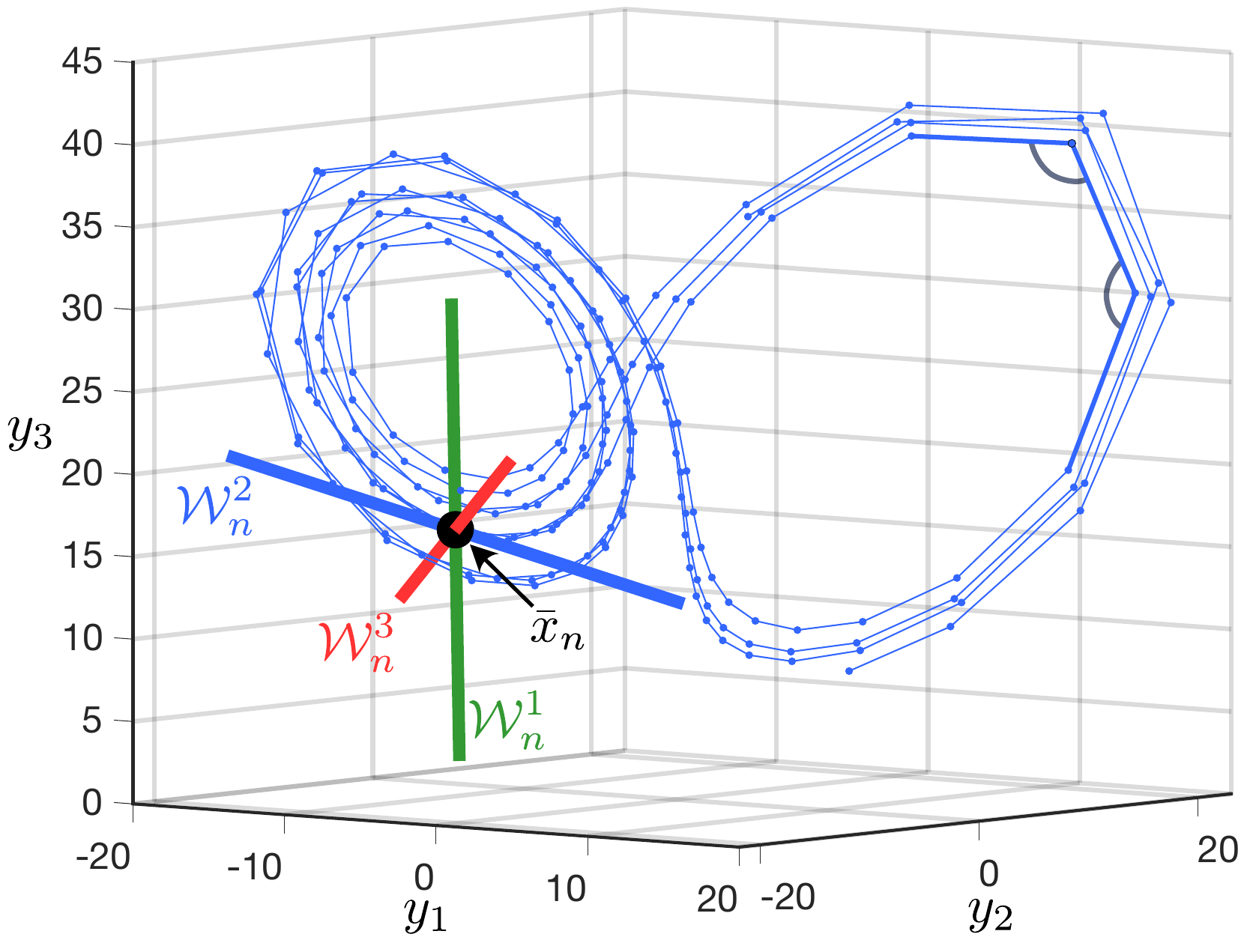}
\end{center}
\caption{Part of an $F_{0.05}$-orbit $(\bar
  x_n)_{n\in\{2300,\dots,2500\}}$ (connected with lines) of the Lorenz system
  \eqref{lorenzh}. For $n = 2400$ the three trace spaces
  $\cW_n^{1,2,3}$ are shown.\label{lorenz1}}
\end{figure}
The trace space $\cW_n^2$ lies in the direction of the flow of the ODE
\eqref{ODElorenz} and
belongs to the spectral interval around $1$ of the dichotomy
spectrum. Here the 
maximal value of the outer angular spectrum $0.4234$ is achieved. The
trace space $\cW_n^1$ also lies in the `plane' of the attractor and
leads to the spectral value $0.3803$. Finally, $\cW_n^3$ points
outside of the `plane' of the attractor and gives the smallest value
$0.2039$ of the outer angular spectrum.  

The outer angular spectrum of dimension $s=2$ is calculated w.r.t.\
the trace spaces $\cD_n(2,3) = \{\cW_n^1\oplus \cW_n^2,
\cW_n^1\oplus \cW_n^3, \cW_n^2\oplus \cW_n^3\}$. The largest value
$0.4188$ of $\Sigma_2$ is attained in $\cW_n^2\oplus \cW_n^3$ while
$\cW_n^1\oplus \cW_n^3$ results in the angular value $0.3604$. The
smallest value $0.0689$ is attained in the `plane' of the attractor
$\cW_n^1\oplus \cW_n^2$. 

The angle between successive iterates of the $h$-step Lorenz map on
average, as indicated in the upper right of 
Figure \ref{lorenz1}, is an alternative way to measure the rotation of
a trajectory. The resulting values are given in Table \ref{num5}. For
sufficiently small $h$, these values turn out to be close to the
largest spectral value of $\Sigma_1^{10001}$. 

\begin{table}[hbt]
\begin{center}
\begin{tabular}{c|c|c|c}
$h$ & $0.05$ & $0.1$ & $0.2$\\\hline
angle on average & $0.4227$ & $0.8322$ & $1.0993$
\end{tabular}
\caption{Angle between successive iterates on average for the $h$-step
  Lorenz 
  map \eqref{lorenzh}.\label{num5}} 
\end{center}
\end{table}

Angular values for continuous time systems have been introduced in
\cite[Definition 4.3]{BeHu23X}. These normalized values are defined as
\begin{equation}\label{conti}
\theta_{s,h,T}^{\mathrm{cont}} = \sup_{V\in\cD_0(s,d)} \frac 1T \sum_{j=1}^N
\ang(\Phi_h(j,0)V , \Phi_h(j-1,0)V),\quad N = \tfrac Th,
\end{equation}
where $T$ denotes the length of the time interval, $h$ the step size
and $\Phi_h$ is the solution operator of the $h$-step Lorenz map. 
In the limit $h\to 0$, $\theta_{s,h,T}^{\mathrm{cont}}$ measures the
average rotation of an $s$-dimensional subspace per unit time interval
when observed during 
$T$ time units.
When the time $T$ of observation  tends to infinity we obtain
an angular value of the continuous time system.
Note that $\theta_{s,h,T}^{\mathrm{cont}}$ is not restricted
to $[0,\frac \pi 2]$, since a subspace may rotate multiple times on a
time interval of length $1$. 
For the Lorenz system, the supremum in \eqref{conti} is attained in
$\cW_0^2$ for $s=1$ and in $\cW_0^2\oplus \cW_0^3$ for $s=2$.
In Table \ref{num6}, we compute $\theta_{s,h,T}^{\mathrm{cont}}$ for
$s\in\{1,2\}$, $T=500$ and $h\in\{0.025,0.05,0.1,0.2\}$. To ensure
accurate results, additional left and right buffer intervals (in time) of
length $T_{\mathrm{buffer}} = 25$ are added. The data in Table
\ref{num6} illustrate the convergence of
$\theta_{s,h,T}^{\mathrm{cont}}$ as $h\to 0$ towards the 
angular value of the continuous system in finite time $T=500$.

\begin{table}[hbt]
\begin{center}
\begin{tabular}{c|c|c|c|c}
$h$ & $0.025$ & $0.05$ & $0.1$ & $0.2$\\\hline
$\theta_{1,h,T}^{\mathrm{cont}}$ & $8.4798$ & $8.4672$ & $8.2896$ &
$4.8752$\\\hline
$\theta_{2,h,T}^{\mathrm{cont}}$ & $8.3816$ & $8.3753$ & $8.2209$ & $4.8941$
\end{tabular}
\caption{Normalized angular values $\theta_{1,h,T}^{\mathrm{cont}}$ and 
$\theta_{2,h,T}^{\mathrm{cont}}$ for the $h$-step Lorenz
  map \eqref{lorenzh}.\label{num6}} 
\end{center}
\end{table}

\section{Further types of angular spectrum}
\label{sec5}
In \cite[Section 3]{BeFrHu20} we defined further types of angular
values, such as inner or uniform angular values. The inner 
versions take  the supremum and  the limit as $n \to \infty$ in
Defintion \ref{def3:1} in reverse order while the uniform versions 
measure averages of angles starting at
an arbitrary time and lasting for sufficiently long time.
In this section we briefly discuss how angular spectra associated to
these modified angular values should be defined 
and we mention some of their basic properties. Moreover, we suggest
how the numerical methods for outer angular 
values from Section \ref{Sec_num} should be modified.

\subsection{Inner angular spectrum}
The {\bf inner angular values} of dimension $s$ are defined as follows
(cf.\ \cite[Definition 4.1]{BeHu23X}) 
\begin{alignat*}{2}
  \theta_{s}^{\varlimsup,\sup} & = 
  \varlimsup_{n\to\infty} \sup_{V \in \mathcal{G}(s,d)} \alpha_n(V), \quad
   \theta_s^{\varliminf,\sup} & =
  \varliminf_{n\to\infty}  \sup_{V \in \mathcal{G}(s,d)}\alpha_n(V),\\
  \theta_{s}^{\varlimsup,\inf} & =
  \varlimsup_{n\to\infty}   \inf_{V \in \mathcal{G}(s,d)} \alpha_n(V), \quad
   \theta_s^{\varliminf,\inf} & =
  \varliminf_{n\to\infty} \inf_{V \in \mathcal{G}(s,d)} \alpha_n(V).
\end{alignat*}
This suggests the following definition of the \textbf{inner angular spectrum}
\begin{equation}\label{inner}
\Sigma^{\mathrm{in}}_s\coloneqq\cl\big\{\theta\in[0,\tfrac \pi 2]: \exists
(V_n)_{n\in\N}\in\cG(s,d)^\N:
\varliminf_{n\to\infty} \alpha_n(V_n) \le \theta\le
\varlimsup_{n\to\infty}\alpha_n(V_n)\big\}. 
\end{equation}
In order to distinguish this from the outer angular spectrum in
Definition \ref{def3:1} we now  write 
$\Sigma_s^{\mathrm{out}}$ instead of $\Sigma_s$. Some properties of 
$\Sigma_s^{\mathrm{in}}$  are collected in the following proposition.  
\begin{proposition} \label{prop5:properties}
  The inner angular values and the inner angular spectrum satisfy
  \[
\begin{matrix}
\theta_s^{\varliminf,\inf} & \le &  \theta_s^{\varlimsup,\inf}\\
\rle && \rle\\ 
\theta_s^{\varliminf,\sup} & \le &  \theta_s^{\varlimsup,\sup}.
\end{matrix} 
\]
  \begin{equation} \label{eq5:specbound}
\Sigma_s^{\mathrm{out}},  \{\theta_s^{\varliminf,\inf},\theta_s^{\varlimsup,\inf},
  \theta_s^{\varliminf,\sup}, \theta_s^{\varlimsup,\sup}\} \subseteq
  \Sigma_s^{\mathrm{in}} 
\subseteq [\theta_s^{\varliminf,\inf},\theta_s^{\varlimsup,\sup}].
  \end{equation}
  If $\varliminf$ and $\varlimsup$ coincide, i.e.\
  $\theta_s^{\varliminf,\inf}=\theta_s^{\varlimsup,\inf}$ and 
  $  \theta_s^{\varliminf,\sup}=\theta_s^{\varlimsup,\sup}$ then
  $\Sigma_s^{\mathrm{in}}$ equals its maximal interval, i.e.\ 
  \begin{equation} \label{eq5:inspeceq}
    \Sigma_s^{\mathrm{in}}=
    [\theta_s^{\varliminf,\inf},\theta_s^{\varlimsup,\sup}]. 
    \end{equation}
\end{proposition}
\begin{proof}
  The proof of the first assertions is very similar to
  \cite[(3.10)]{BeFrHu20} and to Proposition \ref{prop2:relate}. 
  Therefore, we only prove \eqref{eq5:inspeceq}.
     Recall that the function $\alpha_n:\cG(s,d) \to \R$ is continuous
     on the connected 
    manifold $\cG(s,d)$. 
Let $\theta_s^{\varliminf,\inf} < \theta < \theta_s^{\varlimsup,\sup}$
and let  
$0<\eps
<\min(\theta-\theta_s^{\varliminf,\inf},\theta_s^{\varlimsup,\sup}-\theta)$. 
By our assumption there exists an $N\in\N$ such that for each $n\ge N$
we find 
subspaces $V_n^\pm\in\cG(s,d)$ with 
$$
|\theta_s^{\varliminf,\inf}-\alpha_n(V_n^-)|\le \eps,\quad
|\theta_s^{\varlimsup,\sup} - \alpha_n(V_n^+)| \le \eps.
$$
Then the intermediate value theorem applies to $\alpha_n$  on 
a path from $V_n^-$ ot $V_n^+$. Hence, there exists some
$V_n^\theta\in\cG(s,d)$ such that 
$$
\alpha_n(V_n^\theta) = \theta.
$$
As a consequence 
$$
\lim_{n\to\infty} \alpha_n(V_n^\theta) = \lim_{n\to \infty} \theta =
\theta \in \Sigma_s^{\mathrm{in}}.
$$ 
\
\end{proof}

Note that the existence of the limits
$\theta_s^{\varliminf,\inf}=\theta_s^{\varlimsup,\inf}$, 
  $  \theta_s^{\varliminf,\sup}=\theta_s^{\varlimsup,\sup}$
is satisfied in the setting of random
dynamical systems, see \cite[Section 4]{BeFrHu20}.

\subsection{Uniform outer angular spectrum}

We introduce a uniform variant of the outer angular spectrum. Uniform
outer angular values of dimension $s$ are defined in analogy to the
construction of Bohl exponents see \cite{Barabanov2017},
  \cite[Ch.III.4]{dk74}. Let
\[
\alpha_{n,k}(V):
\begin{array}{rcl} \cD_0(s,d) & \to & [0,\tfrac \pi 2]\\[1mm]
V & \mapsto &  \displaystyle\frac 1n
\sum_{j=k+1}^{k+n}\ang(\Phi(j-1,0)V,\Phi(j,0)V)
\end{array}
\]
and 
\begin{equation}\label{unif}
\theta_s^{\inf,\lim,\inf} = \inf_{V\in\cD_0(s,d)} \lim_{n\to \infty}
\inf_{k\in\N_0} \alpha_{n,k}(V),\quad
\theta_s^{\sup,\lim,\sup} = \sup_{V\in\cD_0(s,d)} \lim_{n\to \infty}
\sup_{k\in\N_0} \alpha_{n,k}(V).
\end{equation}
We refer to \cite[Lemma 3.3]{BeFrHu20} for the existence of the limits 
in \eqref{unif}.
The uniform outer angular spectrum is given by
\[
\Sigma_s^\mathrm{out,unif} \coloneqq \cl\big\{\theta\in[0,\tfrac \pi 2]:
\exists V\in\cG(s,d): \lim_{n\to \infty}\inf_{k\in\N_0}
\alpha_{n,k}(V) \le \theta\le 
\lim_{n\to\infty}\sup_{k\in\N_0}\alpha_{n,k}(V)\big\}. 
\]
For autonomous systems we observe that $\Sigma_s^{\mathrm{out}} =
\Sigma_s^{\mathrm{out,unif}}$ and in general, we obtain the relation
\[
\Sigma_s^{\mathrm{out}} , \{\theta_s^{\inf,\lim,\inf},
\theta_s^{\sup,\lim,\sup}\} \subset \Sigma_s^{\mathrm{out,unif}}.
\]
\subsection{Relations between angular spectra}
Revisiting the motivating Example \ref{ex2:exauto},\linebreak which is
autonomous, we observe for $s\in\{1,2\}$ that  
\[
  \Sigma_s^{\mathrm{out}}= \Sigma_s^{\mathrm{out,unif}} =
  \{0,\varphi\}\subset\Sigma_1^{\mathrm{in}} = [0,\varphi] . 
\]
  For Example \ref{E2} we obtain 
\[
\{0\} = \Sigma_1^{\mathrm{out}} = \Sigma_1^{\mathrm{out,unif}} 
\subset [0,\tfrac \pi{12}] \subseteq \Sigma^{\mathrm{in}}_1. 
\]
 Finally, for Example \ref{E1}
note that all subspaces $V\in\cG(s,d)$ have  the same
angular value, hence outer and inner spectra coincide. However, the
uniform outer spectrum is larger, more precisely
\[
\Sigma_1^{\mathrm{out}} = [\tfrac 23 \varphi_0+\tfrac 13 \varphi_1,
\tfrac 13 \varphi_0 + \tfrac 23 \varphi_1] =
\Sigma_1^{\mathrm{in}} \subset [\varphi_0,\varphi_1] = \Sigma_1^{\mathrm{out,unif}}.
\]

In conclusion, the outer angular spectrum provides the angular values
that may occur in the limit on average. The uniform outer angular
spectrum allows to start at arbitrary positions (in $n$) to determine
the angular values. However, a reliable approximation is much harder
to achieve numerically than for the outer angular spectrum. 
The inner angular spectrum provides a less refined 
spectral separation and often consists of a single interval; see
Proposition \ref{prop5:properties}. Furthermore, a reliable
numerical computation seems to be extremely costly due to the repeated search for suitable
subspaces $V_n \in \cG(s,d)$ at each $n \in \N$. Summarizing, we believe that the notion of the
outer angular spectrum is the most fruitful one, both theoretically and numerically.
    

\begin{appendix}
 
\section{Proof of Theorem \ref{thm3:rough}} \label{app2b}
The proof mixes various techniques from \cite{BH04,Hu03,P88}.
For every fixed $m \in J$ we consider an inhomogenous
system for matrices $U_n\in \R^{d,d}$, $n\in J$ 
      \begin{equation} \label{appb:inhom}
        \begin{aligned}
          U_{n+1}&=A_nU_n+ R_n,\quad n \in J, \\
          P_{n_-}^+U_{n_-}&= \delta_{n_-,m}P_{n_-}^+,
        \end{aligned}
      \end{equation}
      where $R_n \in \R^{d,d}$, $n \in J$ are given.
      Let us further introduce $m$-dependent weighted norms and spaces
      for $U_J=(U_n)_{n \in J}$: 
      \begin{alignat*}{2}
        \|U_J\|_{m,1}&= \sum_{\ell \in J} \Lambda_{\ell,m}
        \|U_{\ell}\|, \quad \Lambda_{\ell,m}&&=\begin{cases} 
        \lambda_+^{m-\ell}, & \ell \ge m, \\ \lambda_-^{m-\ell}, &
        \ell <m, \end{cases}\\ 
        \|U_J\|_{m,\infty}&= \max_{\ell \in J} \Lambda_{\ell,m}
        \|U_{\ell}\|, \quad Z_m&&=\{U_J\in (\R^{d,d})^J: 
         \|U_J\|_{m,\infty}< \infty\}, 
      \end{alignat*}
      where we set $\Lambda_{\ell,m}=0$ if $(\lambda_+=0, \ell \ge m)$
      or if $(\lambda_-=\infty, \ell < m)$. 
        The Green's function of the unperturbed system \eqref{diffeq}
        is given by  
      \begin{equation} \label{appb:Green}
        G(n,m)= \begin{cases} \Phi(n,m) P_m^+, & n \ge m \in J, \\
          - \Phi(n,m) P_m^-, &  m > n \in J.
        \end{cases}
      \end{equation}
      The GED shows that  $\|G(\cdot,m)\|_{m,\infty}\le K$ holds  and
      one  verifies that $G(\cdot,m)$ satisfies 
      \begin{equation} \label{appb:inhomG}
        \begin{aligned}
          G(n+1,m)&=A_nG(n,m)+ \delta_{n,m-1}I,\quad n \in J, \\
          P_{n_-}^+G(n_-,m)&= \delta_{n_-,m}P_{n_-}^+.
        \end{aligned}
      \end{equation}
      \begin{proposition} \label{appb:prop1}
        For every $R_J\in (\R^{d,d})^J$ with $\|R_J\|_{m,1}< \infty$
        there exists a unique solution $U_J\in Z_m$ 
        of \eqref{appb:inhom}. Further, the following estimate holds:
        \begin{equation} \label{appb:estu}
          \|U_J\|_{m,\infty}\le K \left(\lambda_{\star} \|R_J\|_{m,1}+
            |\delta_{n_-,m}|\right). 
        \end{equation}
      \end{proposition}
      \begin{proof}
        First, we prove uniqueness. If $U_J\in Z_m$ solves the homogenous
        system \eqref{appb:inhom}, then we 
        have $P_{n_-}^+U_{n_-}=0$ and  
        for $n \ge m$
        \begin{align*}
          \| P_{n_-}^-U_{n_-}\| & = \|\Phi(n_-,n)P_n^-U_n\| \le K
                                  \lambda_-^{n_--n}
                                  \Lambda_{n,m}^{-1}\Lambda_{n,m}
                                  \|P_n^-U_n\|\\ 
          & \le K^2 \lambda_-^{n_-} \Big(
            \tfrac{\lambda_+}{\lambda_-}\Big)^n \lambda_+^{-m}
            \|U_J\|_{m,\infty}\to 0 
          \text{ as } n \to \infty.
        \end{align*}
        Therefore, $U_{n_-}=P_{n_-}^-U_{n_-}+P_{n_-}^+U_{n_-}=0$ and
        $U_J=0$ follows. Next we show that 
        \begin{equation} \label{appb:sol}
          U_n= (G_{[m]} R_J)_n\coloneqq \sum_{\ell \in J} G(n,\ell+1)
          R_{\ell} +\delta_{n_-,m} G(n,n_-), \quad n\in J 
        \end{equation}
        solves \eqref{appb:inhom} and satisfies \eqref{appb:estu} (we
        keep  the dependence of the map 
        $G_{[m]}$ on $m$ but suppress it for $U_n$).
        From the definition of $G$ we have
        \begin{align*}
          P_{n_-}^+U_{n_-}=P_{n_-}^+ \sum_{\ell \in
          J}(-P_{n_-}^-)\Phi(n_-,\ell+1)R_{\ell}+
          \delta_{n_-,m}P_{n_-}^+ 
  G(n_-,n_-)          =\delta_{n_-,m}P_{n_-}^+.
          \end{align*}
        Further, the properties of the transition operator $\Phi$
        yield for $n \in J$ 
        \begin{align*}
          U_{n+1}& = \sum_{\ell+1\le n+1}
                   \Phi(n+1,\ell+1)P_{\ell+1}^+R_{\ell} - 
          \sum_{\ell+1>n+1}\Phi(n+1,\ell+1)P_{\ell+1}^- R_{\ell}\\
          &+ \delta_{n_-,m}\Phi(n+1,n_-)
          = A_n\Big(\sum_{\ell+1\le n}
            \Phi(n,\ell+1)P_{\ell+1}^+R_{\ell}+
            \delta_{n_-,m}\Phi(n,n_-) \Big)\\ 
          &+
          P_{n+1}^+R_n 
           - A_n \Big(\sum_{\ell+1>n}\Phi(n,\ell+1)P_{\ell+1}^-
            R_{\ell}\Big) + P_{n+1}^-R_n 
          = A_nU_n +R_n.
        \end{align*}
        Finally, we prove the estimate \eqref{appb:estu}. For $n\ge m
        > n_-$ and $0<\lambda_+< \lambda_-$  we obtain 
        \begin{align*}
          \lambda_+^{m-n}\|U_n\| & \le  \lambda_+^{m-n} \Big(
                                   \sum_{\ell+1 \le
                                   m}\|\Phi(n,\ell+1)P_{\ell+1}^+
                                   R_{\ell}\| 
          + \sum_{m<\ell+1 \le n}\|\Phi(n,\ell+1)P_{\ell+1}^+
                                   R_{\ell}\| \Big. \\ 
          & \qquad  \qquad + \Big. \sum_{n<\ell+1
            }\|\Phi(n,\ell+1)P_{\ell+1}^- R_{\ell}\| \Big) \\ 
          & \le K\Big( \sum_{\ell+1 \le m}
            \Big(\tfrac{\lambda_+}{\lambda_-}\Big)^{m  -\ell-1}
            \lambda_-^{-1} 
          \lambda_-^{\ell-m} \|R_{\ell}\|  + \sum_{m<\ell+1 \le
            n}\lambda_+^{-1} \lambda_+^{m-\ell}  
          \|R_{\ell}\|\Big.\\
          & \qquad \qquad \Big. + \sum_{n<\ell+1 } \Big(
            \tfrac{\lambda_+}{\lambda_-}\Big)^{\ell-n} \lambda_-^{-1} 
          \lambda_-^{m-\ell} \|R_{\ell}\| \Big)\\
          & \le K\lambda_+^{-1} \sum_{\ell \in J}
            \Lambda_{\ell,m}\|R_{\ell}\| = K
            \lambda_+^{-1}\|R_J\|_{m,1}. 
        \end{align*}
        Note that in case $m=n_-$ this estimate holds for the first
        term in \eqref{appb:sol}. 
        For the second term we have for $n\ge n_-$
        \begin{align*}
          \lambda_+^{n_--n} \|G(n,n_-)\| &
                                           =\lambda_+^{n_--n}\|\Phi(n,n_-)P_{n_-}^+\| \le K. 
        \end{align*}
                Similarly, we find for $n_-\le n<m$ and $\lambda_+<
                \lambda_- < \infty$: 
        \begin{align*}
          \lambda_-^{m-n}\|U_n\| & \le \lambda_-^{m-n} \Big(
                                   \sum_{\ell+1 \le n}
                                   \|\Phi(n,\ell+1)P_{\ell+1}^+R_{\ell}\| 
          +\sum_{n<\ell+1 \le m }
                                   \|\Phi(n,\ell+1)P_{\ell+1}^-R_{\ell}\| \Big. \\ 
          &\qquad \qquad \Big . +\sum_{m<\ell+1 }
            \|\Phi(n,\ell+1)P_{\ell+1}^-R_{\ell}\| \Big) \\ 
         & \le K \Big( \sum_{\ell+1\le n}
           \Big(\tfrac{\lambda_+}{\lambda_-}\Big)^{n-\ell -1}
           \lambda_-^{-1} \lambda_-^{m-\ell} \|R_{\ell}\| 
          + \sum_{n< \ell+1\le m}\lambda_-^{-1} \lambda_-^{m-\ell}
           \|R_{\ell}\| \Big. \\ 
            & \qquad \quad +\Big. \sum_{m< \ell+1} \Big(
              \tfrac{\lambda_+}{\lambda_-}\Big)^{\ell-m}
              \lambda_-^{-1} \lambda_+^{m-\ell} 
          \|R_{\ell}\|\Big)\\
          & \le K \lambda_-^{-1} \sum_{\ell \in J} \Lambda_{\ell,m}
            \|R_{\ell}\|= K \lambda_-^{-1} \|R_J\|_{m,1}. 
          \end{align*}
          Collecting the estimates shows the assertion \eqref{appb:estu}.
        \end{proof}
       Next we show that the perturbed system
        \begin{equation} \label{appb:Greenpert}
          \begin{aligned}
            X(n+1,m)& = (A_n+E_n)X(n,m)+ \delta_{n,m-1}I, \quad n \in J, \\
            P_{n_-}^+X(n_-,m)& = \delta_{n_-,m}P_{n_-}^+,
          \end{aligned}
        \end{equation}
        has a unique solution $X(\cdot,m)\in Z_m $ for every $m \in J$.
        By Proposition \ref{appb:prop1} and \eqref{appb:sol} this
        system is equivalent to the fixed point problem 
        \begin{equation} \label{appb:intform}
          X(\cdot,m) = G_{[m]}(E_J X(\cdot,m)+\delta_{\cdot,m-1}I).
        \end{equation}
        Note that $\| \delta_{\cdot,m-1}I\|_{m,1}<\infty$ and $
        \|E_JX(\cdot,m)\|_{m,1}<\infty$ hold for $X(\cdot,m) \in Z_m$ as the 
        following estimate shows.
        From the derivation of \eqref{appb:estu} we obtain for
        $X_1,X_2\in Z_m$ 
        \begin{align*}
          \|\big[G_{[m]}(E_JX_1)-G_{[m]}(E_JX_2)\big](\cdot,m)\|_{m,\infty}
          & \le K\lambda_{\star} \sum_{\ell \in
            J}\Lambda_{\ell,m}\|E_{\ell} 
        (X_1(\ell,m)-X_2(\ell,m))\| \\
        & \le  q \|(X_1-X_2)(\cdot,m)\|_{m,\infty},
        \end{align*}
        where $q=K \lambda_{\star}\|E_J\|_{\ell^1(J)}<1$ by
        \eqref{eq3:condpert}. Hence, the contraction mapping theorem 
        applies to the system \eqref{appb:intform} which then has 
        a unique solution $\tilde{G}(\cdot,m)\coloneqq X(\cdot,m) \in
        Z_m$. Since $G(\cdot,m)$ solves \eqref{appb:intform} 
        with $E_J \equiv 0$ we obtain the estimates 
        \begin{align*}
          \|G(\cdot,m)- \tilde{G}(\cdot,m)\|_{m,\infty}& \le
               \frac{1}{1-q}\|G_{[m]}(E_J G(\cdot,m))-G_{[m]}(0)\|_{m,\infty}\\  
          & \le \frac{q}{1-q} \|G(\cdot,m)\|_{m,\infty}\le \frac{qK}{1-q}, \\
          \|\tilde{G}(\cdot,m)\|_{m,\infty} & \le \frac{qK}{1-q} + K =
                                              \frac{K}{1-q}=
                                              \tilde{K}. 
        \end{align*}
        This yields \eqref{Qest} provided we have shown that
        $\tilde{G}(n,m)$ has the representation \eqref{appb:Green}
        with projectors 
        $\tilde{P}_m^+\coloneqq  \tilde{G}(m,m)$, $\tilde{P}_m^-=I-
        \tilde{P}_m^+$ and solution operator 
        $\tilde{\Phi}$  of the perturbed system
        \eqref{eq3:perturbsyst}. 
        Indeed, this representation holds since $\tilde{G}$ solves
        \eqref{appb:Greenpert}:  
        \begin{align*}
          n\ge m: &\; \tilde{G}(n,m)=\tilde{\Phi}(n,m)\tilde{G}(m,m)=
                    \tilde{\Phi}(n,m)\tilde{P}_m^+, \\ 
          n=m-1: & \;
                   \tilde{P}_m^+=\tilde{\Phi}(m,m-1)\tilde{G}(m-1,m)+
                   I, \quad \tilde{G}(m-1,m)= 
          \tilde{\Phi}(m-1,m)(\tilde{P}_m^+-I),\\
          n<m-1: & \;
                 -\tilde{P}_m^-=\tilde{\Phi}(m,m-1)\tilde{G}(m-1,m)=
                 \tilde{\Phi}(m,m-1)\tilde{\Phi}(m-1,n)
                 \tilde{G}(n,m), \\ 
          & \; \tilde{G}(n,m)= - \tilde{\Phi}(n,m) \tilde P_m^-.
        \end{align*}
        To show that $\tilde{P}_m^+$ is a projector we prove that
        \begin{align*}
          Y(n,m)= \begin{cases} \tilde{\Phi}(n,m)\big(
            \tilde{P}_m^+\big)^2=\tilde{G}(n,m) \tilde{P}_m^+, & n \ge
            m, \\ 
            \tilde{\Phi}(n,m)\big(\big(\tilde{P}_n^+ \big)^2 -I
            \big)=\tilde{G}(n,m)(I+ \tilde{P}_m^+), & n <m 
          \end{cases}
        \end{align*}
        is another solution of  \eqref{appb:Greenpert} in $Z_m$, so
        that 
        $\tilde{P}_m^+=\tilde{G}(m,m)=\tilde{G}(m,m)\tilde{P}_m^+=
        \big(\tilde{P}_m^+)^2$ follows 
        by uniqueness. First, note that $\tilde{G}(\cdot,m) \in Z_m$
        and the bound $\|\tilde{P}_m^+\| \le \tilde{K}$ imply
        $Y(\cdot,m)\in Z_m$. 
        Then  $Y(\cdot,m)$
        solves \eqref{appb:Greenpert} for $n > m-1$ resp. $n<m-1$ as
        follows by multiplication with $\tilde{P}_m^+$
        resp. $I+\tilde{P}_m^+$ from the right. At $n=m-1$ we have 
        \begin{align*}
          Y(n+1,m)& = Y(m,m)=\tilde{\Phi}(m,m-1) \tilde{\Phi}(m-1,m)
          \big(\big(\tilde{P}_m^+\big)^2-I\big) + I \\
          &= \tilde{\Phi}(m,m-1)Y(m-1,m) + I.
        \end{align*}
        The proof is finished by using that  $\tilde{G}$ is a solution
        of  \eqref{appb:Greenpert}: 
        \begin{align*}
          P_{n_-}^+ Y(n_-,m) = \begin{cases}
            P_{n_-}^+\tilde{G}(n_-,m)(I + \tilde{P}_m^+)=0, & m>
            n_-,\\ 
            P_{n_-}^+ \tilde{G}(n_-,n_-)^2=P_{n_-}^+, & m=n_-. 
          \end{cases}
        \end{align*}
        As a last step we verify the invariance condition (i) of
        Definition \ref{edDef}. An induction shows that 
        it suffices to prove invariance for one step, i.e.
        \begin{align} \label{appb:inv}
          \tilde{\Phi}(m+1,m) \tilde{P}_m^+ & = \tilde{P}_{m+1}^+
                                              \tilde{\Phi}(m+1,m),
                                              \quad m \in J. 
        \end{align}
        Similar to the previous argument we consider, for $m \in J$ fixed,
        \begin{align*}
          Y(n,m+1)= \begin{cases}
            \tilde{\Phi}(n,m)\tilde{P}_m^+\tilde{\Phi}(m,m+1)
            =\tilde{G}(n,m)\tilde{\Phi}(m,m+1), & 
            n \ge m+1, \\
            - \tilde{\Phi}(n,m) \tilde{P}^-_m \tilde{\Phi}(m,m+1)=
            \tilde{G}(n,m) \tilde{\Phi}(m,m+1), & n < m+1, 
          \end{cases}
        \end{align*}
        and show that $Y$ solves \eqref{appb:Greenpert} for $m+1$
        instead of $m$. The unique solvability then yields 
        \begin{align*}
        \tilde{\Phi}(m+1,m) \tilde{P}_m^+ \tilde{\Phi}(m,m+1)=
          Y(m+1,m+1)= \tilde{G}(m+1,m+1) = \tilde{P}_{m+1}^+, 
          \end{align*}
        hence the assertion \eqref{appb:inv}. First note that
        $Y(\cdot,m+1)\in Z_{m+1}$ follows from $\tilde{G}(\cdot,m) \in
        Z_m$ and 
        the boundedness of $\tilde{\Phi}(m,m+1)$. Then the equation
        \eqref{appb:Greenpert} holds 
        for $n \neq m $ as can be seen by multiplying with
        $\tilde{\Phi}(m,m+1)$ from 
        the right. At $n=m$ we find
        \begin{align*}
          Y(m+1,m+1)& = \tilde{\Phi}(m+1,m) \tilde{P}_m^+
                      \tilde{\Phi}(m,m+1) \\ 
          & = \tilde{\Phi}(m+1,m)(I-
            \tilde{P}_m^-)\tilde{\Phi}(m,m+1)= I + \tilde{\Phi}(m+1,m)
            Y(m,m+1). 
        \end{align*}
        Finally, we conclude from \eqref{appb:Greenpert} for $n_-<m$
        \begin{align*}
          P_{n_-}^+ Y(n_-,m+1)& =  P_{n_-}^+  \tilde{G}(n_-,m)
                                \tilde{\Phi}(m,m+1) =0 =
                                \delta_{n_-,m+1}P_{n_-}^+ 
        \end{align*}
        and for $n_-=m$
        \begin{align*}
          P_{n_-}^+Y(n_-,n_-+1)& = P_{n_-}^+(\tilde{P}_{n_-}^+ - I)
                                 \tilde{\Phi}(n_-,n_-+1) \\ 
          & =P_{n_-}^+(\tilde{G}(n_-,n_-) - I) \tilde{\Phi}(n_-,n_-+1)=0
          = \delta_{n_-,n_-+1}P_{n_-}^+.
        \end{align*}
        This finishes the proof.  \hfill \proofbox

\end{appendix}


\section*{Acknowledgments}
Both authors are grateful to the Research Centre for Mathematical
Modelling ($\text{RCM}^2$) at Bielefeld University for continuous
support of their joint research. 
The work of WJB was funded by the Deutsche Forschungsgemeinschaft
(DFG, German Research Foundation) – SFB 1283/2 2021 – 317210226. 
TH thanks the Faculty of Mathematics at Bielefeld
University for further support. We thank the referees for useful
hints which improved the presentation.
 

\bibliographystyle{abbrv}

\end{document}